\documentclass[english, numberwithinsect]{lipics-v2021-hacked}
\pdfoutput=1 
\usepackage{fewerfloatpages}
\hideLIPIcs
\graphicspath{{figures}{plots/images}}

\title{Computing a Link Diagram from its Exterior}

\author{Nathan M. Dunfield}
{Dept.~of Math., University of Illinois at Urbana-Champaign, USA,
\url{https://dunfield.info}}
{nathan@dunfield.info}
{https://orcid.org/0000-0002-9152-6598}
{Partially supported by US National Science Foundation grants
  DMS-1510204 and DMS-1811156 and by a Simons Fellowship.}

\author{Malik Obeidin}
{Google, Inc., USA}
{obeidinm@gmail.com}
{}
{Partially supported by US National Science Foundation grants
  DMS-1510204 and DMS-181115.}

\author{Cameron Gates Rudd}
{Max-Planck-Institut f\"ur Mathematik, Bonn, Germany,
\url{https://www.camrudd.page}}
{cameron.rudd@gmail.com}
{https://orcid.org/0000-0001-6065-1110}
{Partially supported by US National Science Foundation grant DMS-1811156.}

\authorrunning{N. Dunfield, M. Obeidin, and C. Rudd}
\Copyright{Nathan M. Dunfield, Malik Obeidin, and Cameron Gates Rudd}

\acknowledgements{We thank Matthias Goerner and Henry Segerman for
  helpful correspondence, and thank each of the referees for their
  detailed comments which helped improve this paper.}

\makeatletter
\relatedversion{}
\relatedversiondetails{Proceedings version}{https://doi.org/10.4230/LIPIcs.SoCG.2022.37}
\makeatother

\supplement{}
\supplementdetails{Code and data}{https://doi.org/10.7910/DVN/BT1M8R}

\definecolor{nmdlight}{cmyk}{0.128, 0.176, 0.0, 0.0}
\definecolor{nmdmedium}{cmyk}{0.41, 0.49, 0.06, 0.0}
\definecolor{nmddark}{rgb}{0.1, 0.0, 0.5}

\definecolor{cameroncomment}{rgb}{0.067, 0.412, 0.067}

\usepackage{fewerfloatpages}
\usepackage{float, placeins}
\usepackage{tikz}
\usepackage{tikz-cd}
\usetikzlibrary{arrows.meta}

\newcommand{\FindDiagram}{\textsc{Find Diagram}}
\newtheorem*{finddiag}{\FindDiagram}
\newtheorem{algothm}[theorem]{Algorithm}
\theoremstyle{remark}
\newtheorem*{inputdata}{Input}

\definecolor{mplone}{HTML}{1f77b4}
\definecolor{mpltwo}{rgb}{0.99, 0.78, 0.07}
\definecolor{mplthree}{HTML}{2ca02c}
\newcommand{\SKCKlegend}[1]{%
  \begin{scope}[shift={#1}]
    \draw[color=mplone!70, fill] (0, 3) circle (0.7)
    node[right=1.3, color=black] {$\SK$};
    \draw[color=mpltwo!70, fill] (0, 0) circle (0.7)
    node[right=1.3, color=black] {$\CK$};
  \end{scope}
}

\newcommand{\hyp}{\nobreakdash-\hspace{0pt}}
\newcommand{\3}[1]{3\hyp}
\let\2\undefined
\newcommand{\2}[1]{2\hyp}
\let\1\undefined
\newcommand{\1}[1]{1\hyp}

\newcommand{\eK}{E(K)}
\newcommand{\zK}{Z(K)}
\newcommand{\eKi}{E(K_i)}

\newcommand{\cT}{{\mathcal T}}
\newcommand{\cS}{{\mathcal S}}
\newcommand{\cK}{{\mathcal K}}
\newcommand{\cC}{\mathcal C}

\newcommand{\Tdot}{\mathring{\cT}}
\newcommand{\Mdot}{\mathring{M}}
\newcommand{\CK}{\cC\cK}
\newcommand{\SK}{\cS\cK}
\newcommand{\assign}{:=}
\newcommand{\PSL}[2]{\mathrm{PSL}_{#1} #2}
\newcommand{\SL}[2]{\mathrm{ SL}_{#1} #2}
\newcommand{\Z}{{\mathbb Z}}
\newcommand{\R}{{\mathbb R}}
\newcommand{\Q}{{\mathbb Q}}
\newcommand{\pairm}[2]{{#1\langle #2 #1\rangle}}

\newcommand{\NP}{\mathbf{NP}}
\newcommand{\coNP}{\mbox{\textbf{co-NP}}}
\renewcommand{\H}{{\mathbb{H}}}

\tikzset{%
  nmdstd/.style={%
    line join=round,
    line cap=round,
    font=\footnotesize,
    >={Computer Modern Rightarrow[length=1pt 5, width'=0pt 1]},
  }
}

\makeatletter
\def\tikzoverlay{%
    \pgfutil@ifnextchar[{\tikzoverlay@opt}{\tikzoveraly@opt[]}%
}
\def\tikzoverlay@opt[#1]#2{%
    \begin{tikzpicture}
        \node[anchor=south west, inner sep=0] (image) at (0,0) {\includegraphics[#1]{#2}};
        \newdimen\nmd@tikzoverlaywidth
        \pgfextractx{\nmd@tikzoverlaywidth}{\pgfpointanchor{image}{south east}}
        \begin{scope}[nmdstd]
          \tikzset{x=0.01\nmd@tikzoverlaywidth}
          \tikzset{y=0.01\nmd@tikzoverlaywidth}
}
\def\endtikzoverlay{%
    \end{scope}
    \end{tikzpicture}
}

\makeatother

\begin{document}

\maketitle

\begin{abstract}
  A knot is a circle piecewise-linearly embedded into the 3-sphere.  The
  topology of a knot is intimately related to that of its exterior,
  which is the complement of an open regular neighborhood of the knot.
  Knots are typically encoded by planar diagrams, whereas their
  exteriors, which are compact 3-manifolds with torus boundary, are
  encoded by triangulations.  Here, we give the first practical
  algorithm for finding a diagram of a knot given a triangulation of
  its exterior.  Our method applies to links as well as knots, and
  allows us to recover links with hundreds of crossings.  We use it to
  find the first diagrams known for 23 principal congruence arithmetic
  link exteriors; the largest has over 2,500 crossings.  Other
  applications include finding pairs of knots with the same 0-surgery,
  which relates to questions about slice knots and the smooth 4D
  Poincar\'e conjecture.
\end{abstract}

\begin{figure}[bth]
  \centering
  \includegraphics[angle=90, width=0.4\textwidth]{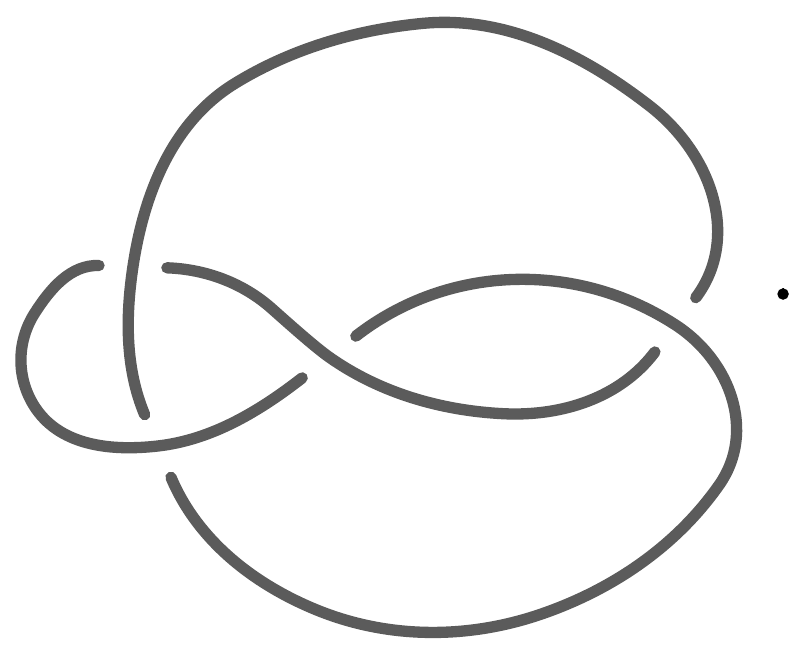}
  \caption{A planar diagram for a knot can be viewed as a 4-valent
    graph (the ``shadow'' of the above figure) with a planar embedding
    where every vertex represents a \emph{crossing}, a place where one
    part of the knot crosses in front of the other in 3D. }
  \label{fig:diagram}
\end{figure}


\section{Introduction}
\label{sec: intro}

A knot is a piecewise-linear (PL) embedding of a circle $S^1$ into the
3-sphere $S^3$.  The study of knots goes back to the 19th century, and
today is a central focus of low-dimensional topology, with
applications to chemistry \cite{Flapan2000}, biology
\cite{FlapanHeWong2019}, engineering \cite{Peddada2021}, and
theoretical computer science \cite{deMesmayRieckSedgewick2021}.  Two
knots are topologically equivalent when they are isotopic, that is,
when one can be continuously deformed to the other without passing
through itself.  Computationally, knots are typically encoded as
planar diagrams (Figure~\ref{fig:diagram}); there are more than 350
million distinct knots with diagrams of at most 19 crossings
as enumerated by \cite{Burton2020}.

The topology of knots is intimately related to that of their
exteriors, where the \emph{exterior} of a knot $K$ is the compact
3-manifold with torus boundary $\eK \assign S^3 \setminus N(K)$, where
$N(K)$ is an open tubular neighborhood of $K$.  Indeed, the
orientation-preserving homeomorphism type of the exterior $\eK$
determines the knot $K$ \cite{GordonLuecke1989}.  Many algorithms for
knots work via their exteriors, starting with Haken's foundational
method for deciding when a knot is equivalent to the unknot
\cite{Haken1961}.  Consequently, the problem of going from a diagram
$D$ of $K$ to a triangulation of $\eK$ is well-studied \cite[\S
7]{HassLagariasPippenger1999}; for ideal triangulations (see
Section~\ref{sec: tri} below), one needs only four tetrahedra per
crossing of $D$ \cite[\S 3]{Weeks2005}.  Here, we study the inverse
problem:

\begin{figure}[p]
  \centering
  \includegraphics[height=0.95\textwidth, angle=90]%
  {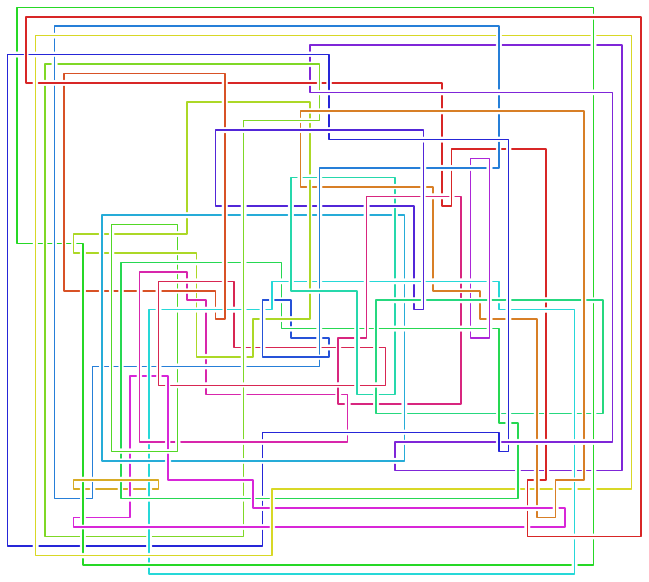}
  \caption{The first known diagram of a link whose exterior is
    $\Mdot = \H^3/\Gamma(I)$ where $\Gamma(I)$ is the principal
    congruence subgroup of $\PSL{2}{\Z[\frac{1 + \sqrt{15}i}{2}]}$ of
    level $I = \pairm{\big}{6, \frac{-3 + \sqrt{15}i}{2}}$ from
    \cite{BakerGoernerReid2019a}; it has 24 components and 294
    crossings. The input ideal triangulation $\Tdot$ for $\Mdot$ had
    249 tetrahedra.  Since the hyperbolic volume of
    $\Mdot \approx 225.98$, any diagram must have at
    least 66 crossings by \cite[Theorem~5.1]{Adams2013}.}
  \label{fig: cong 1}
\end{figure}

\begin{figure}[p]
  \centering
  \includegraphics[width=0.95\textwidth]%
  {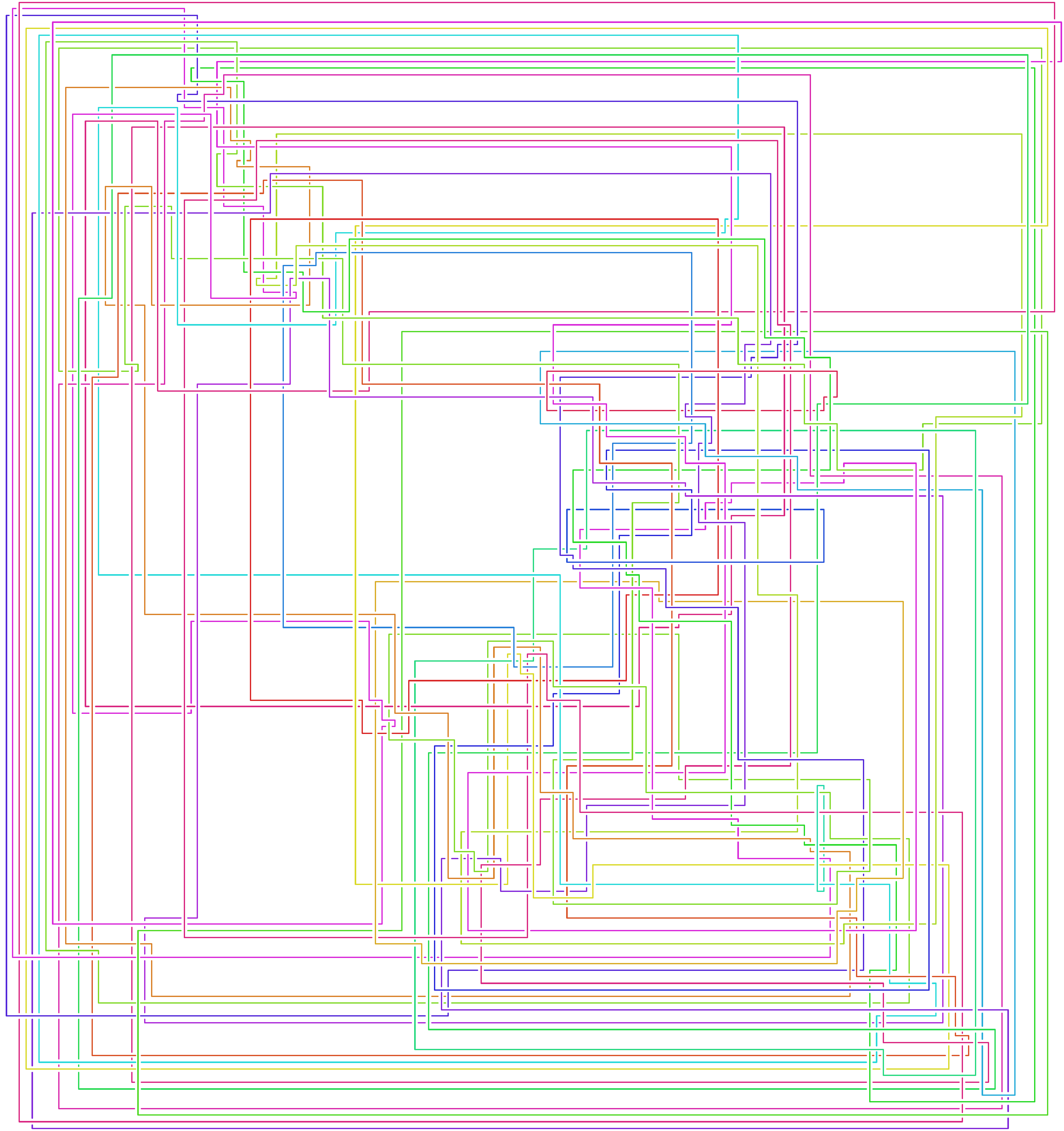}
  \caption{The first known diagram of a link whose exterior is
    $\Mdot = \H^3/\Gamma(I)$ where $\Gamma(I)$ is the principal
    congruence subgroup of $\PSL{2}{\Z[\frac{1 + \sqrt{15}i}{2}]}$ of
    level $I = \pairm{\big}{5, \frac{5 + \sqrt{15}i}{2}}$ from
    \cite{BakerGoernerReid2019a}; it has 24 components and 1,092
    crossings. The input ideal triangulation $\Tdot$ for $\Mdot$ had
    211 tetrahedra.  Since the hyperbolic volume of
    $\Mdot \approx 188.32$, any diagram must have at least 56
    crossings \cite[Theorem~5.1]{Adams2013}.}
  \label{fig: cong 2}
\end{figure}

\begin{finddiag}
Input a triangulation $\cT$ of a knot exterior $\eK$, output a
diagram of $K$.
\end{finddiag}

If the input triangulation $\cT$ is guaranteed to be that of a knot
exterior (in fact, this is decidable by Algorithm S of
\cite{JacoSedgwick2003}), then a useless algorithm to find $D$ is just
this: start generating all knot diagrams, triangulate each exterior,
and then do Pachner moves (see Section~\ref{sec: Pachner}) on these
triangulations.  Since any two triangulations of a compact 3-manifold
are connected by a sequence of such moves, one eventually stumbles
across $\cT$, thus finding a diagram for the underlying knot.  We do
not explore the computational complexity of \FindDiagram\, beyond
showing it is at least exponential space in Theorem~\ref{thm: fib},
but rather give the first algorithm that is highly effective in
practice.  We work more generally with links, where a \emph{link} is a
disjoint union of knots.  While a link exterior does not uniquely
determine a link \cite[Figure 9.28]{Adams1994}, this indeterminacy is
removed by specifying meridional curves for the link, see Section
\ref{sec: dehn}; hence we require such curves as part of the input in
Section~\ref{sec: outline}. Figures~\ref{fig: cong 1} and \ref{fig:
  cong 2} show diagrams that were found by our method; these are the
first known diagrams of these particular link exteriors, see
Section~\ref{sec: cong}.

\subsection{Prior work}
\label{sec: priors}

In general, an efficient algorithmic solution to the homeomorphism
problem has not been implemented, and would be quite complicated; see
\cite{Kuperberg2019}.  However, when the interior of $\eK$ has a
complete hyperbolic structure, in short is \emph{hyperbolic}, the
homeomorphism problem can be quickly solved in practice using
hyperbolic geometry, even for triangulations with 1,000 tetrahedra
\cite{Weeks1993}. This case is in practice generic for prime knots;
for example, 99.999\% of the knots in \cite{Burton2020} are
hyperbolic.  This allows a table lookup method for \FindDiagram\ when
$K$ is small enough; one uses hyperbolic and homological invariants to
form a hash of $\eK$, queries a database of knots to get a handful of
possible $K_i$, and then checks if any $\eKi$ is homeomorphic to
$\eK$. This technique is used by the \texttt{identify} method of
\cite{SnapPy}, but is hopeless for something like Figure~\ref{fig:
  cong 1}, as the number of links of that size exceeds the number of
atoms in the visible universe \cite{SundbergThistlethwaite1998}.

A related approach was used in \cite{ChampanerkarKofmanMullen2014,
  BakerKegel2021} to find knot diagrams for all 1,267 knots where
$\eK$ is hyperbolic and can be triangulated with at most 9 ideal
tetrahedra \cite{Burton2014, Dunfield2020}. While knots with few
crossings have simple exteriors, the converse is not the case, and the
simplest known diagrams for about 25\% of these knots have 100--300
crossings.  However, these knots either fall into very special
families which can be tabulated to a large number of crossings, or one
can drill out additional curves to get a link exterior that
appears in an existing table and has special properties allowing the
recovery of a diagram of the knot itself.

The exteriors of the special class of alternating knots have nice
topological characterizations given in \cite{Greene2017} and
\cite{Howie2017}. Using these characterizations, Howie
\cite{Howie2017} and independently Juh\'asz and Lackenby in an
appendix to \cite{Greene2017} describe normal surface theory
algorithms for determining whether a 3-manifold $\eK$ is the exterior
of an alternating knot. The certificate that $\eK$ is the exterior of
an alternating knot can then be used to produce an alternating knot
diagram of $K$; see \cite[page 2370]{Howie2017}.

There are other ad hoc methods in the literature, see
e.g.~\cite{BakerGoernerReid2019b} and references therein, but this
paper is the first to give a generically applicable method for
\FindDiagram.

\subsection{Outline of the algorithm}
\label{sec: outline}

As Figures~\ref{fig: cong 1} and \ref{fig: cong 2} show, our method
can solve \FindDiagram\ in cases where any diagram for the link has 66
and 55 or more crossings respectively. It also easily handles any
example covered by one of the techniques discussed in
Section~\ref{sec: priors}, and more applications are given in
Sections~\ref{sec: impl} and \ref{sec: apps}. Experimental mean
running time was $O(1.07^n)$, where $n$ is the number of tetrahedra in
the input ideal triangulation, see Figure~\ref{fig: running}.  With
the definitions of Section~\ref{sec: background}, the input for our
algorithm is:

\begin{inputdata}
  \label{algo input}
  \begin{enumerate}[a.]
  \item
    \label{item: ideal input}

    An ideal triangulation $\Tdot$ of a compact \3-manifold $\Mdot$
    with toroidal boundary, with an essential simple closed curve
    $\alpha_i$ for each boundary component of $\Mdot$.

  \item
    \label{item: input moves}

    A sequence $(P_i)$ of Pachner moves transforming the layered
    filling triangulation $\cT$ of the manifold
    $M = \Mdot(\alpha_1, \ldots, \alpha_k)$ into a specific
    \2-tetrahedra \emph{base triangulation} $\cT_0$ of $S^3$.
  \end{enumerate}
\end{inputdata}

One might object that (\ref{item: input moves}) is effectively
cheating, since no polynomial-time algorithm for finding $(P_i)$ is
known, or indeed for deciding if $M$ is $S^3$. Using the estimates in
\cite{Mijatovic2003}, one can perform a naive search to find some
$(P_i)$, but the complexity of this is
super-exponential. However, recognizing $S^3$ by finding such moves is
easy in practice, see Section~\ref{sec: cert}, with the length of
$(P_i)$ linear in the size of $\cT$ as per Figure~\ref{fig: expanded
  moves}.  The output of the algorithm is a knot diagram $D$, encoded
as a planar graph with over/under crossing data for the vertices.

The main data structure is a triangulation $\cT$ of $S^3$ with a PL
link $L$ that is disjoint from the 1-skeleton.  The link $L$ is
encoded as a sequence of line segments, each contained in a single
tetrahedron of $\cT$, with endpoints recorded in barycentric
coordinates.  An initial pair $(\cT, L)$ in (\ref{item: input moves})
is constructed from input (\ref{item: ideal input}) as described in
Section~\ref{sec: layered tri}.  The algorithm proceeds by performing
the Pachner moves $P_i$ from (\ref{item: input moves}), keeping track
of the PL arcs encoding the link $L$ throughout using the techniques
of Section~\ref{sec: mod tri}.  The result is the base triangulation
enriched with PL arcs representing the link $L$. As detailed in
Section~\ref{sec: embed}, this triangulation of $S^3$ can be cut open
along faces and embedded in $\R^3$, giving an embedding of the
cut-open link into $\R^3$ as a collection of PL arcs with endpoints on
the boundary of these tetrahedra. As in Figures~\ref{fig: base tri}
and~\ref{fig: fins and lenses}, these PL arcs are then tied up using
the face identifications to obtain a collection of closed PL curves
that represent $L$. An initial link diagram $D$ is obtained by
projecting this PL link onto a plane and recording crossing
information.  We then apply generic simplification methods to $D$ and
output the result.

This outline turns out to be deceptively simple. Some key difficulties
are:

\begin{enumerate}
\item Understanding what $2\to3$ and $3\to2$ Pachner moves do to the
  link $L$ is fairly straightforward as these correspond to changing
  the triangulation of a convex polyhedron in $\R^3$.  However, while
  these two moves theoretically suffice for (\ref{item: input moves}),
  in practice one wants to use $2\to0$ moves as well, see
  Section~\ref{sec: cert}, and these are much harder to deal with, as
  Figure~\ref{fig: 2 to 0 prob} shows. We thus expand each $2\to0$
  move into a (sometimes quite lengthy) sequence of $2\to3$ and
  $3\to2$ moves as discussed in Section~\ref{sec: 2 to 0}. We give a
  simplified expansion for the trickiest part, the
  endpoint-through-endpoint move, using 6 of the basic $2\to3$ and
  $3\to2$ moves instead of 14.

\item The complexity of the link grows very rapidly as we do Pachner
  moves, resulting in enormously complicated initial diagrams.  We
  greatly reduce this by elementary local simplifications to the link
  after each Pachner move, see Section~\ref{sec: simp arcs}.

\item Prior work on simplifying link diagrams was focused on those
  with 30 or fewer crossings, where random application of Reidemeister
  moves (plus flypes) are extremely effective.  Here, we need to
  simplify diagrams with 10,000 or even 100,000 crossings down to
  something with less than 100, and such methods proved ineffective
  for this.  Instead, we used the more global \emph{strand pickup}
  method of Section~\ref{sec: simp link}.
\end{enumerate}

\section{Background}
\label{sec: background}

\subsection{Triangulations}
\label{sec: tri}

Let $M$ be a compact orientable \3-manifold, possibly with boundary.
A \emph{triangulation} of $M$ is a cell complex $\cT$ made from
finitely many tetrahedra by gluing some of their \2-dimensional faces
in pairs via orientation-reversing affine maps so that the resulting
space is homeomorphic to $M$. These triangulations are not necessarily
simplicial complexes, but rather what are sometimes called semi\hyp
simplicial, pseudo\hyp simplicial, or singular triangulations. Of
particular importance are those with a single vertex, the
\emph{1-vertex triangulations}.  For any triangulation, we use $\cT^i$
to denote the $i$-skeleton of $\cT$, that is, the union of cells of
dimension at most $i$.

When $M$ has nonempty boundary, an \emph{ideal
  triangulation} of $M$ is a cell complex $\cT$ made out of finitely
many tetrahedra by gluing \emph{all} of their \2-dimensional faces in
pairs as above so that $M \setminus \partial M$ is homeomorphic to
$\cT \setminus \cT^0$. Put another way, the manifold $M$ is what you
get by gluing together \emph{truncated} tetrahedra in the
corresponding pattern. See \cite{Tillmann2008} for background on
ideal triangulations, which we use only for \3-manifolds whose
boundary is a union of tori.  We always include the modifier ``ideal'',
so throughout ``triangulation'' means a non-ideal, also called
``finite'', triangulation.

\subsection{Triangulations with PL curves}
\label{sec: tris with curves}

Consider a tetrahedron $\Delta$ in $\R^n$ as the convex hull of its
vertices $v_0, v_1, v_2,$ and $v_3$.  We encode points in $\Delta$
using \emph{barycentric coordinates}, that is, write $p \in \Delta$ as
the unique convex combination $\sum_i x_i v_i$ and then represent $p$
by the vector $(x_0,x_1,x_2,x_3)$, where of necessity
$\sum_i x_i = 1$. For a 3-manifold triangulation $\cT$, we view each
tetrahedron $\tau$ as having a fixed identification with the
tetrahedron in $\R^4$ whose vertices are the standard basis vectors;
we use this to encode points in $\tau$ by barycentric coordinates.

An oriented PL curve in $\cT$ will be described by a sequence of such
barycentric coordinates as follows.  A \emph{barycentric arc} $a$ is
an ordered pair of points $(u,v)$ in a tetrahedron $\tau$,
representing the straight segment joining them.  We write
$a.{\tt{start}} = u$ and $a.{\tt{end}} = v$. A \emph{barycentric
  curve} $C$ is a sequence of barycentric arcs $a_i$ such that
$a_i.{\tt{end}}$ and $a_{i+1}.\tt{start}$ correspond to the same point
in $M$ under the face identifications of $\cT$.  For a barycentric
curve, we define $a_i.{\tt{next}} = a_{i+1}$ and
$a_{i+1}.{\tt past }= a_i$; these may not lie in the same
tetrahedron. Suppose the barycentric curve $C$ consists of $N$
barycentric arcs.  If $a_0.{\tt{start}}$ and $a_N.\tt{end}$
correspond to the same point in $M$, we have a \emph{barycentric
  loop}. An embedded barycentric loop is a \emph{barycentric knot}. A
\emph{barycentric link} is a finite disjoint union of such knots.

We always require that a barycentric curve $C$ is in the following
kind of general position with respect to $\cT$.  First, $C$ is
disjoint from  $\cT^1$.  Second, any intersection of a
constituent barycentric arc $a$ with $\cT^2$ is an endpoint of
$a$.  Finally, arcs do not bounce off faces of $\cT^2$, so if an arc
ends in a face, the next arc must be in the adjacent tetrahedron on
the other side of that face.  Throughout, we use only points whose
barycentric coordinates are in $\Q$.

\subsection{Dehn filling}
\label{sec: dehn}

Suppose $\Mdot$ is a compact \3-manifold whose boundary is a union of
tori.  A simple closed curve on a surface is \emph{essential} if it
does not bound a disk.  Given an essential simple closed curve
$\alpha_i$ on each boundary component $T_i$, the \emph{Dehn filling}
of $\Mdot$ along $\alpha = (\alpha_1,\ldots,\alpha_k)$ is the closed
3-manifold $\Mdot(\alpha)$ obtained from $\Mdot$ by gluing a solid
torus $D^2 \times S^1$ to each $T_i$ so that
$\partial D^2 \times \{\mathrm{point}\}$ is $\alpha_i$.  When $\Mdot$
is the exterior of a link $L$ in $S^3$ and each $\alpha_i$ is a small
meridional loop about the $i$-th component of $L$, then
$\Mdot(\alpha)$ is just $S^3$. Given an ideal triangulation $\Tdot$ of
$\Mdot$ and Dehn filling curves $\alpha$, we follow
\cite{WeeksCloseCusp, JacoRubinstein2014, JacoSedgwick2003} to create
a 1-vertex triangulation $\cT$ of $\Mdot(\alpha)$ that we call the
\emph{layered filling triangulation}; see Section~\ref{sec: layered
  tri}. A key point is that the link $L$ consisting of the cores of
the $k$ added solid tori is a barycentric link in $\cT$ made of just
$k$ barycentric arcs.

\subsection{Pachner moves}
\label{sec: Pachner}

\begin{figure}[!tbp]
  \centering
  \begin{tikzpicture}[nmdstd]
    \node[above right] at (0, 4.5)
      {\includegraphics[width=6.25cm]{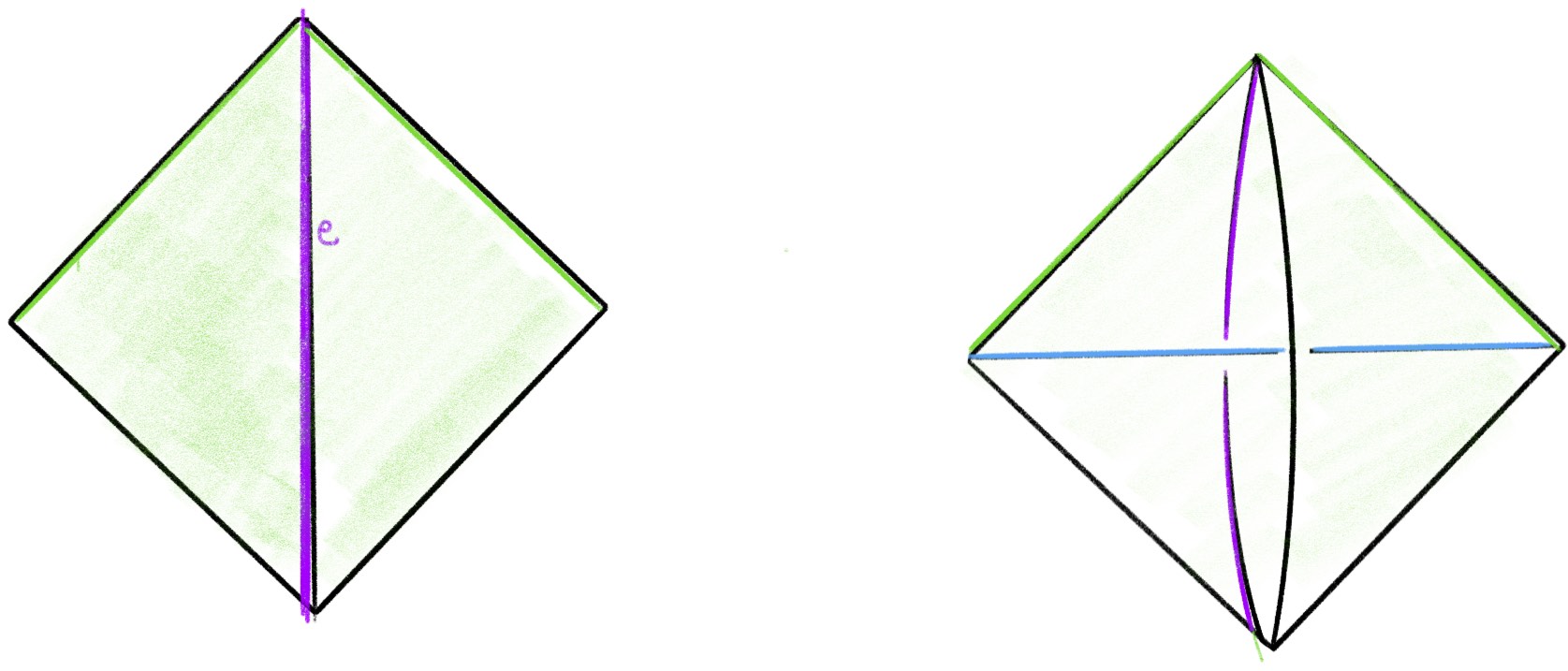}};

    \node at (3.25, 5) {$0 \to 2$};

    \node[above right] at (7, 4.3)
      {\includegraphics[width=5.5cm]{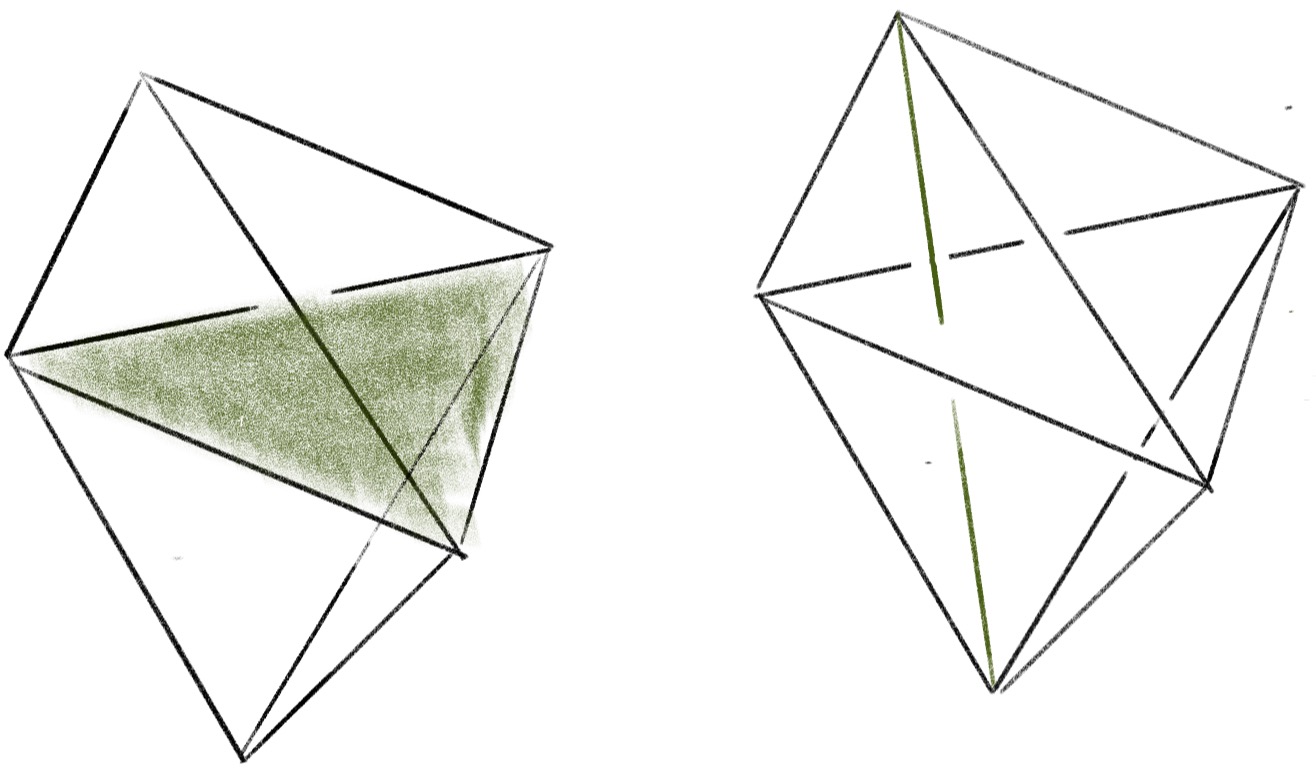}};

    \node at (10.0, 5) {$2 \to 3$};

    \node[above] at (7, 0.25)
      {\includegraphics[width=6cm]{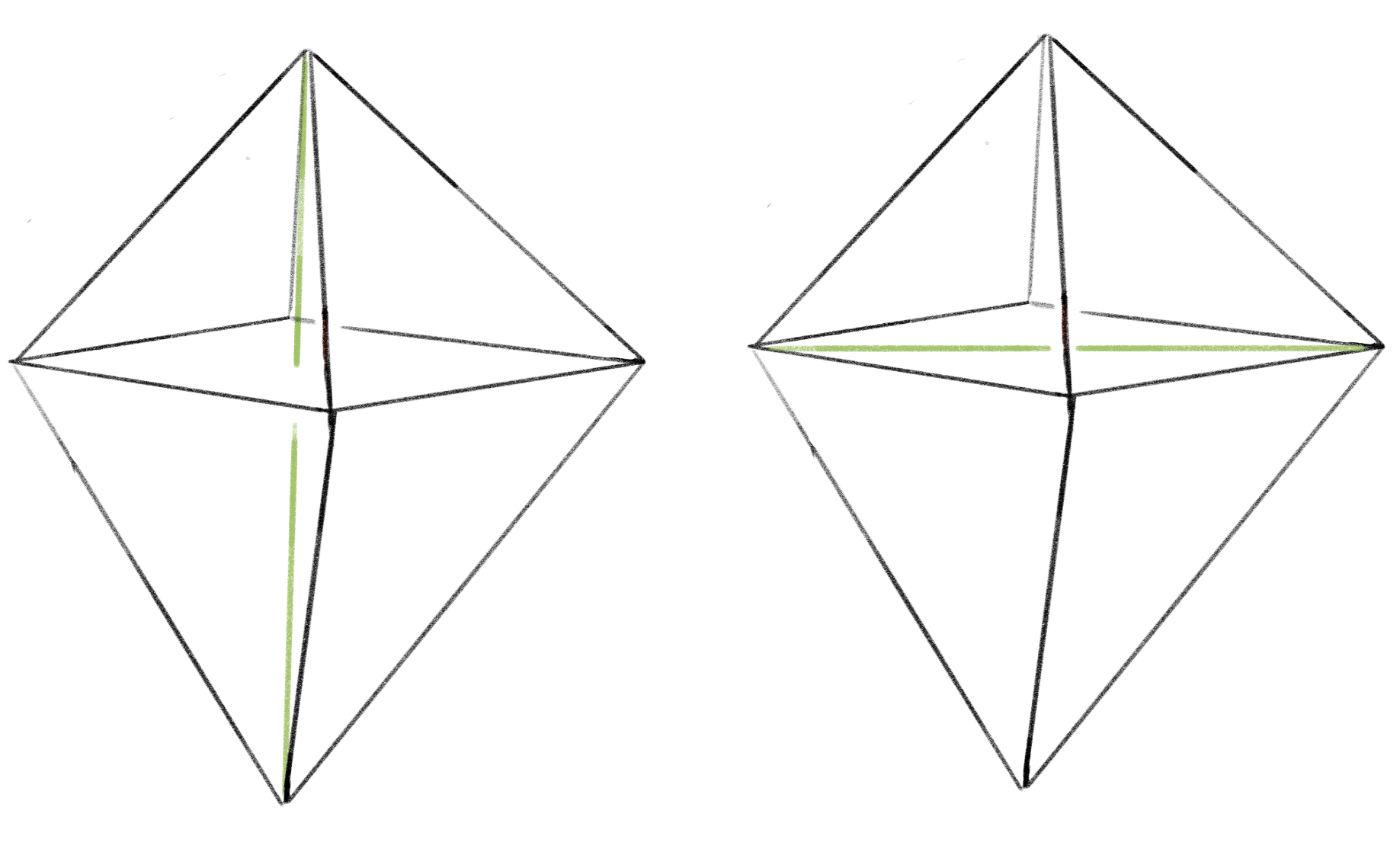}};

    \node[above] at (7, 0.75) {$4 \to 4$};
  \end{tikzpicture}

  \caption{Pachner moves that preserve the number of vertices.}
    \label{fig: pachner moves}
\end{figure}
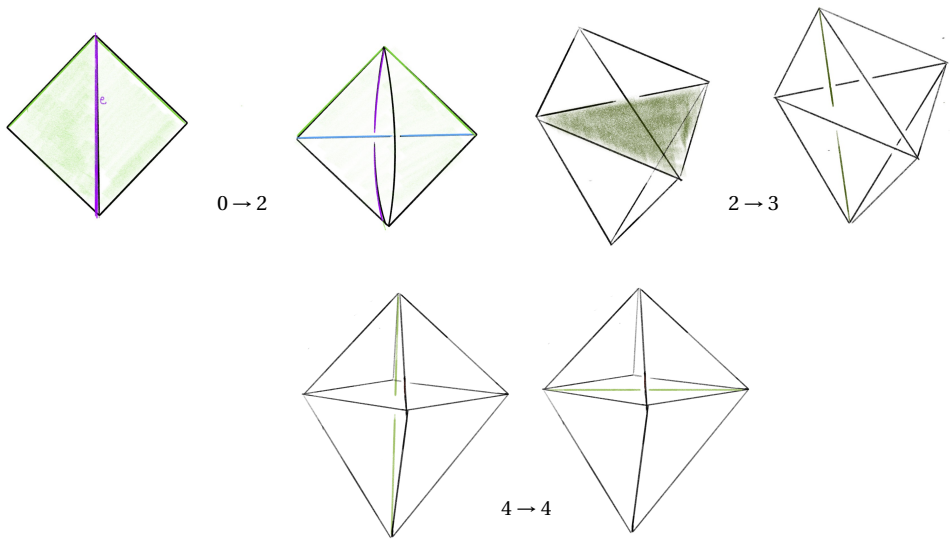

A 3-manifold triangulation $\cT$ can be modified by local
\emph{Pachner moves}, also known as \emph{bistellar flips} to give a new
triangulation of the same underlying manifold.  Those we
use are shown in Figure~\ref{fig: pachner moves} and are as follows:

\begin{enumerate}

\item The $2\to3$ move and its inverse $3\to2$ move.  These take a
  triangulation of a ball, possibly with boundary faces glued
  together, and retriangulate the interior without changing the
  boundary triangulation.  Specifically, the $2\to3$ move takes a pair
  of distinct tetrahedra sharing a face and replaces them with three
  new tetrahedra around a new central edge. The $3\to2$ move reverses
  this, replacing three distinct tetrahedra around a valence-3 edge
  with two tetrahedra sharing a face.

\item The $4\to4$ move. The $4\to4$ move takes four tetrahedra around
  a central edge and replaces them with four new tetrahedra assembled
  around a new valence-4 edge.

\item The $2\to0$ move and its inverse $0\to2$ move. The $2\to0$ move
  takes a pair of tetrahedra sharing two faces to form a valence-2
  edge and collapses them onto their common faces. The $0\to 2$ move
  reverses this by puffing air into a pair of faces sharing an edge
  and adding two new tetrahedra. We call the complex created by the
  $0 \to 2$ move a \emph{pillow}.  The $0\to2$ move inflates a pillow
  and the $2\to0$ move collapses a pillow.
\end{enumerate}

If $\cS$ and $\cT$ are two \1-vertex triangulations of the same closed
3-manifold $M$, then there is a sequence of Pachner moves that
transforms $\cS$ into $\cT$, provided both $\cS$ and $\cT$ have at
least two tetrahedra. To do this, one need only use $2\to3$ and
$3\to2$ moves by \cite[Theorem 1.2.5]{Matveev2007} (see also
\cite{Pachner1991, Piergallini1988}). As noted in the introduction,
the $2 \to 0$ and $0 \to 2$ moves are much harder to deal with than
the others.  We will call the $2 \to 3$, $3 \to 2$ and $4 \to 4$ moves
the \emph{simple Pachner moves}, and note that one needs only these
moves to connect two triangulations as above. However, as discussed in
Remark~\ref{rem: atomic}, the $2 \to 0$ and $0 \to 2$ moves are
extremely useful in practice.  When $M$ is $S^3$, any triangulation
$\cT$ with $n$ tetrahedra is related to a standard triangulation by at
most $12\cdot 10^6n^22^{2\cdot 10^3n^2}$ Pachner moves
\cite{Mijatovic2003}. Experimentally, one needs many fewer moves
\cite{Burton2011}.  In our data shown in Figure~\ref{fig: expanded
  moves}, the number is $O(n)$; this is essential for the utility of
our algorithm for \FindDiagram.


\section{Building the initial triangulation}
\label{sec: layered tri}

In this section, we detail the construction of the layered filling
triangulation $\cT$, mentioned in Section~\ref{sec: dehn}, from part
(\ref{item: ideal input}) of the input: an ideal triangulation $\Tdot$
and Dehn filling slopes $\alpha$.  This procedure is nearly identical
to the approach used in the SnapPy kernel \cite{WeeksCloseCusp} for
constructing triangulations of Dehn fillings, with a slight tweak at
the very last step to end up with a triangulation in the style of
\cite{JacoSedgwick2003} containing layered triangulations of the Dehn
filling solid tori.

\begin{figure}[b]
  \centering
  \includegraphics[width=0.35\textwidth]{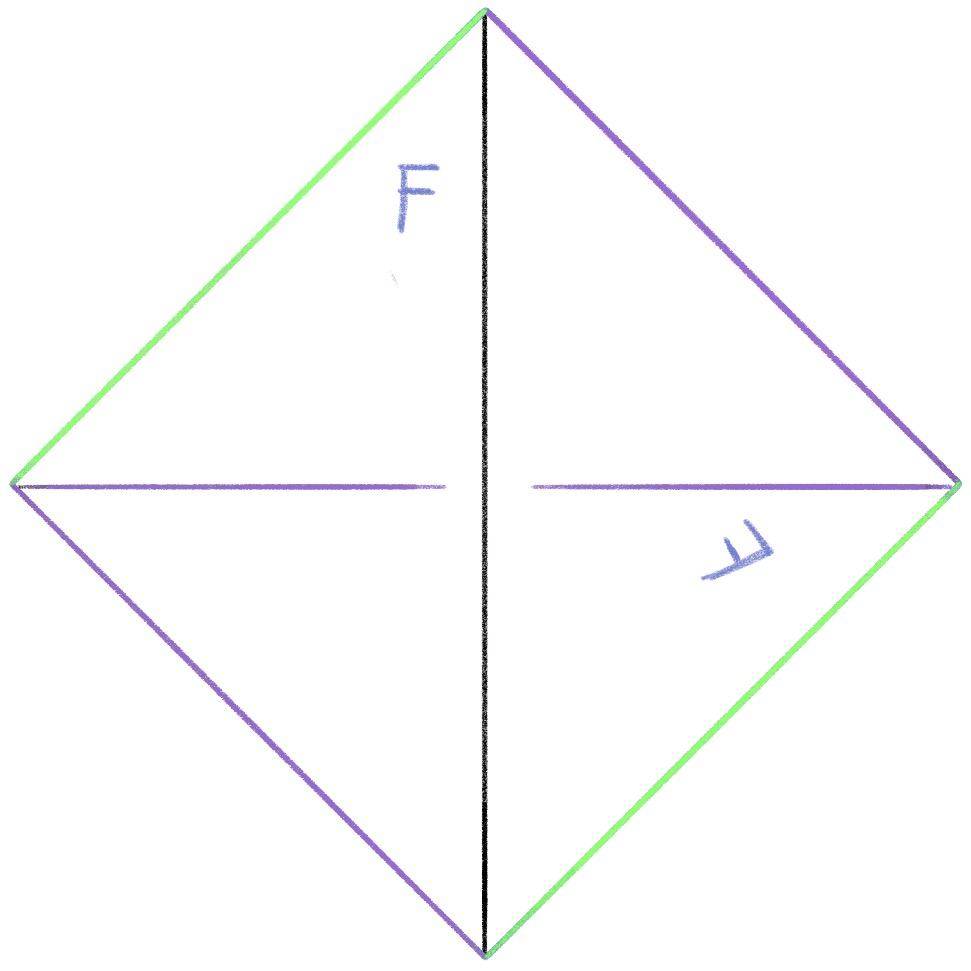}
  \caption{A 1-tetrahedron triangulation of a solid torus. Here, the
    back two faces are identified with the twist indicated by the
    letter ``F''; the edge colors indicate the equivalence classes in
    the glued-up triangulation \cite[Appendix A]{HuszarSpreer2019}.}
  \label{fig: solid}
\end{figure}

Given a single tetrahedron, the face identification indicated in
Figure~\ref{fig: solid} produces a solid torus. Any \1-tetrahedron
triangulation of a solid torus is combinatorially equivalent to this
one. The triangulation of the boundary torus induced by the
\1-tetrahedron triangulation of the solid torus is the standard
\1-vertex triangulation of a torus. For any edge $e$ of the boundary
torus, there is a move modifying the triangulation that, after cutting
open the torus so that the edge $e$ is the diagonal set inside a
square, flips the diagonal. This is commonly called an \emph{edge
  flip} move. Any such flip move can be realized by attaching a
tetrahedron as in Figure~\ref{fig: flip}; this produces a
triangulation of a solid torus with an additional tetrahedron and with
boundary triangulated according to the flip move. A layered solid
torus with $t$ layers is a triangulation of a solid torus that is
obtained from a $(t-1)$-layer layered solid torus by attaching a new
tetrahedron realizing some bistellar flip of the boundary torus. A
0-layer layered solid torus is the \1-tetrahedron solid torus. This
0-layer solid torus contained in the layered solid torus is called the
\emph{core solid torus}.  While every layered solid torus has boundary
given by the standard one vertex triangulation of the torus, the
isotopy class, or \emph{slope}, of the boundary of a meridian disk
changes as layers are added.

\begin{figure}
  \centering
  \begin{tikzoverlay}[width=0.3\textwidth]{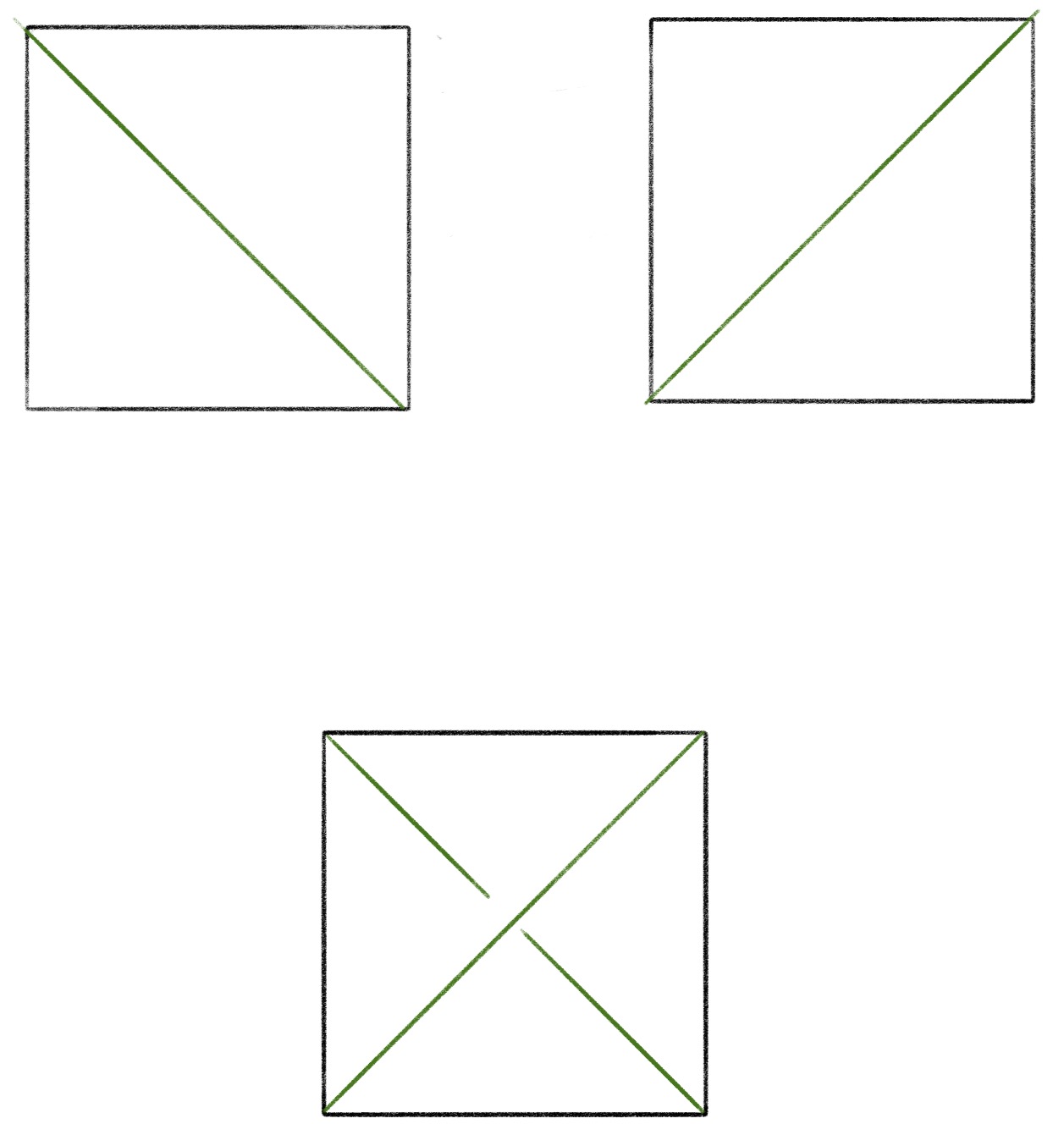}
    \draw[->, color=black!60, line width=0.65pt] (45.5,85.0) -- +(10,0);
  \end{tikzoverlay}
  \caption{An edge flip (top) and corresponding tetrahedron
    (bottom).}
  \label{fig: flip}
\end{figure}

Let $T$ be the standard triangulation of a torus and $\alpha$ a slope
on $T$.  There is an algorithm for producing a layered solid torus so
that filling $\alpha$ bounds a meridian disk, see
\cite[Theorem~4.1]{JacoSedgwick2003}.  One can then attach this
layered solid torus to a triangulated manifold $\Mdot$ whose torus
boundary is triangulated in the standard \1-vertex way to obtain a
triangulation of the Dehn filling $\Mdot(\alpha)$.  We build the
layered filling triangulation from $\Tdot$ and $\alpha$ as
follows. For notational simplicity, we assume $\Mdot$ has only one
boundary component.

\begin{algothm}
  \label{alg: weeks}

  $\tt layered\_filling\_triangulation(\Tdot,\alpha)$
  \begin{enumerate}

  \item Truncate the ideal tetrahedra of $\Tdot$ to obtain a cell
    complex homeomorphic to the compact manifold with torus boundary
    $\Mdot$.

  \item Subdivide this cell complex by placing a vertex at the center
    of each hexagonal face to divide it into 6 triangles, and
    then coning to the middle of every 3-cell. This produces a
    triangulation of $\Mdot$.

  \item Simplify the triangulation of the boundary using the procedure
    in \cite{WeeksCloseCusp}, which largely consists of folding two
    adjacent triangles across their common edge, until the boundary tori
    are triangulated in the standard way.

  \item Add layers to the boundary until the slope $\alpha$ is
    standard, that is, corresponds to the meridian curve of the
    0-layer solid torus.

  \item \label{item: add last}
    Attach the 0-layer solid torus.

  \item Collapse edges joining distinct vertices to obtain a
    \1-vertex triangulation of $M(\alpha)$.
  \end{enumerate}
\end{algothm}

\noindent
Note Algorithm~\ref{alg: weeks} is identical to \cite{WeeksCloseCusp}
except for Step~\ref{item: add last}, where instead one adds a single
tetrahedron with two faces folded together to form a valence-1 edge;
with vertices identified, this single tetrahedron is a solid torus
with a meridian disk collapsed to a point.

Because we constructed the layered filling triangulation $\cT$ from an
ideal triangulation of a manifold with toroidal boundary, we know
which layered solid tori come from Dehn filling. For each such layered
solid torus, its core curve can be represented by the line segment
running between the barycenters of the faces in the core solid torus
that are glued together. In particular, there are natural barycentric
arcs that represent the link $L$ consisting of the core curves of all
the Dehn fillings.  We add these arcs to produce the initial
triangulation $\cT$ of $M(\alpha)$ with its associated barycentric
link.


\section{Finding certificates}
\label{sec: cert}

Part (\ref{item: input moves}) of the input to our algorithm is a
certificate that the Dehn filling $M = \Mdot(\alpha)$ is $S^3$ in the
form of Pachner moves simplifying a triangulation $\cT$ of $M$ to the
base triangulation $\cT_0$ of $S^3$. In practice, one starts with an
ideal triangulation $\Tdot$ and Dehn filling slopes $\alpha$ where it
is unknown if $M(\alpha)$ is $S^3$.  We therefore need a way of
finding this sequence of Pachner moves when it exists.  While deciding
if a closed 3-manifold $M$ is $S^3$ is in $\NP$ by \cite{Ivanov2008,
  Schleimer2011} and additionally in $\coNP$ assuming the Generalized
Riemann Hypothesis \cite[Theorem~11.2]{Zentner2018}, no
sub-exponential time algorithm is known.  The current algorithm that
is best in practice for $S^3$ recognition is to first heuristically
simplify the input triangulation using Pachner moves and then apply
the theory of almost normal surfaces; see Algorithm~3.2 of
\cite{Burton2013}.  However, triangulations of $S^3$ that are truly
hard to simplify using Pachner moves have not been encountered in
practice, and it is open whether they exist at all \cite{Burton2011}.
Thus, when $M$ is $S^3$, the initial stage of Algorithm~3.2 of
\cite{Burton2013} nearly always arrives at a 1-tetrahedron
triangulation of $S^3$ and no normal surface theory is needed.  The
usefulness of our algorithm for \FindDiagram\ relies on the fact that
a heuristic search using Pachner moves gives a practical recognition
algorithm for $S^3$.

\begin{remark}
  \label{rem: atomic}
  The effectiveness of our heuristic search procedure relies on the
  $2\to0$ move being atomic.  Initially, we tried restricting our
  heuristic search to just the \emph{simple} Pachner moves (recall
  these are $2 \to 3$, $3 \to 2$, and $4 \to 4)$, but were typically
  unable to find a sequence that simplified the input triangulation of
  $S^3$ down to one with just a few tetrahedra.  (To square this with
  \cite{Burton2011}, note from Figure~\ref{fig: expanded moves} that
  our triangulations are much larger.)  As is clear from
  Section~\ref{app: 2 to 0}, factoring the $2\to0$ move as a
  sequence of $2\to3$ and $3\to2$ moves is complicated enough that one
  cannot expect to stumble upon these sequences when the triangulation
  is large and the search is restricted to simple Pachner moves.
\end{remark}

Our simplification heuristic closely follows that of SnapPy
\cite{SnapPy}, with some modifications that reduce the complexity of
the final barycentric link in $\cT_0$.  These include:\begin{enumerate}
\item Simplifying the layered filling triangulation $\cT$ of
  Section~\ref{sec: dehn} as much as possible without modifying the few
  tetrahedra containing the initial link.
\item Finding sequences of Pachner moves to $\cT_0$ for several
  different layered filling triangulations, and then using the one
  requiring the fewest moves for the computations in
  Sections~\ref{sec: mod tri}--\ref{sec: simp link}.
\item Ensuring the tail of the sequence of moves is a geodesic
  in the Pachner graph of \cite{Burton2011}.
\end{enumerate}
The details are in Section~\ref{app: cert}.


\subsection{Basic triangulation simplification}
\label{app: cert}

In this section, we detail how
the initial triangulation $\cT$ and Pachner moves transforming it into
$\cT_0$ are constructed from the input pair $(\Tdot, \alpha)$.  Our
overall goal is to minimize the number of Pachner moves and,
especially, minimize the number that involve any arcs.

The SnapPy kernel \cite{SnapPy} provides two routines for trying to
simplify a triangulation.  The first is \texttt{simplify}, which
greedily does various moves that immediately reduce the number of
tetrahedra, as well as random $4 \to 4$ moves in hopes of setting up
such a reduction; it is very similar to Algorithm 2.5 of
\cite{Burton2013} which is \texttt{intelligentSimplify} in Regina
\cite{Regina}.  The second is \texttt{randomize}, which first does
$4 t$ random $2 \to 3$ moves, where $t$ is the number of tetrahedra,
and then calls \texttt{simplify}; it is key for escaping local minima
in the set of triangulations.  In practice, one sometimes needs
\texttt{randomize} in order to reduce a layered filling triangulation
$\cT$ to $\cT_0$.  Because \texttt{randomize} increases the number of
tetrahedra drastically, however temporarily, we work hard to avoid
applying it when there are arcs present.  We modified
\texttt{simplify} and \texttt{randomize} so that one can specify a
subcomplex of the triangulation that is to remain unchanged. Our basic
strategy is:
\begin{enumerate}
\item
  \label{item: go go go}

  Construct the layered filling triangulation $\cT$ from $(\Tdot,
  \alpha)$.

\item
  \label{item: side step}

  Apply \texttt{simplify} and \texttt{randomize} extensively to $\cT$
  with the proviso that each tetrahedron that is the core of a
  filling layered solid torus is not modified.  Call the new
  triangulation $\cT'$.  It contains a barycentric link $L'$
  consisting of the cores of the layered solid tori, which is isotopic
  to the original $L$ in $\cT$.

\item If \texttt{simplify} reduces $\cT'$ to $\cT_0$, record the
  sequence of Pachner moves and consider $\big(\cT', L, (P_i)\big)$ a
  candidate input for the core algorithm.  Otherwise, throw it away.

\item Go back to \ref{item: go go go} until we have several candidates
  for $\big(\cT', L, (P_i)\big)$ or we get tired.  If no candidate is
  found, raise an error; otherwise output the
  one where $(P_i)$ is shortest.
\end{enumerate}
Despite needing \texttt{randomize} to simplify some triangulations of
$S^3$ to $\cT_0$, so far the above has always succeeded.

Finally, it turns out the last few Pachner moves are the most
expensive, since the link is usually quite complicated at that
point. Therefore, we built a look-up table of all triangulations of
$S^3$ with at most five tetrahedra, along with geodesic Pachner move
sequences reducing these triangulations to $\cT_0$. If, when searching
for a sequence of Pachner moves, we reduce the initial triangulation
to one with fewer than five tetrahedra, we can look up whether we
indeed have $S^3$, and we then append to the certificate the geodesic
Pachner move sequence reducing to $\cT_0$.  While this only shortens
the sequence by a few moves, it gave us a major speedup.  For further
details, see the file \texttt{simplify\_to\_base\_tri.py} in
\cite{CodeAndData}.  One referee points out that we might be able to do
even better with a larger look-up table.  Currently, it contains 1,448
triangulations, and adding six or seven tetrahedra would increase this
by 13,660 and 169,077 triangulations respectively
\cite[Section~3.1]{Burton2011}.  However, since the simplification
heuristic of Section~\ref{sec: cert} greedily applies $3 \to 2$ moves
whenever they are available, one only needs store geodesics for those
triangulations where either there are no $3 \to 2$ moves or there is a
$3 \to 2$ move that is not the start of a geodesic to $\cT_0$.  This
idea seems worth exploring, but we leave it for others to pursue.


\section{Modifying triangulations with arcs}
\label{sec: mod tri}

Part (\ref{item: ideal input}) of the input data produced
the layered filled triangulation $\cT$ of Section~\ref{sec: layered tri},
which comes enriched with a barycentric link $L$. Part (\ref{item:
  input moves}) of the input data is a sequence of Pachner moves
$(P_i)$ converting $\cT$ to the base triangulation $\cT_0$ described in
Section~\ref{sec: embed}. The next step of our algorithm is to apply
the moves $(P_i)$ to $\cT$, carrying the link $L$ along as we go.


\subsection{Pachner moves with arcs}

\label{sec: pachner arcs}

In this section, we describe how we keep track of the barycentric
arcs as Pachner moves are performed. Throughout this section, $\cT$ is a
triangulation with barycentric arcs and $P$ is a Pachner move
transforming $\cT$ into the triangulation $\cS = P\cT.$

Recall that the \emph{simple Pachner moves} are $2\to3$, $3\to2$,
and $4\to4$.  Each simple Pachner move $P$ takes a triangulated ball
$B$ in $\cT$, possibly with boundary faces glued together, and
re-triangulates $B$ without changing the triangulation of $\partial B$
to obtain $\cS$. The arcs of the link $L$ contained in the ball are
initially encoded using the barycentric coordinates of $\cT$, and we
need to re-express these arcs in the new barycentric coordinate system
of $\cS$.  We model each simple Pachner move as a pair of
triangulations of concrete bipyramids in $\R^3$, as shown in
Figure~\ref{fig: bipyramids}. We identify the tetrahedra in $\cT$ and
$\cS$ involved in $P$ with tetrahedra in the corresponding bipyramid
in $\R^3$. This identification allows us to map barycentric arcs from
$\cT$ into $\R^3$, and then to map these arcs in $\R^3$ into
$\cS$. The remainder of Section \ref{sec: mod tri} details how this
is used to give a method ${\tt with\_arcs}[P]$ that applies a simple
Pachner move $P$ to $\cT$ while transferring the barycentric arcs from
$\cT$ to $\cS$. This approach cannot work for the $2\to0$ move, as
demonstrated by Figure~\ref{fig: 2 to 0 prob}.  To implement
${\tt with\_arcs[2\to0]}$, we factor the $2\to0$ move into a sequence
of $2\to3$ and $3\to 2$ moves as described in Section~\ref{sec: 2 to
  0}.

\begin{figure}
\centering
{\includegraphics[width=0.55\textwidth]{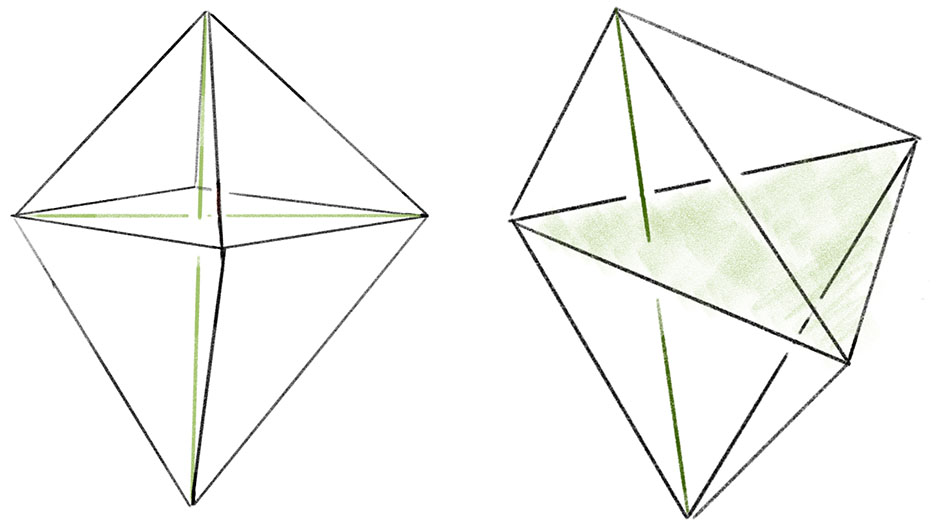}}
\caption{Two bipyramids with superimposed triangulations
  corresponding to before and after applying the $4\to4$ move and
  $2\to3$ or $3\to2$ moves.}
    \label{fig: bipyramids}
\end{figure}

\begin{figure}[b]
\centering
{\includegraphics[width=0.85\textwidth]{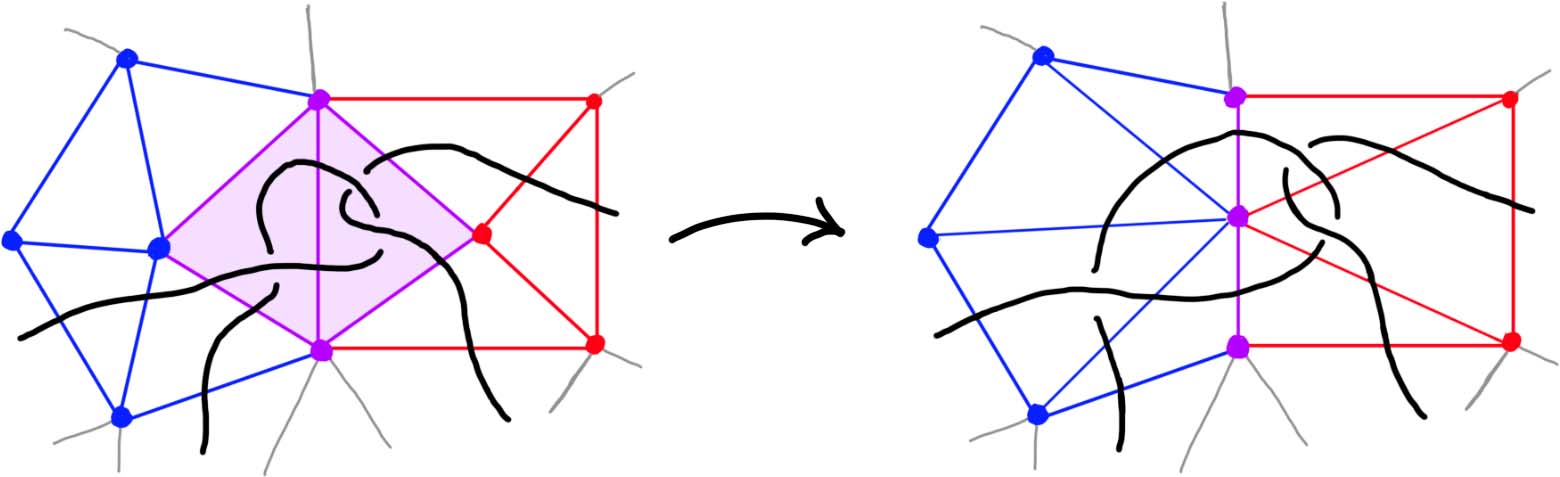}}
\caption{Cartoon showing the difficulty of doing a $2 \to 0$ move with
  arcs present.  At left, the two tetrahedra in the pillow to be
  collapsed are shaded.  Here, you should regard the vertical purple
  arc as the valence-2 edge, with the blue and red dots opposite being
  cross-sections of the two edges of the pillow that become identified
  in the collapse. The problem is that we have to push all the
  topology of the \emph{link} out of the pillow before we collapse it,
  requiring us to move arcs into many of the tetrahedra adjacent to
  the pillow.}
    \label{fig: 2 to 0 prob}
\end{figure}

\subsection{Weak barycentric arcs}
\enlargethispage{0.4cm}

It will be useful to have a mild generalization of barycentric arcs
that allows for negative barycentric coordinates. Identify $\R^3$ with
the hyperplane $\sum_i x_i = 1$ in $\R^4$. The barycentric coordinates
$(x_1, x_2, x_3, x_4)$ associated to the vertices of the standard
simplex in $\R^4$ extend to all of $\R^3$. A \emph{weak barycentric
  coordinate} is a tuple $(x_1, x_2, x_3, x_4)$ describing a point in
this hyperplane. Identifying a tetrahedron $\tau$ in $\cT$ with the
standard 3-simplex in $\R^4$, a \emph{weak barycentric point} in
$\tau$ is a weak barycentric coordinate defining a point in the
hyperplane $\sum_i x_i = 1$. A \textit{weak barycentric arc} $a$ is a
pair of weak barycentric points $(u,v)$ associated to a tetrahedron
$\tau$.

Note that a weak barycentric arc does not generally define a geometric
object in the triangulation $\cT$. However, a weak barycentric arc may
contain a sub-arc that is a genuine barycentric arc. From our
identification of $\tau$ with the standard 3-simplex in $\R^4$, the
intersection of a weak barycentric arc with the standard simplex is a
possibly empty barycentric arc. As ultimately we only want to work
with barycentric arcs, we use a trimming procedure that takes a weak
barycentric arc $a$ in $\tau$ and replaces it with the maximal
barycentric arc it contains.

\subsection{Putting the problem into $\R^3$}

We model the simple Pachner moves as a pair of triangulations of
concrete bipyramids in $\R^3$. We can identify the tetrahedra in
$\cT$ and $\cS$ involved in the move with tetrahedra in the
corresponding bipyramid. This identification allows us to map
barycentric arcs from $\cT$ into $\R^3$ and then to map these arcs in
$\R^3$ back to $\cS$. Combining these processes transfers arcs from
$\cT$ to $\cS$. We explain this in detail for the $2\to3$ move.

Recall that the $2\to3$ move takes two tetrahedra glued along a face
$F$, deletes the face $F$, adds a vertical edge, and re-triangulates
so that there are three tetrahedra assembled around the new vertical
edge; see Figure \ref{fig: pachner moves}.  The bipyramid is built
from a pair of tetrahedra $\Delta_N$ and $\Delta_S$ in $\R^3$ sharing
a single face. Let $A = (3,0,0) , B = (0,0,3),$ and $C = (0,3,0)$ be
the corners of the common face and let $N = (0,0,0)$ and $S = (2,2,2)$
be corners of the bipyramid lying over and under the common face.
These particular points are chosen as they are symmetric with respect
to affine maps. There is also a triangulation of this bipyramid
consisting of three tetrahedra assembled around a central edge running
from $N$ to $S$. We label these tetrahedra $\Delta_A$, $\Delta_B$,
$\Delta_C$, where $\Delta_A$ is the tetrahedron excluding $A$, and
likewise for $B$ and $C$. The $2\to3$ move replaces the \2-tetrahedron
triangulation with this \3-tetrahedron triangulation.

Let $\tau_N$ and $\tau_S$ be tetrahedra in $\cT$ sharing a face
$F$. To do the $2\to3$ move determined by the face $F$ while
preserving the barycentric arc data, we map the arcs in $\tau_N$ and
$\tau_S$ into the model bipyramid in $\R^3$ using the barycentric
coordinates. We first require a fixed map from the vertices of
$\tau_N$ and $\tau_S$ to the bipyramid of $\R^3$. The particular
assignment is determined by whether the move is called from the point
of view of $\tau_N$ or $\tau_S$ and from the internal labeling of the
vertices of the face $F$. We use a map identifying $\tau_N$ with
$\Delta_N$ and $\tau_S$ with $\Delta_S$. We also require a map
identifying the three new tetrahedra in $\cS$ with the tetrahedra
$\Delta_A$, $\Delta_B$, $\Delta_C$ in the other triangulation of the
bipyramid. In $\cS$, let $\sigma_A,\sigma_B$, and $\sigma_C$ be the
three tetrahedra incident the new central edge.  For precise details
on these vertex correspondences, see the file
\texttt{barycentric\_geometry.py} in \cite{CodeAndData}.

Barycentric coordinates in $(\tau_N$, $\tau_S)$, and
$(\sigma_A,~\sigma_B,~\sigma_C)$ determine unique points in the
bipyramid via the vertex correspondence and the barycentric
coordinates of the two triangulations of the
bipyramid. This then determines a map sending barycentric arcs from
$\cT$ and $\cS$ to arcs in the bipyramid. If $a$ is a barycentric arc
in a tetrahedron $\tau$ in $\cT$ or $\cS$, let
${\tt{arc\_embedding}}_{\tau}(a)$ be the corresponding arc in the
bipyramid.

\begin{remark}
  It is crucial that we check if the arcs in the bipyramid are in
  general position (in the sense of Section~\ref{sec: tris with
    curves}) with the \3-tetrahedron triangulation. For example, it
  can happen that an arc that is disjoint from the 1-skeleton of the
  \2-tetrahedron triangulation of the bipyramid intersects the
  vertical edge in the \3-tetrahedron triangulation. When there is
  any such \emph{general position failure}, we perturb the north and
  south poles of the bipyramid slightly and repeat the above process.
\end{remark}

PL arcs in the bipyramid can also be described by barycentric arcs in
$\cT$ and $\cS$. Given distinct points $p$ and $q$ in the bipyramid,
there is a sequence of barycentric arcs $a_j$ contained in either
$(\tau_N, \tau_S)$ or in $(\sigma_0, \sigma_1, \sigma_2)$, such that
the line segment between $p$ and $q$ is the concatenation of the line
segments ${\tt{arc\_embedding}}_{a_j.{\tt{tet}}}(a_j)$ is where
$a_j.\tt{tet}$ is the tetrahedron containing $a_j$.  This is done by
describing the line segment from $p$ to $q$ as a weak barycentric arc
$a_{p,q}(\tau)$ for each tetrahedron $\tau$ in the relevant
triangulation, then trimming these weak barycentric arcs. The function
that takes an arc $a$ in the bipyramid and adds this sequence of
barycentric arcs to the tetrahedra $(\tau_N, \tau_S)$ or
$(\sigma_A, \sigma_B, \sigma_C)$ in $\cT$ or $\cS$ is denoted
${\tt{arc\_pullback}}_{\cT}(a)$ or ${\tt{arc\_pullback}}_{\cS}(a)$.

\subsection{Transferring the arcs}

By combining ${\tt{arc\_pullback}}$ and ${\tt{arc\_embedding}}$, we
define a function $\tt transfer\_arcs$ that takes barycentric arcs in
$\cT$, maps them into the bipyramid, then pulls them back to
$\cS$. This enables us to build a version of the $2\to3$ move that
transfers arcs from $\cT$ to $\cS$:

\begin{algothm} ${\tt with\_arcs[2\to3]}$
  \begin{enumerate}
  \item Identify $\tau_N$ with $\Delta_N$ and $\tau_S$ with $\Delta_S$
    in the bipyramid in $\R^3$.
  \item Do the $~2\to3$ move to get new tetrahedra in $\cS$ that are
    identified with the corresponding tetrahedra in the
    \3-tetrahedron triangulation of the bipyramid.
  \item Apply $\tt transfer\_arcs$: for each arc $a$ in $\tau_N$ and
    $\tau_S$, add ${\tt{arc\_embedding}}_{a.\tt tet}(a)$ to list
    $\tt arcs$. For each arc $b$ in $\tt arcs$, apply
    ${\tt{arc\_pullback}}_{\cS}(b)$.
  \end{enumerate}
\end{algothm}

The above approach is easily modified to accommodate the $3\to 2$
move. For the $4\to4$ move, this approach works with minor
modifications if one uses a bipyramid with square base. We therefore
can can implement methods $\tt with\_arcs[3\to2]$ and
$\tt with\_arcs[4\to4]$ that handle barycentric arcs.

\subsection{Simplifying arcs}

\label{sec: simp arcs}
Given the inputs (\ref{item: ideal input}) and (\ref{item: input
  moves}) of Section~\ref{sec: outline}, the Pachner moves with arcs
machinery always produces the desired link $L$ in the base
triangulation $\cT_0$.  However, even in the smallest examples,
applying the sequence of Pachner moves to $\cT$ produces incredibly
complicated configurations of arcs in $\cT_0$ encoding $L$.  This
complexity makes necessary computational geometry tasks prohibitively
expensive. Fortunately, much of this complexity is not topologically
essential, and the number of arcs can be decreased dramatically by the
basic simplifications we now describe.  Without these, applying our
full algorithm to an ideal triangulation $\Tdot$ with just two
tetrahedra resulted in 838 arcs and an initial link diagram with 5,130
crossings; with the simplifications, there are 19 arcs and 35
crossings. A 3-tetrahedra ideal triangulation resulted in 129,265 arcs
compared to 27 with simplifications. An example with 10 tetrahedra
would be impossible without these simplifications. Our two kinds of
simplification moves are shown in Figure~\ref{fig: simplify}.

\begin{figure}
  \centering
  \begin{tikzpicture}[nmdstd]
    \begin{scope}[shift={(0, 4.25)}]
      \node[above right] at (0, 0)
         {\includegraphics[height=3cm]{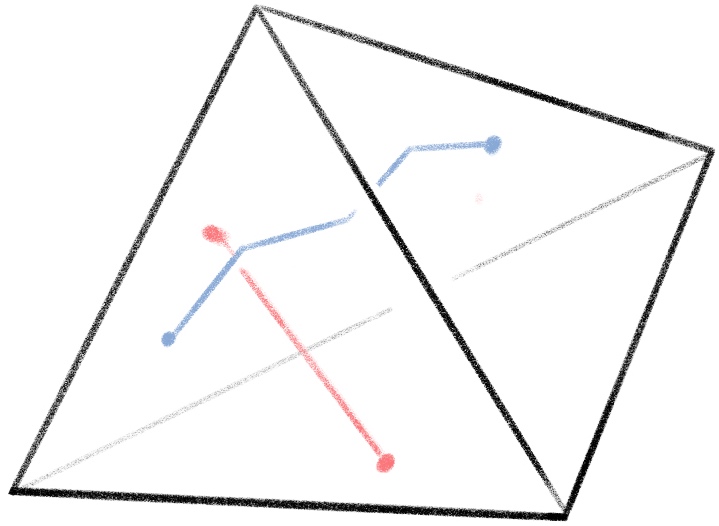}};
      \node[above right] at (7, 0)
         {\includegraphics[height=3cm]{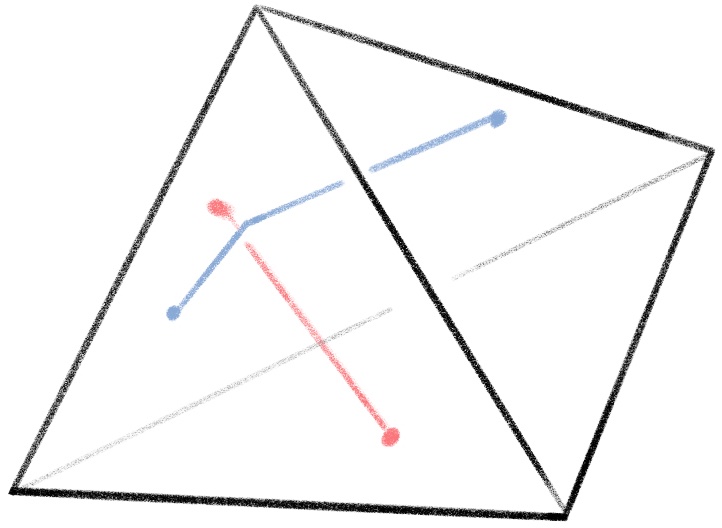}};
      \draw [->, line width=1.1pt] (5, 1.75) -- +(1.6, 0) node
         [midway, below=2pt] {straighten};
      \node at (2.27, 2.3) {$a$};
      \node at (2.65, 2.45) {$b$};
    \end{scope}

    \node[above right] at (0, 0) {\includegraphics[height=3cm]{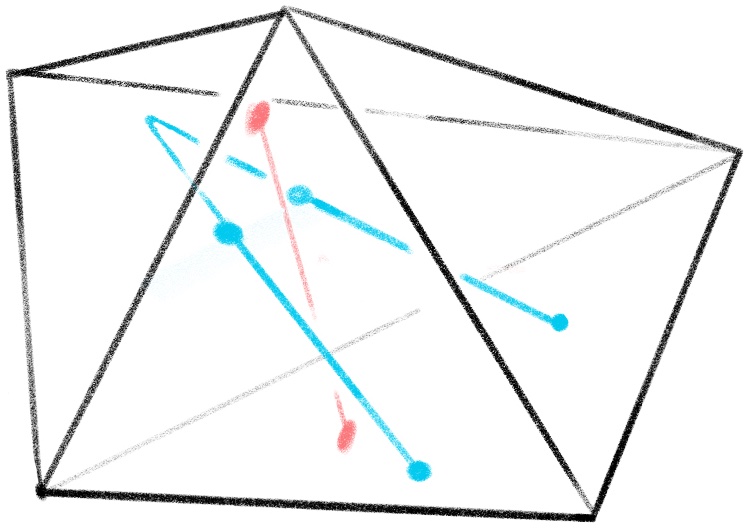}};
    \node[above right] at (7, 0) {\includegraphics[height=3cm]{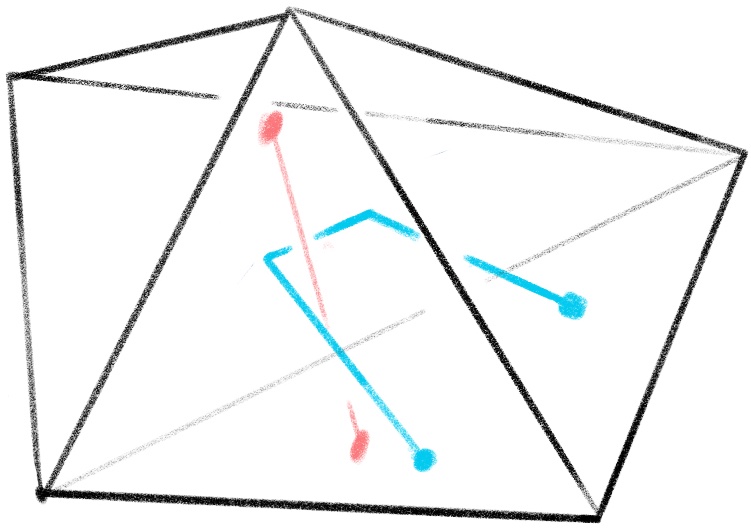}};
    \draw [->, line width=1.1pt] (5, 1.75) -- +(1.6, 0) node [midway, below=2pt] {push};
    \node at (1.25, 2.48) {$a$};
    \node at (1.0, 2.15) {$b$};
    \node at (0.9, 3.1) {$\tau$};
    \node at (4.15, 1.1) {$\tau'$};
    \node[rotate=25] at (1.15, 1.1) {$F$};
  \end{tikzpicture}

  \vspace{0.5cm}

  \caption{The \emph{straighten} move removes unnecessary bends
    in the link, and the \emph{push} move reduces unnecessary
    intersections with the 2-skeleton.}
  \label{fig: simplify}
\end{figure}
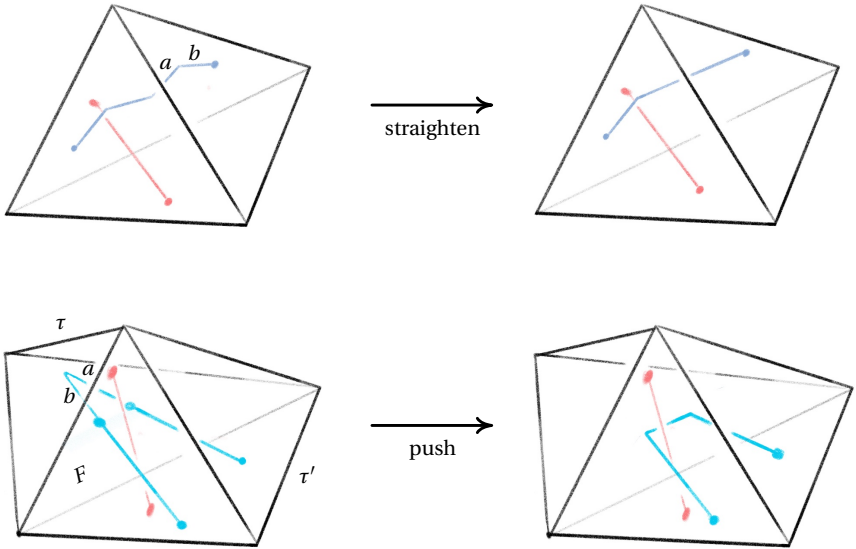

The first is \texttt{straighten}, which takes as input a tetrahedron
$\tau$ with barycentric arcs. It then checks for each arc $a$ in
$\tau$ if the pair of arcs $a$ and $b = a.\tt next$ can be replaced
with a single arc that runs from $a.\tt start$ and $b.\tt end$. The
check is that no other arc in $\tau$ has an interior intersection with
the triangle spanned by $a$ and $b$.  The other move is \texttt{push},
which removes unnecessary intersections with $\cT^2$.  When $a$ starts
on the same face $F$ that $b = a. \tt next$ ends on, it checks whether
any other arc intersects the triangle $a$ and $b$ span. If there are
none, the move replaces $a$ and $b$ with an arc in the tetrahedron
$\tau'$ glued to $\tau$ along $F$. This often produces a bend
that can then be removed by a straighten move.

\subsection{Putting the pieces together}

Let $\cT$ be a layered filling triangulation with arcs encoding the
core curves of the filling and let $(P_i)$ be simple Pachner moves
reducing $\cT$ to $\cT_0$. By factoring any $2\to0$ or $0\to2$ moves,
see Section \ref{sec: 2 to 0}, we can always turn a sequence of
Pachner moves into a sequence of simple Pachner moves.

Our process for producing a barycentric link
in $\cT_0$ that is isotopic to the initial $L$ is:

\begin{algothm}
  ${\tt with\_arcs[apply\_Pachner\_moves]}\big(\cT, (P_i)\big)$\\
  Start with $\cT' \assign \cT$ and loop over the $P_1, P_2, \ldots
  P_n$ as follows:
  \begin{enumerate}
  \item Apply ${\tt with\_arcs}[P_i]$ to $\cT'$ to get $P_i \cT'$ with
    arcs representing $L$.  Set $\cT' \assign P_i\cT'$.

  \item Loop over the tetrahedra $\tau$ in $\cT'$, applying $\tt push$
    and $\tt straighten$ until the arcs stabilize.
  \end{enumerate}
\end{algothm}

\subsection{Computational geometry issues}
\label{sec: CGAL}

Our algorithm requires many geometric computations with barycentric
arcs, e.g.~to test for one of our simplifying moves and to ensure we
do not violate the general position requirement of Section~\ref{sec:
  tris with curves}.  Difficult and subtle issues can arise here, and
much work has been done to ameliorate them; see \cite{Schirra2000} for
a survey.  We took the approach of having all coordinates in $\Q$ so
that we can do these computations exactly.  This entails a stiff
speed penalty and leads to points represented by rational numbers with
overwhelmingly large denominators.

We handle such denominators by rounding coordinates at each stage. As
long as the rounding process does not move the endpoint of an arc more
than one-fourth the minimum distance between any pair of arcs in the
link, the isotopy class of the link is unchanged by rounding, see
\cite[page 316]{Simon94}. We therefore guarantee correctness by
computing the minimum distance between arcs at every step of the
algorithm and then varying the rounding precision accordingly.

Alternatively, when the input manifold is hyperbolic, one can
generally efficiently certify correctness of the output diagram after
the fact by checking that its exterior is homeomorphic to the manifold
in part (\ref{item: ideal input}) of the input.  This is only
marginally faster than dynamically varying the rounding precision in
our current implementation, so the final version only uses the
approach that is independent of hyperbolicity.  However, this does
give us a completely independent check on the correctness of our code.


\section{Factoring the 2-to-0 move}
\label{sec: 2 to 0}
As mentioned in Section~\ref{sec: pachner arcs}, we factor each
$2 \to 0$ move into a sequence of $2\to3$ and $3\to2$ moves so that we
can carry along the barycentric link. This factorization is quite
delicate in certain unavoidable corner cases; we next outline our
method and then provide a detailed account in Section~\ref{app: 2 to
  0}. To begin to understand the $2\to0$ move, first consider its
inverse $0 \to 2$ move shown in Figure~\ref{fig: main birdseye}.  The
possible $0\to2$ moves in Figure~\ref{fig: main birdseye} correspond
to a pillow splitting open the \emph{book of tetrahedra} around the
edge $e$. Following \cite{Segerman2017}, we call this pillow a
\emph{bird beak} with upper and lower mandibles that pivot around the
two outside edges of the beak (viewed from above, these are the purple
and black vertices in the top right of Figure~\ref{fig: main
  birdseye}).  On both sides of the bird beak are \emph{half-books} of
tetrahedra, together forming a \emph{split-book}. When applying the
inverse $2\to0$ move, the two half-books combine to form a book of
tetrahedra assembled around the central edge.

The simplest $2\to0$ move is when there are two valence-2 edges that
are opposite each other on a single tetrahedron, as shown in
Figure~\ref{fig: main basecase}; equivalently, one of the half-books
has a single tetrahedron.  This \emph{base case} is handled by
Matveev's $V$ move, the composition of four $2\to3$ and $3\to2$ moves
of \cite[Figure 1.15]{Matveev2007}. To reduce other instances of the
$2\to0$ move to the base case, we rotate a mandible of the bird beak,
moving tetrahedra from one half-book to the other until one contains
only a single tetrahedron. Because the tetrahedra in the split-book
may repeat or be glued together in strange ways, this is rather
delicate.  When things are sufficiently embedded, Segerman
\cite{Segerman2017} showed:

\begin{figure}
  \centering
  \begin{tikzoverlay}[width=0.55\textwidth]{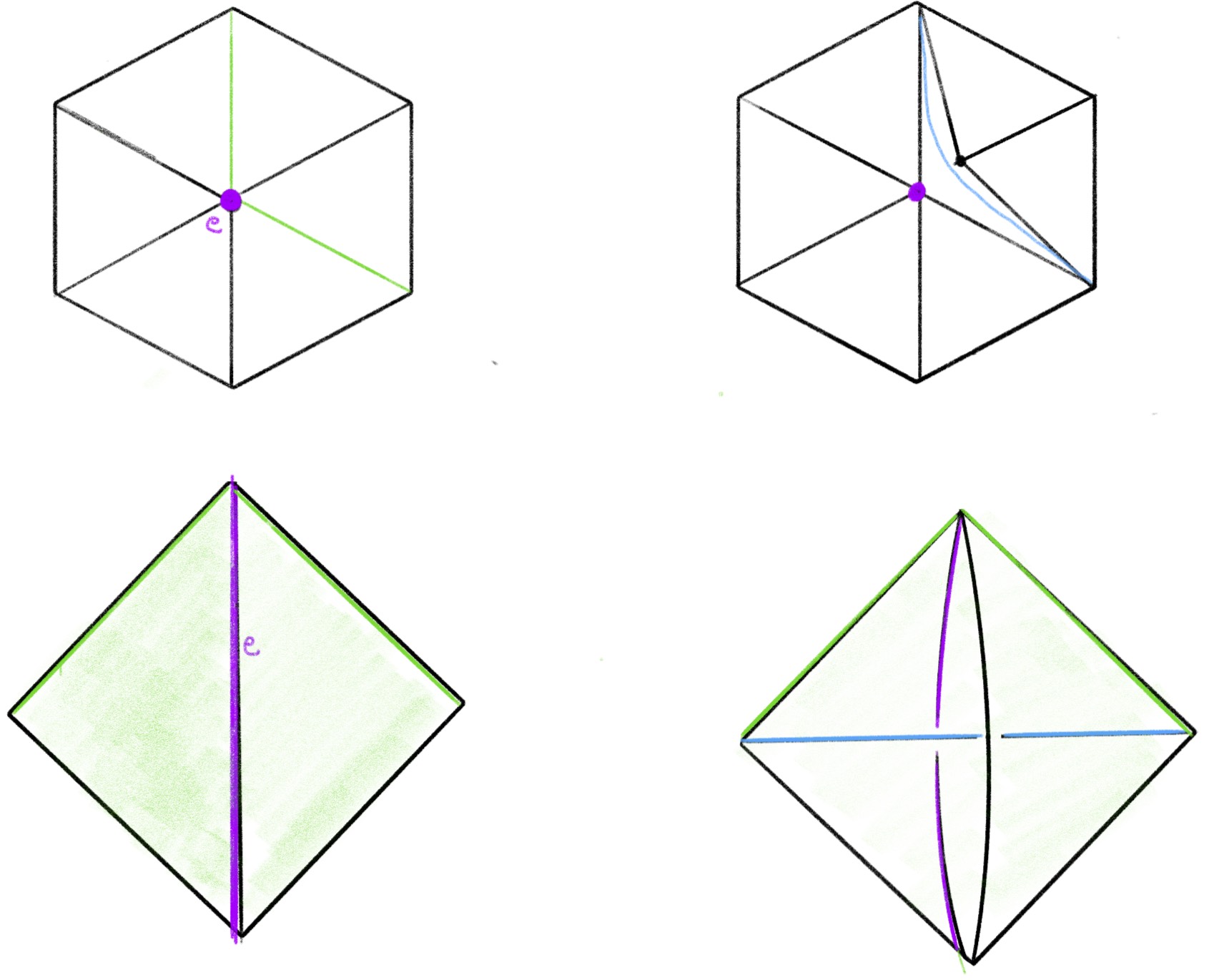}
    \begin{scope}[line width=0.7pt, ->]
      \draw (42.0,64.3) -- +(13, 0);
      \draw (42.0,19.7) -- +(13, 0);
    \end{scope}
  \end{tikzoverlay}

  \vspace{-0.25cm}

  \caption{At top, a cross section of a $0\to2$ move; at bottom is a
    close-up of the inflation of the pillow.  The move is performed on
    the pair of green faces meeting along the purple edge $e$ at left.
    The resulting pillow is a \emph{bird beak}, which splits open the
    book of tetrahedra about $e$. In the top right, the purple
    and black dots give edges that join together above and below the
    cross section. }
     \label{fig: main birdseye}
\end{figure}

\begin{figure}
  \centering
  \begin{tikzoverlay}[width=0.55\textwidth]{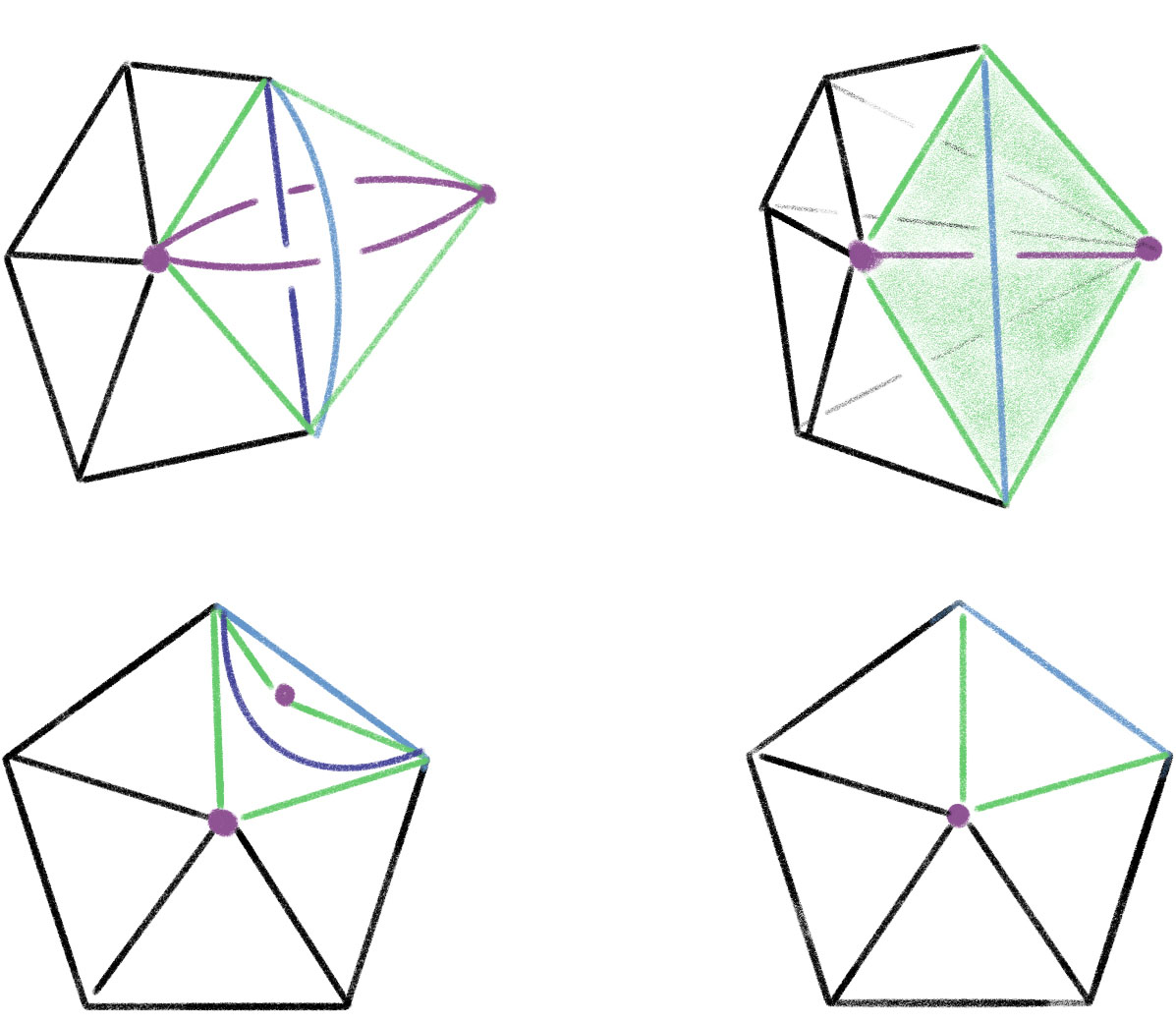}
    \begin{scope}[line width=0.7pt, ->]
      \draw (40.0, 58.9) -- +(17.5, 0);
      \draw (40.0, 11) -- +(17.5, 0);
    \end{scope}
  \end{tikzoverlay}

  \caption{The base case of the $2\to0$ move at top with
    the cross section at bottom.}
  \label{fig: main basecase}
\end{figure}

\begin{proposition}
\label{prop: linear rotation}
Suppose $e$ is a valence-2 edge where the half-books
adjacent to the bird beak are embedded and contain $m$ and $n$
tetrahedra respectively.  Then the $2 \to 0$ move can be implemented
by at most $2 \cdot \min(m, n) + 2$ basic $2 \to3$ and $3\to2$ moves.
\end{proposition}

\begin{proof}
We can rotate a mandible by one tetrahedron using the two basic moves of
\cite[Figure~11]{Segerman2017}.  With $\min(m, n) - 1$ such rotations
we can reduce the smaller of the half-books to a single tetrahedron.
As already noted, the base case can be done in four moves.
\end{proof}

\begin{remark}
One cannot in general factor a $2 \to 0$ move into a
sublinear number of $2\to3$ and $3\to2$ moves: the $2 \to 0$ move
amalgamates two edges of valence $m + 1$ and $n + 1$ into a single
edge of valence $m + n$, and each $2\to3$ or $3\to2$ move only
changes valences by a total of 12 (counting with multiplicity).
\end{remark}

\begin{figure}
  \centering
  \begin{tikzoverlay}[width=0.75\textwidth]{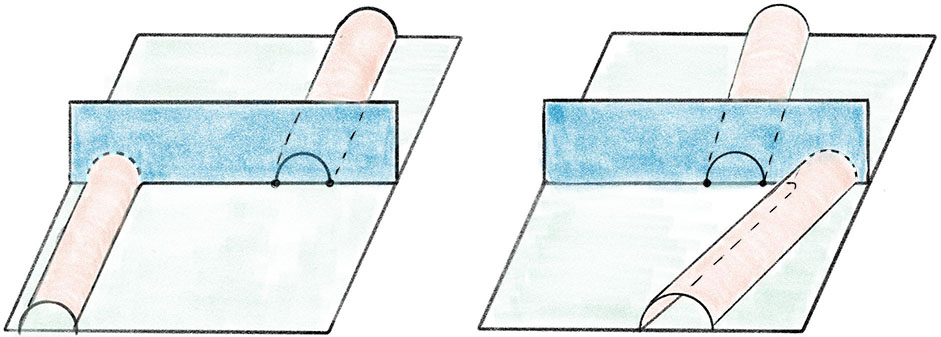}
    \draw[->, line width=0.8pt] (45.0,16.5) -- (54.9,16.5);
  \end{tikzoverlay}

  \vspace{0.5cm}

  \caption{The \emph{endpoint-through-endpoint} move in a special
    spine.}
  \label{fig: main end to end}
\end{figure}

The tricky case is when additional faces of the bird beak are glued to
each other.  There are two fundamentally different ways for this to
happen, shown in Figures~\ref{fig: twist1} and~\ref{fig: twist2} of
Section~\ref{app: 2 to 0}.  The untwisting of these extremely
confusing arrangements is done by the endpoint-through-endpoint move
of Figure~\ref{fig: main end to end}, which is in the dual language of
special spines from Section~\ref{app: 2 to 0}. Matveev's
factorization of the endpoint-through-endpoint move is described in
Figure 1.19 of \cite{Matveev2007}. We simplify this factorization from
14 moves to 6; the key is Proposition~\ref{prop: just four
  moves}. This simplification was essential for determining the exact
sequence of moves needed to factor the $2\to0$ move. Dual to the
endpoint-through-endpoint move are a pair of \emph{untwist the beak}
moves, one for each of the situations in Figures~\ref{fig: twist1}
and~\ref{fig: twist2}, see Section~\ref{app: 2 to 0}. We can
thus factorize the $2\to0$ move as follows:

\begin{algothm} ${\tt factor[2\to0]}$
\begin{enumerate}
\item \label{item: main base case}
  If we are in the base case, do the sequence of moves
  in the triangulation dual to the factorization of the $V$ move in
  Figure 1.15 of \cite{Matveev2007} and exit.

\item If we are in the twisted cases described by Figures~\ref{fig:
    twist1} and~\ref{fig: twist2} in Section~\ref{app: 2 to 0}, do
  the appropriate untwist the beak move. Otherwise, rotate the
  mandible by one tetrahedron.

\item  Go to to step \ref{item: main base case}.

\end{enumerate}
\end{algothm}


\subsection{Dealing with twisted beaks using special spines}
\label{app: 2 to 0}

In this section, we detail how to
handle the situation where additional faces of the bird beak are
identified, which can happen because the tetrahedra in the split-book,
shown in the upper-right of Figure~\ref{fig: main birdseye}, may
repeat or be glued together in strange ways.  To handle this delicate
issue, we need to take the dual viewpoint.

Dual to a triangulation is a \emph{special spine}. A \emph{spine} of a
compact 3-manifold $M$ with nonempty boundary is a polyhedron $P$ such
that $M$ collapses to $P$. A spine of a closed 3-manifold $M$ is a
spine of the complement of an open ball in $M$. Any triangulation
$\cT$ of $M$ determines a dual cellulation $\cS$ whose 2-skeleton is a
spine, see Figure 1.5 of \cite{Matveev2007}; the class of such spines
are the \emph{special spines},

The Pachner moves correspond to dual moves in the special
spine. Generally, it is easier to reason about the dual moves
modifying special spines, as there is a type of graphical
calculus. The move on a special spine dual to the $2\to3$ is called
the $T$ move and its inverse dual to the $3\to2$ move is denoted
$T^{-1}$, see Figure~\ref{fig: T move}. The special spine move
corresponding to the $0\to2$ move is the \emph{lune} move shown in
Figure~\ref{fig: L move}, and is denoted $L$. The move dual to $2\to0$
is the inverse lune move $L^{-1}$. It is also useful to have a certain
compound move, Matveev's $V$ move, shown in Figure~\ref{fig: V move},
which is a special case of the lune move that can be realized as a
sequence of four $T$ and $T^{-1}$ moves \cite[Figure
1.15]{Matveev2007}. Matveev showed that generally the $0\to2$ and
$2\to0$ moves can be factored as a sequence of $2\to3$ and $3\to2$
moves. In this section, we describe a simplification of one part of
Matveev's factorization, reducing the number of moves necessary for
that part from 14 to 6.

\begin{figure}
  \centering
  \includegraphics[width=0.55\textwidth]{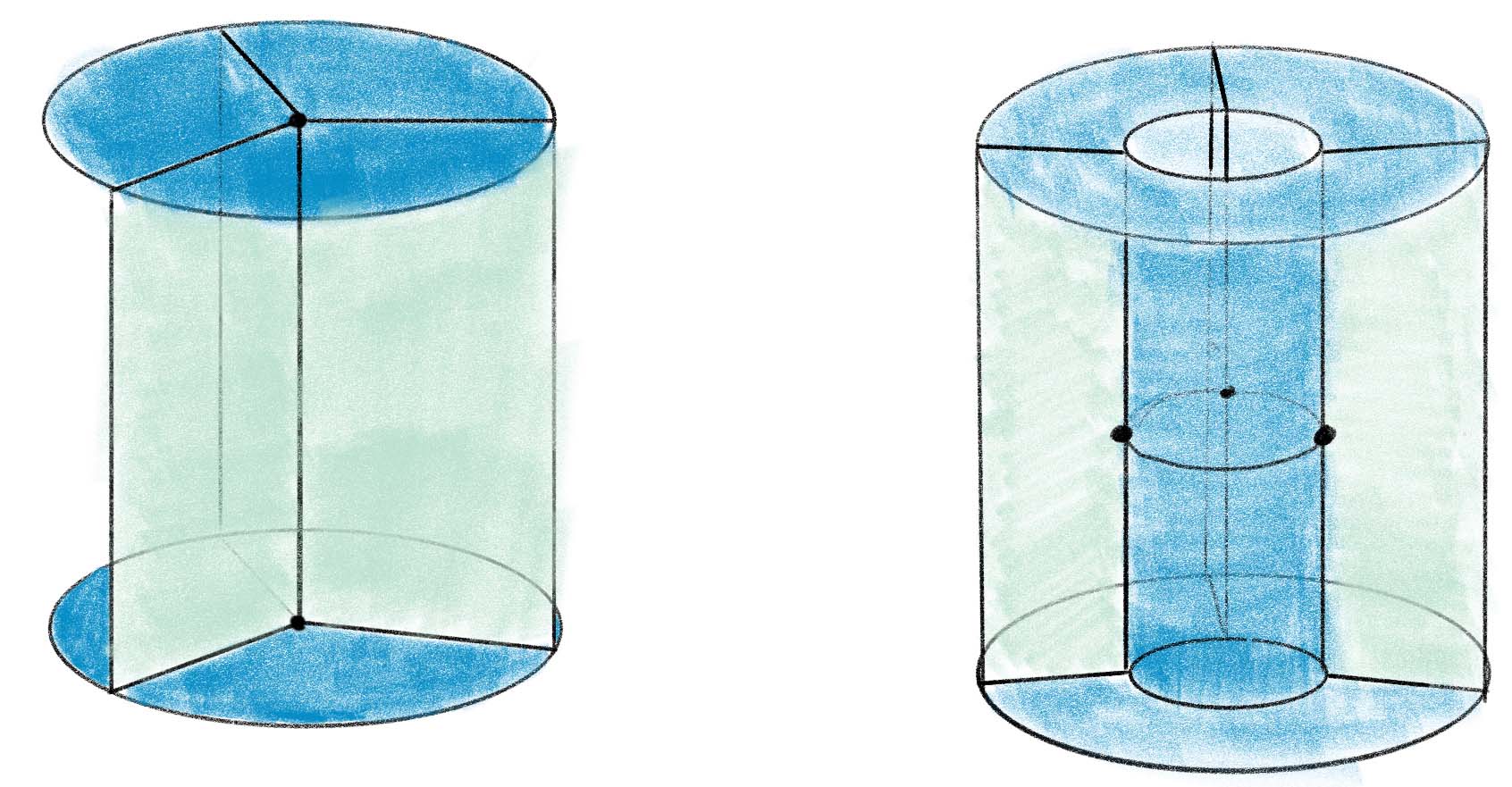}
  \caption{The $T$ move in the special spine dual to the $2\to3$ move.}
  \label{fig: T move}
\end{figure}

\begin{figure}
  \centering
  \includegraphics[width=0.56\textwidth]{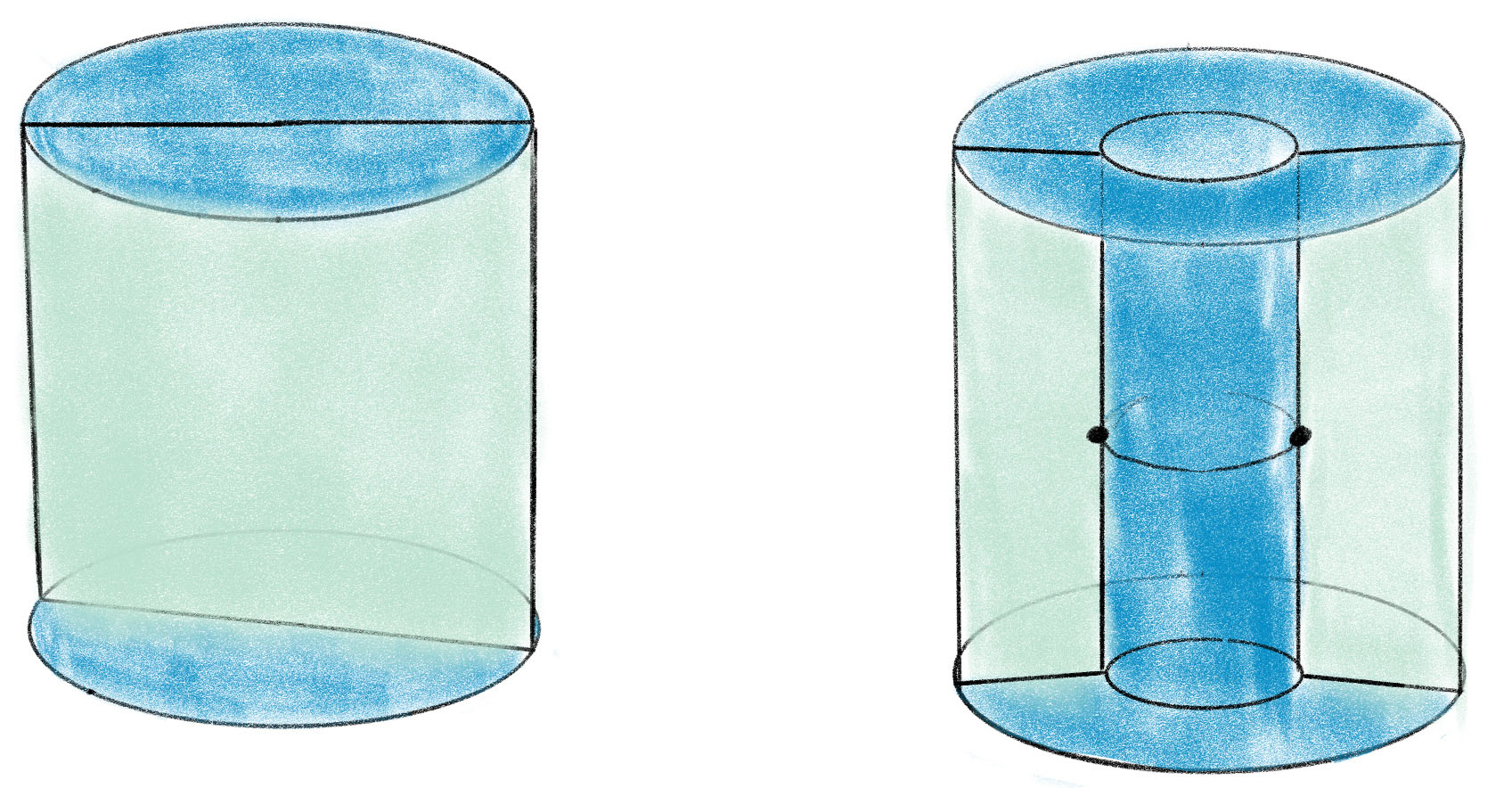}
  \caption{The lune move $L$ in the special spine dual to the $0\to2$ move.}
  \label{fig: L move}
\end{figure}

\begin{figure}
  \centering
  \includegraphics[width=0.75\textwidth]{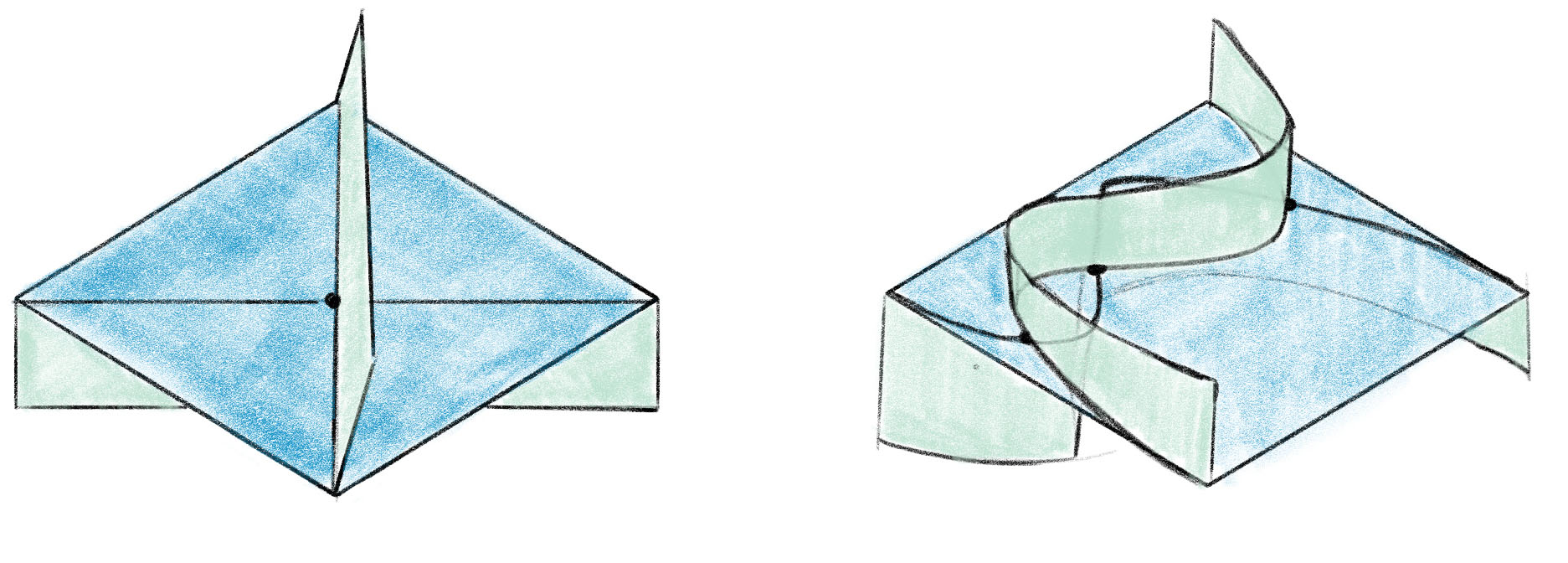}

  \vspace{-0.5cm}
  
  \caption{Matveev's V move.}
  \label{fig: V move}
\end{figure}

\begin{figure}[!tbp]
  \centering
  \begin{tikzoverlay}[width=0.9\textwidth]{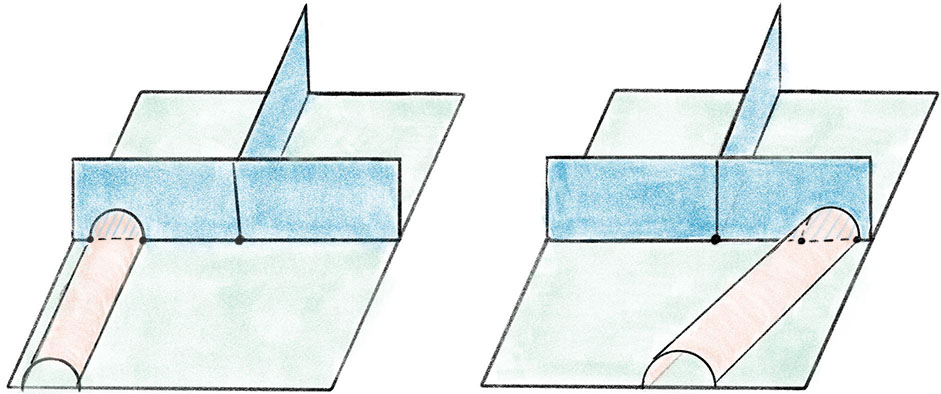}
    \draw[->, line width=0.8pt] (45.0,16.5) -- (54.9,16.5);
  \end{tikzoverlay}

  \caption{The \emph{endpoint-through-vertex} move on a special
    spine. There are three vertices of the singular locus in both
    configurations. This is dual to the move rotating a mandible. }
  \label{fig: end to vertex}
\end{figure}

The first move we need to transfer tetrahedra from one half-book to
the other is the dual version of the endpoint-through-vertex move in
Figure \ref{fig: end to vertex}. The endpoint-through-vertex move
involves a single $T$ and $T^{-1}$ move. When the split-book is not
twisted in strange ways, this move suffices to reduce to the base
case. A thorough account of this move, describing both the
triangulation and special spine versions, is given in
\cite{Segerman2017}.

The special spine move needed to deal with twisted half-books is the
endpoint-through-endpoint move pictured in Figure~\ref{fig: end to
  end}. A factorization of the endpoint-through-endpoint move is given
in Figure 1.19 of \cite{Matveev2007}, and it is this factorization
that we simplify from 14 to 6 moves.  Converting from the spine
version of the endpoint-through-endpoint move to the triangulation
version is quite subtle, so reducing the number of moves in the
factorization is incredibly helpful.

\begin{figure}
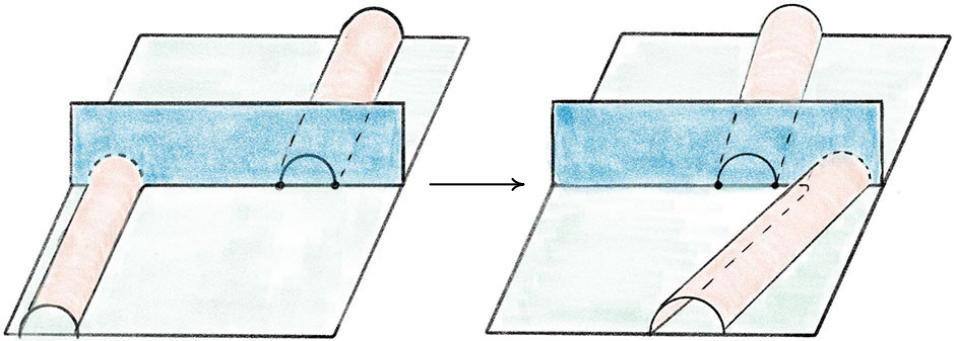

  \centering
  \begin{tikzoverlay}[width=0.9\textwidth]{end2end}
    \begin{scope}[font=\small]
      \draw[->, line width=0.8pt] (45.0,16.5) -- (54.9,16.5);
    \end{scope}
  \end{tikzoverlay}

  \caption{The \emph{endpoint-through-endpoint} move in a special
     spine. There are only two vertices of the singular locus in each
     picture. You can walk through the ``tunnel''
     starting in the lower left and end up ``behind'' the vertical
     blue wall; in contrast, the back tunnel dead-ends into the vertical
     blue wall.  This move is needed to reduce to the base case when the
     tetrahedra in the split-book are twisted up in various ways. We
     implement this move for the dual triangulation as a sequence of
     six $2\to3$ and $3\to 2$ moves.}
     \label{fig: end to end}
\end{figure}
Our simplification of the factorization is in fact a simplification of
one part of the endpoint-through-endpoint move. Specifically, the
factorization in Figure 1.19 of \cite{Matveev2007} consists of three
stages.  The first stage uses the move sequence $V, T, T^{-1}$, the
second the sequence $T, T^{-1}$, and the final stage is, up to taking
the mirror image, the inverse of the first stage.  Expanding $V$ to
its four constituent $T$ and $T^{-1}$ moves means both the first and
third stages require six $T$ and $T^{-1}$ moves each.  However, the
next proposition shows that one only needs two $T$ moves for each of
these stages, saving us eight moves overall:

\begin{proposition}
  \label{prop: just four moves}
  The below move on a special spine can be achieved by two $T$ moves:
  {

  \centering
    \begin{tikzoverlay}[width=.85\textwidth]{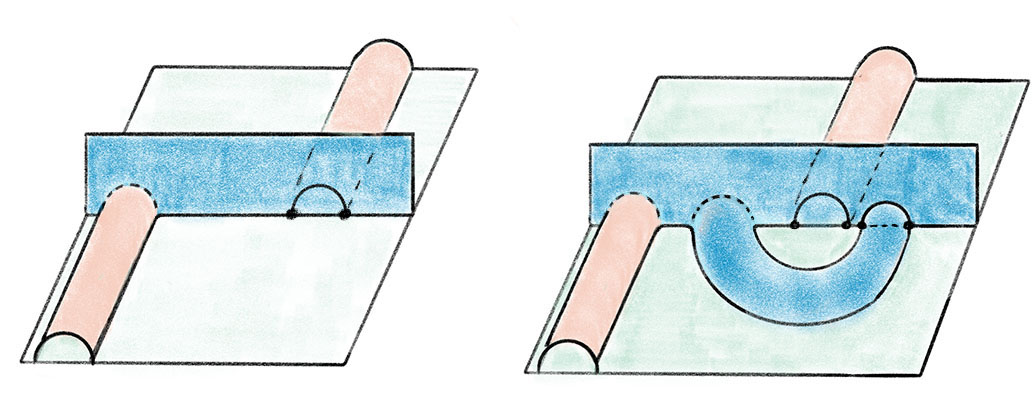}
      \begin{scope}[nmdstd, font=\small]
        \draw[->, line width=0.8pt] (42.2,17.9) -- +(11, 0);
        \node[] at (30.8,20.4) {$\alpha$};
        \node[] at (79.6,19.7) {$\alpha$};
        \node[] at (85.9,19.0) {$\mu$};
      \end{scope}
    \end{tikzoverlay}

  }
  \noindent
  There are two vertices initially, and four at the end.  You can
  enter any of the lower tunnels from the top-left green region; for
  the U-tunnel at right you would reach a dead-end at the vertical blue wall
  at $\mu$.
\end{proposition}

\begin{proof}
Zooming in on the part of the figure near the vertices, we see:

{
\centering
  \begin{tikzpicture}[nmdstd, font=\small]
    \node[above right] at (0, 0) {
      \begin{tikzoverlay}[width=.4\textwidth]{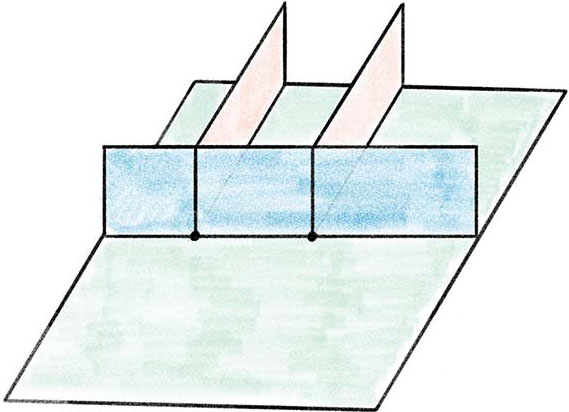}
        \node[] at (23.5,35.0) {$\delta$};
        \node[] at (40.9,35.0) {$\alpha$};
        \node[] at (65.5,35.0) {$\epsilon$};
        \node[] at (40,12.5) {$\gamma$};
      \end{tikzoverlay}};

    \node[above right] at (6, 0) {
      \begin{tikzoverlay}[width=.4\textwidth]{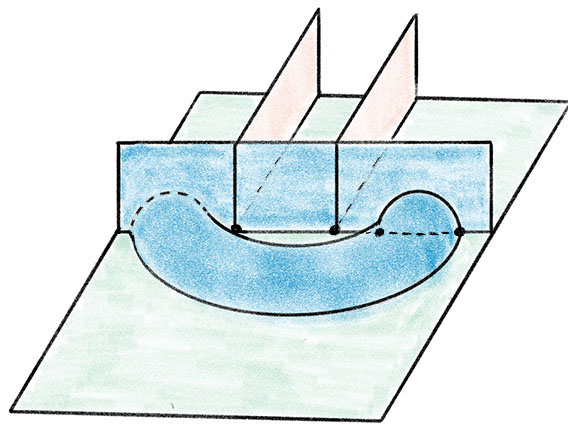}
        \node[] at (26.5,41.5) {$\delta$};
        \node[] at (48.5,41.5) {$\alpha$};
        \node[] at (70.5,41.5) {$\epsilon$};
      \end{tikzoverlay}};

      \draw[->, line width=0.7pt] (5.25, 2) -- +(1.5, 0);
    \end{tikzpicture}
}
  
\noindent
We can then realize the move as follows:

\begin{enumerate}
\item Rotate the walls left of $\alpha$:

  \vspace{-0.0cm}
  
  \hspace{4cm}%
  \begin{tikzoverlay}[width=.35\textwidth]{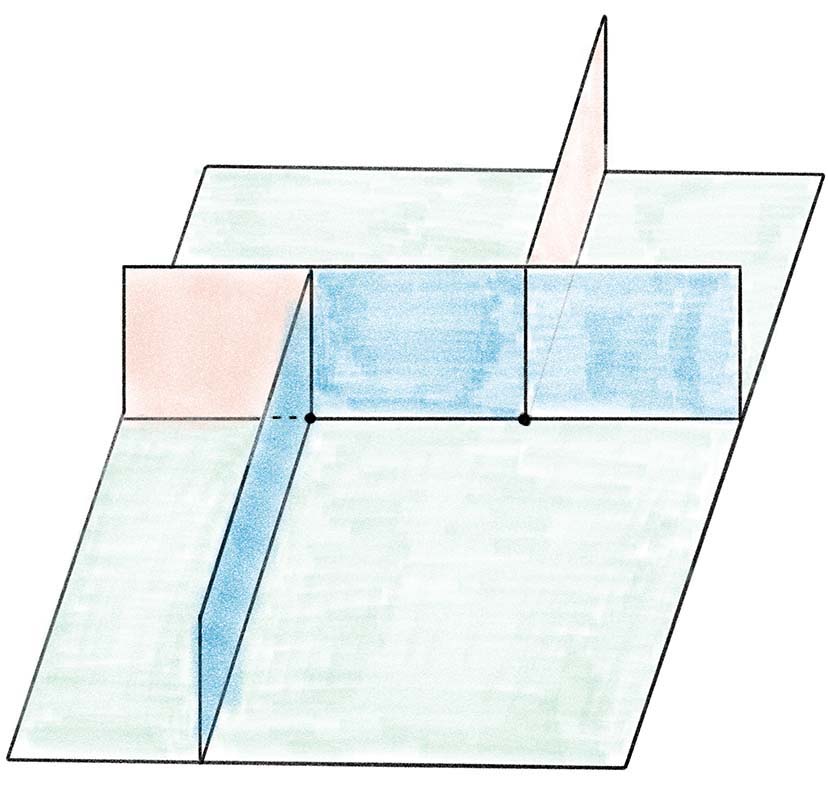}
    \begin{scope}[font=\small]
      \node[] at (32.5,38.6) {$\delta$};
      \node[] at (50.6,53.4) {$\alpha$};
      \node[] at (55.6,24.5) {$\gamma$};
      \node[] at (76.4,54.2) {$\epsilon$};
    \end{scope}
  \end{tikzoverlay}

  \vspace{0.25cm}
  
\item Fold so $\alpha, \varepsilon, $ and $\gamma$ are in the same
  plane:

\vspace{-0.25cm}
  
\begin{center}
  \begin{tikzoverlay}[width=.4\textwidth]{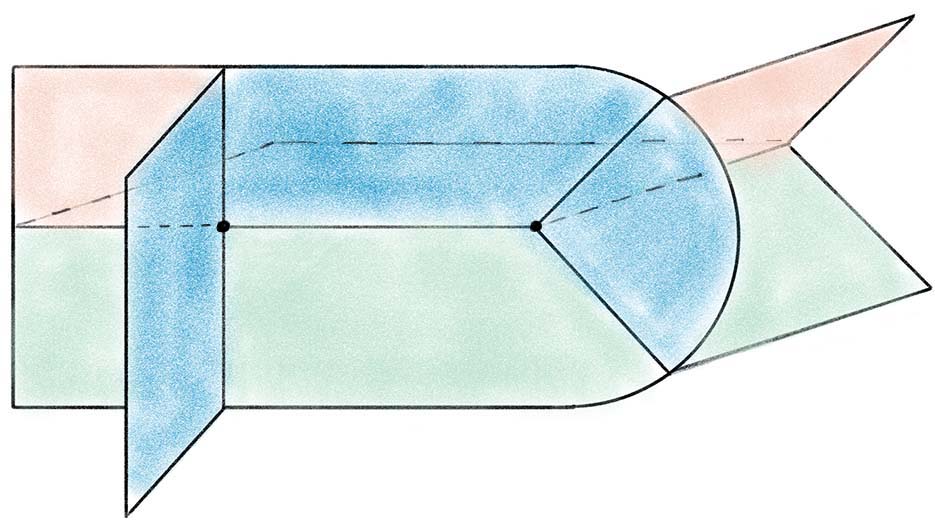}
    \begin{scope}[font=\small]
      \node[] at (18.6,25.6) {$\delta$};
      \node[] at (41.8,44.1) {$\alpha$};
      \node[] at (43.3,16.6) {$\gamma$};
      \node[] at (69.6,31.6) {$\epsilon$};
    \end{scope}
  \end{tikzoverlay}
\end{center}

\item Follow Figure 1.15 of \cite{Matveev2007} to obtain:

  \begin{center}
    \begin{tikzoverlay}[width=.9\textwidth]{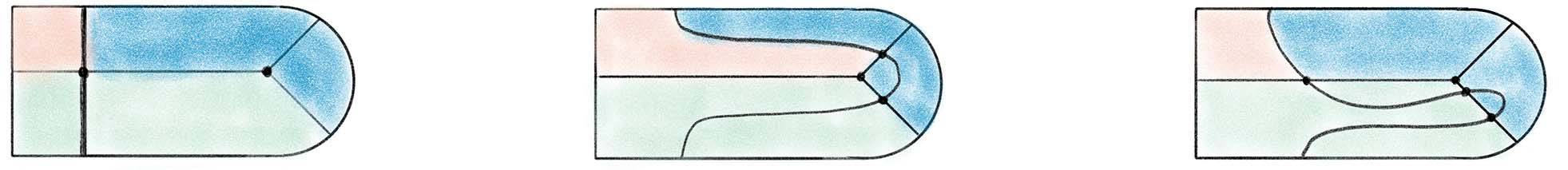}
      \begin{scope}[nmdstd, line width=0.8pt, font=\small]
        \draw[->, line width=0.5pt] (13, -3) to [bend left=25] (5.6, 0.0);
        \node[right] at (13, -3) {$\delta$ viewed head-on};
        \node[] at (11.4,8.7) {$\alpha$};
        \node[] at (11.8,3.4) {$\gamma$};
        \node[] at (20.4,6.3) {$\epsilon$};
        \draw[->] (26.6,6.1) -- node[above] {$T$} +(8, 0);
        \draw[->] (64.0,6.1) -- node[above] {$T$} +(8, 0);
        \node[] at (88.2,8.2) {$\alpha$};
      \end{scope}
    \end{tikzoverlay}
  \end{center}

\item Redraw as:

  \vspace{-0.5cm}
  \hspace{3cm}%
  \begin{tikzoverlay}[width=.5\textwidth]{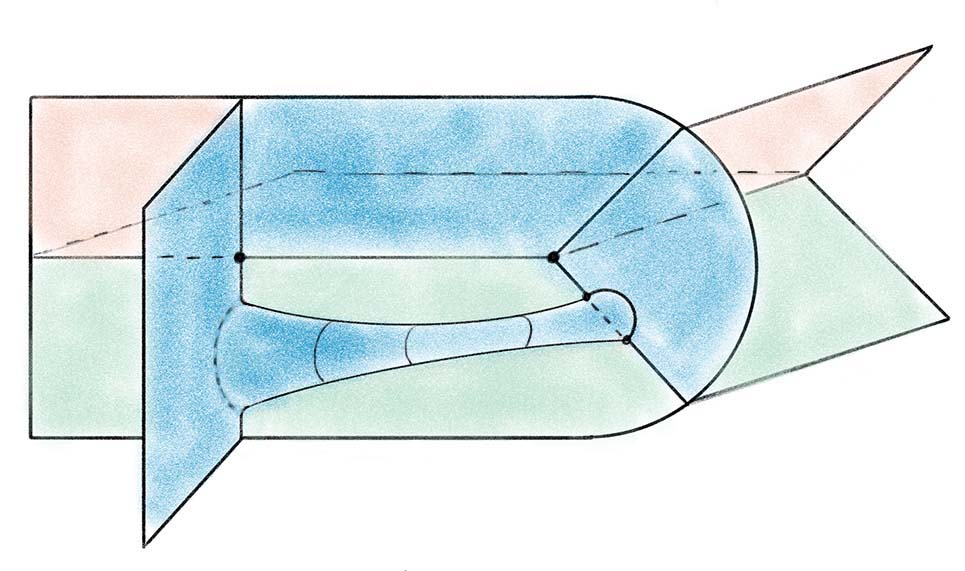}
    \begin{scope}[font=\small]
      \node[] at (20.0,28.9) {$\delta$};
      \node[] at (41.8,45.6) {$\alpha$};
      \node[] at (70.7,32.5) {$\epsilon$};
    \end{scope}
  \end{tikzoverlay}

\item Reverse the folds to obtain:

\vspace{-0.5cm}
  
\begin{center}
  \begin{tikzoverlay}[width=.4\textwidth]{part5-zoom}
    \begin{scope}[font=\small]
      \node[] at (29.6,45.2) {$\delta$};
      \node[] at (52.5,45.2) {$\alpha$};
      \node[] at (73.2,45.2) {$\epsilon$};
    \end{scope}
  \end{tikzoverlay}
\end{center}

\end{enumerate}
This completes the proof.
\end{proof}

\begin{corollary}
  \label{cor: end to end}
  The endpoint-through-endpoint move can be completed as a sequence of
  6 moves using the moves $T$ and $T^{-1}$ dual to the $2\to 3$ and
  $3\to2$ moves.
\end{corollary}

\begin{proof} Following Figure 1.19 in \cite{Matveev2007}, the
endpoint-through-endpoint move is obtained doing the move in
Proposition~\ref{prop: just four moves}, then doing one $T$
and one $T^{-1}$ move, and ending with the inverse of
Proposition~\ref{prop: just four moves}.
\end{proof}

We now need to dualize the endpoint-through-endpoint move to the
triangulation setting. We call the dual version of the
endpoint-through-endpoint move the \emph{untwist the beak}
move. Suppose we have a bird beak in a split-book of tetrahedra. There
are two possible twisted complexes contained in a half-book that can
arise, preventing us from rotating the mandible using the dual
endpoint-through-vertex move. Both of these are overcome using the
untwist the beak move.

In the first twisted case, there are at least two distinct tetrahedra
glued to the front of the pillow, and the back of the pillow is glued
to itself as indicated in Figure \ref{fig: twist1}, producing a solid
torus. In this case, the half-book around the front edge and the
half-book around the back edge overlap. In Figure \ref{fig: twist1}, the
leftmost and rightmost tetrahedra in the front half-book, labeled
$X$ and $Y$, appear in both the front and back half-books.

The other possible twisted case involves a front face of the pillow
being glued to the back face of the pillow, as indicated in Figure
\ref{fig: twist2}. In this case, the front tetrahedron of the pillow
is part of the half-book around the back edge, and the back tetrahedron
of the pillow is part of the half-book around the front edge.

\begin{figure}
  \centering
  \begin{tikzoverlay}[width=.65\textwidth]{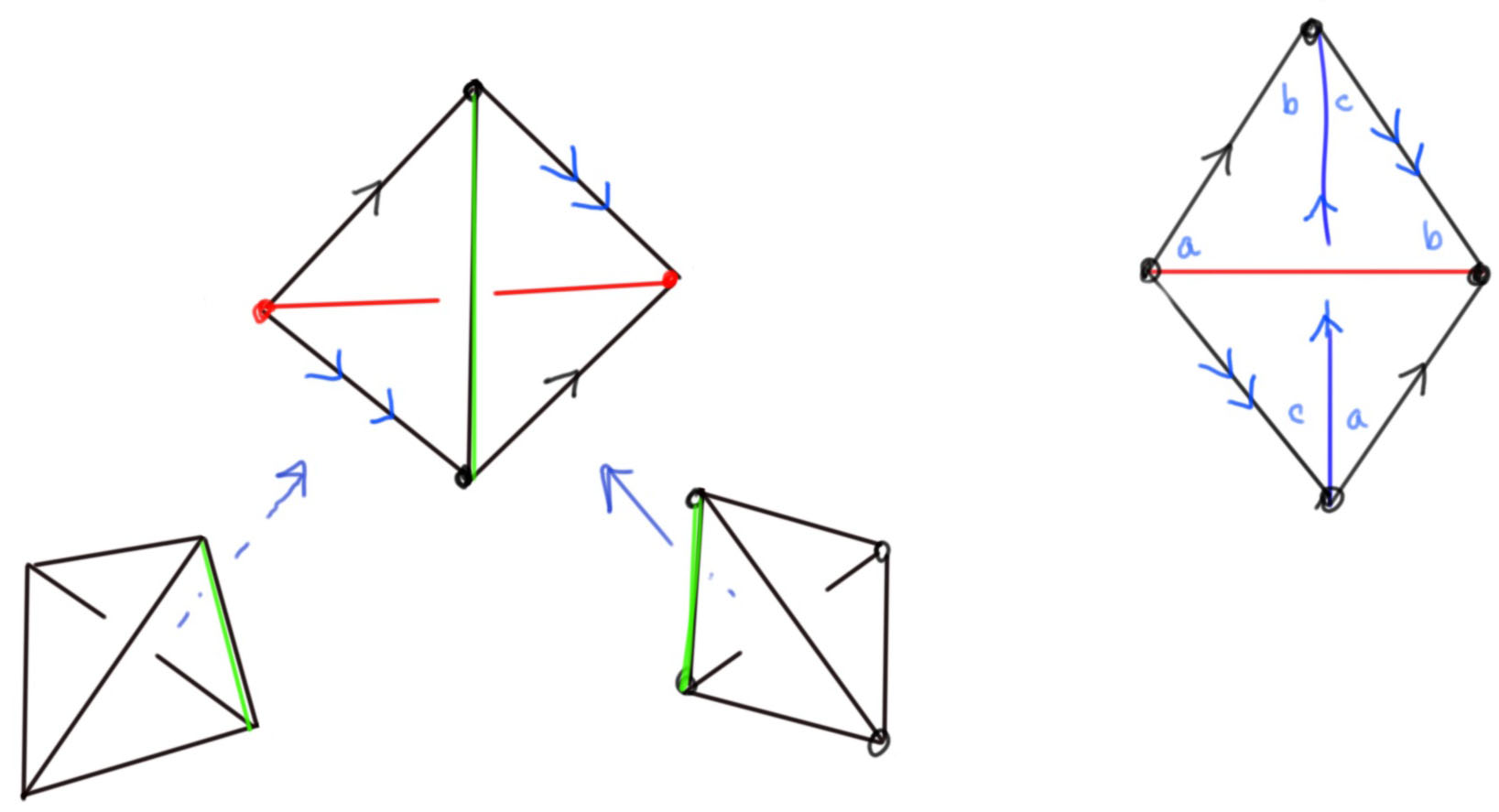}
    \begin{scope}[font=\small]
      \node[] at (10.6,1.2) {$X$};
      \node[below right] at (17.3,5.9) {$0$};
      \node[above right] at (13.4,18.4) {$1$};
      \node[above left] at (1.9,16.4) {$3$};
      \node[below left] at (1.5,1.1) {$2$};

      \node[] at (50.9,3.2) {$Y$};
      \node[below left] at (45.8,8.4) {$0$};
      \node[above left] at (46.5,21.0) {$1$};
      \node[above right] at (59.0,17.5) {$2$};
      \node[below right] at (58.5,4.7) {$3$};

      \node[] at (20.5,45.0) {$F$};
      \node[below=0.5] at (31.0,22.4) {$0$};
      \node[above=0.5] at (31.6,48.1) {$1$};
      \node[left=0.3] at (17.3,33.5) {$2$};
      \node[right=0.5] at (44.7,35.7) {$3$};

      \node[] at (76.9,47.1) {$B$};
      \node[above=0.5] at (87.4,52.2) {$0$};
      \node[below=0.5] at (88.6,20.9) {$1$};
      \node[left=0.5] at (76.4,36.2) {$2$};
      \node[right=0.5] at (98.5,36.2) {$3$};
    \end{scope}
  \end{tikzoverlay}

  \caption{One of the possible twisted complexes in a split-book of
    tetrahedra.  Here, the tetrahedron $F$ sits in front of $B$
    forming the valence-2 edge; said edge connects vertices 2 and 3 in
    both tetrahedra, with the pair of opposite edges joining vertices
    0 and 1 forming the ``equator'' around it.  The back tetrahedron
    $B$ forms a solid torus with its back two faces glued as indicated
    by the labels $\{a, b, c\}$.  The face opposite vertex 2 in
    tetrahedron $X$ is glued to the face of $F$ opposite vertex 3 so
    that the two edges connecting vertices 0 and 1 coincide; the
    tetrahedron $Y$ is glued to $F$ analogously.
  }
\label{fig: twist1}
\end{figure}

\begin{figure}
  \centering
  \begin{tikzoverlay}[width=.65\textwidth]{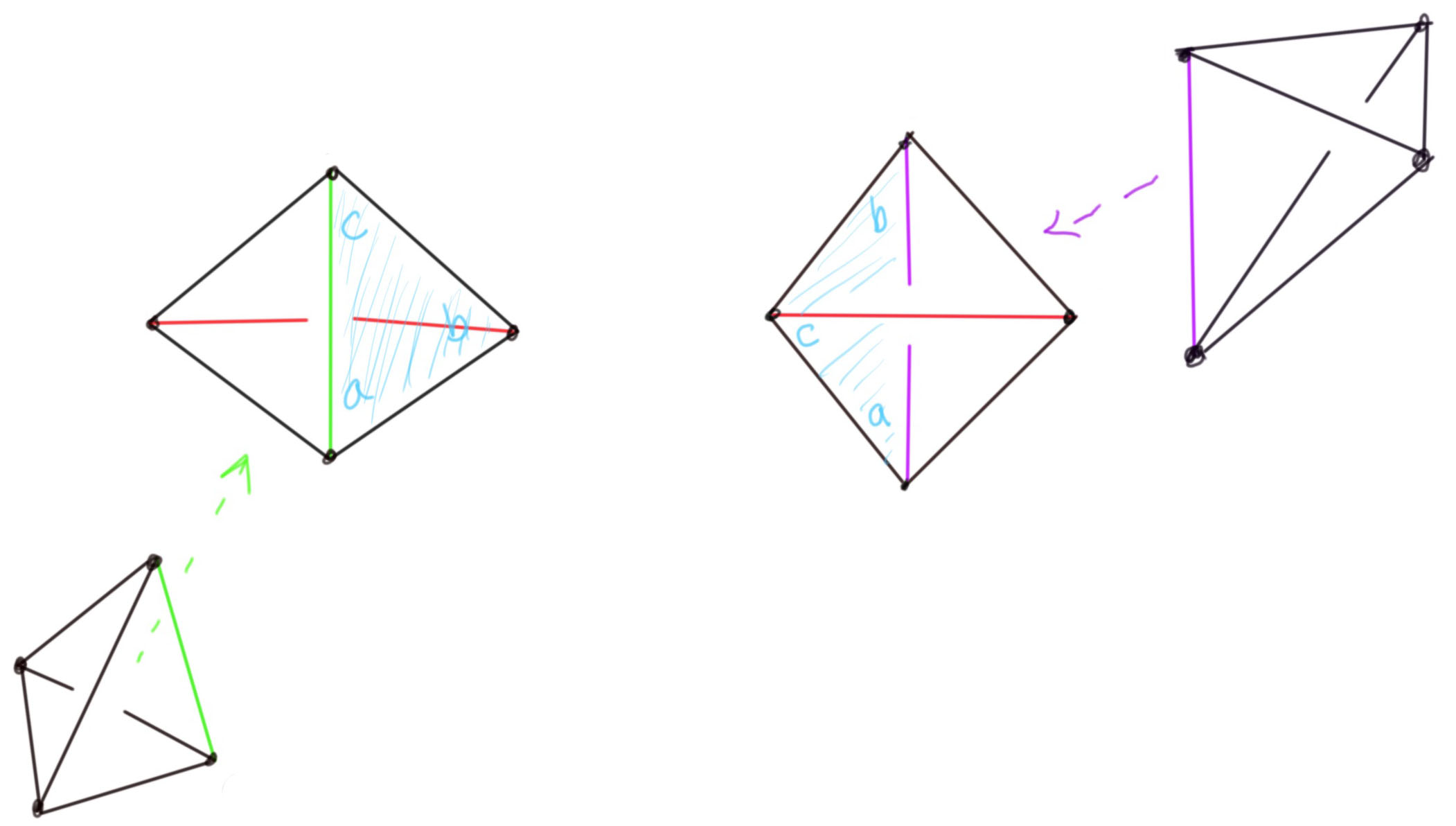}
    \begin{scope}[font=\small]
      \node[] at (9.9,-1.5) {$X$};
      \node[below right] at (14.5,4.9) {$0$};
      \node[above] at (10.2,18.8) {$1$};
      \node[below left] at (2.6,1.3) {$2$};
      \node[above left] at (1.3,10.0) {$3$};

      \node[] at (14.1,44.0) {$F$};
      \node[below=0.5] at (22.5,25.6) {$0$};
      \node[above=0.5] at (22.8,45.1) {$1$};
      \node[left=0.5] at (10.5,34.3) {$2$};
      \node[right=0.5] at (35.1,33.7) {$3$};

      \node[] at (53.8,44.3) {$B$};
      \node[above=0.5] at (62.4,47.1) {$0$};
      \node[below=0.5] at (62.0,23.5) {$1$};
      \node[left=0.5] at (52.8,35.1) {$2$};
      \node[right=0.3] at (73.5,35.1) {$3$};

      \node[] at (93,35.3) {$Y$};
      \node[above=0.7] at (81.4,52.9) {$0$};
      \node[below=0.7] at (81.8,32.3) {$1$};
      \node[below right] at (98.4,46.5) {$2$};
      \node[above right] at (97.8,55.1) {$3$};
    \end{scope}
    \end{tikzoverlay}
    \caption{Another possible twisted complex in a split-book of
      tetrahedra.  The tetrahedron $F$ sits in front of $B$ forming
      the valence-2 edge as in Figure~\ref{fig: twist1}. These
      tetrahedra share a third pair of faces (shaded with corners
      labeled $\{a, b, c\}$) which must be offset as shown as
      otherwise the $2\to0$ move would change the topology. Here, the
      back face of $X$ is glued to the unshaded front face of $F$ so
      that the vertices 0 and 1 coincide; similarly, the front face of
      $Y$ is glued to unshaded back face of $B$ so that the vertices 0
      and 1 coincide.}
    \label{fig: twist2}
\end{figure}
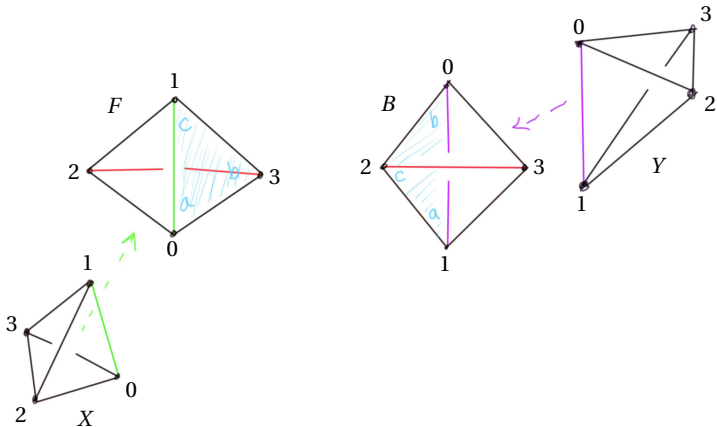

For both possible twisted cases, the particular sequence of $2\to3$
and $3\to2$ moves corresponding to factorization of the
endpoint-through-endpoint move in Corollary \ref{cor: end to end} was
found by searching possible sequences, using the valences of various
edges as a guide. More details on the sequence can be found in the
implementation, see the file \texttt{mcomplex\_with\_expansion.py} in
\cite{CodeAndData}. The factorization of the $2\to0$ described at the
end of Section \ref{sec: 2 to 0} then follows.


\section{Building the initial diagram}
\label{sec: embed}

\begin{figure}
  \centering
  \begin{tikzpicture}[nmdstd]
    \node[above right] at (0.65, 0.85) {\includegraphics[height=5.36cm]{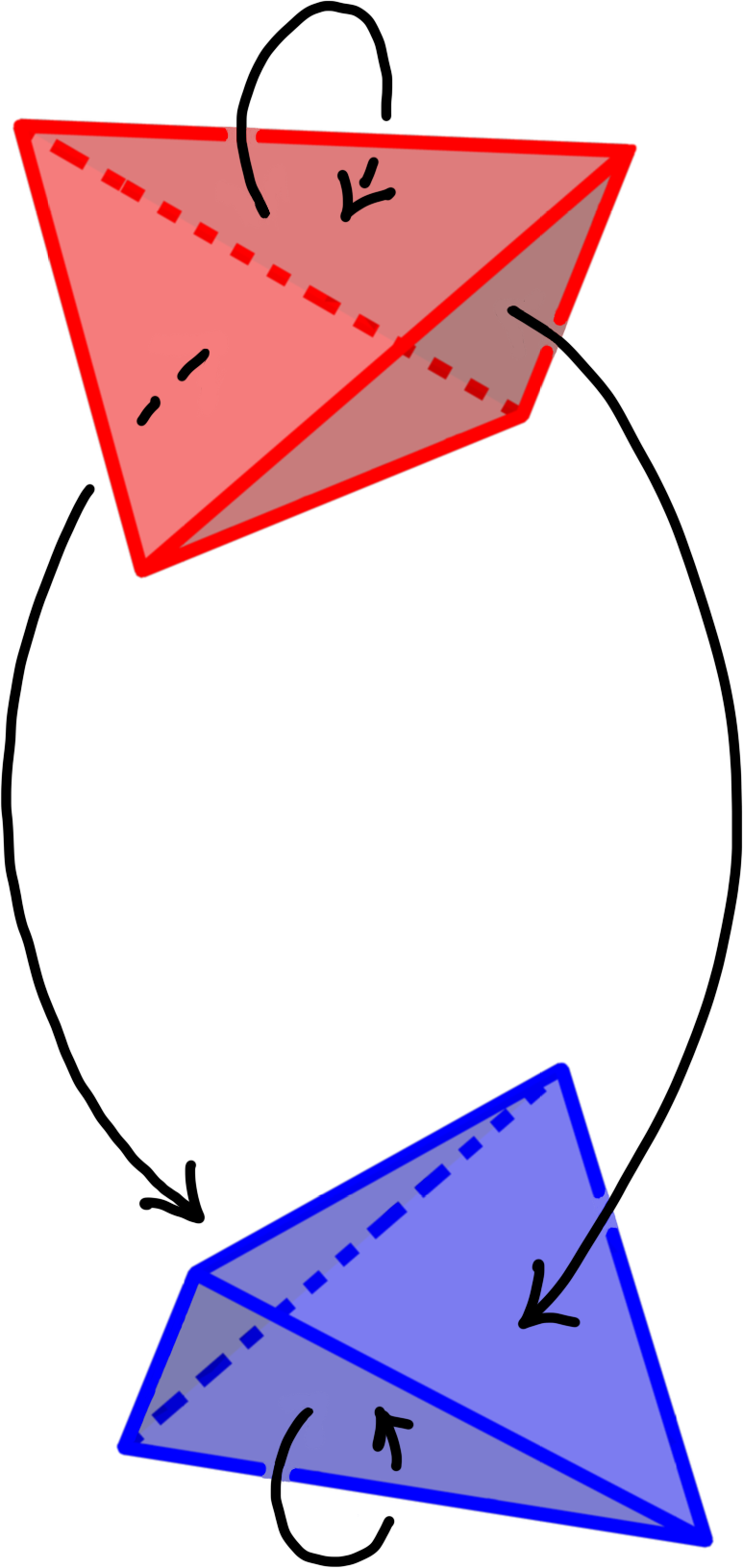}};
    \node at (1.15, 2) {$b_2$};
    \node at (1.1, 1.2) {$b_1$};
    \node at (3.4, 0.9) {$b_3$};
    \node at (2.8, 2.92) {$b_0$};

    \node at (0.6, 5.95) {$a_3$};
    \node at (3.2, 5.85) {$a_2$};
    \node at (1.25, 4.15) {$a_0$};
    \node at (2.6, 4.7) {$a_1$};

    \node at (1.9, 6.7) {$a_{023} \to a_{123}$};
    \node at (2.1, 0.5) {$b_{213} \to b_{013}$};

    \node at (0.4, 3.5) {\begin{tikzcd} a_{031} \arrow[d] \\  b_{021} \end{tikzcd}};
    \node at (3.8, 3.5) {\begin{tikzcd} a_{012} \arrow[d] \\  b_{032} \end{tikzcd}};
  \end{tikzpicture}

  \caption{The base triangulation $\cT_0$ in $\R^3$.}
  \label{fig: base tri}
  
\end{figure}

\begin{figure}
  \centering
  \begin{tikzpicture}[nmdstd]
    \node[above right] at (2.4, 0) {\includegraphics[height=7cm]{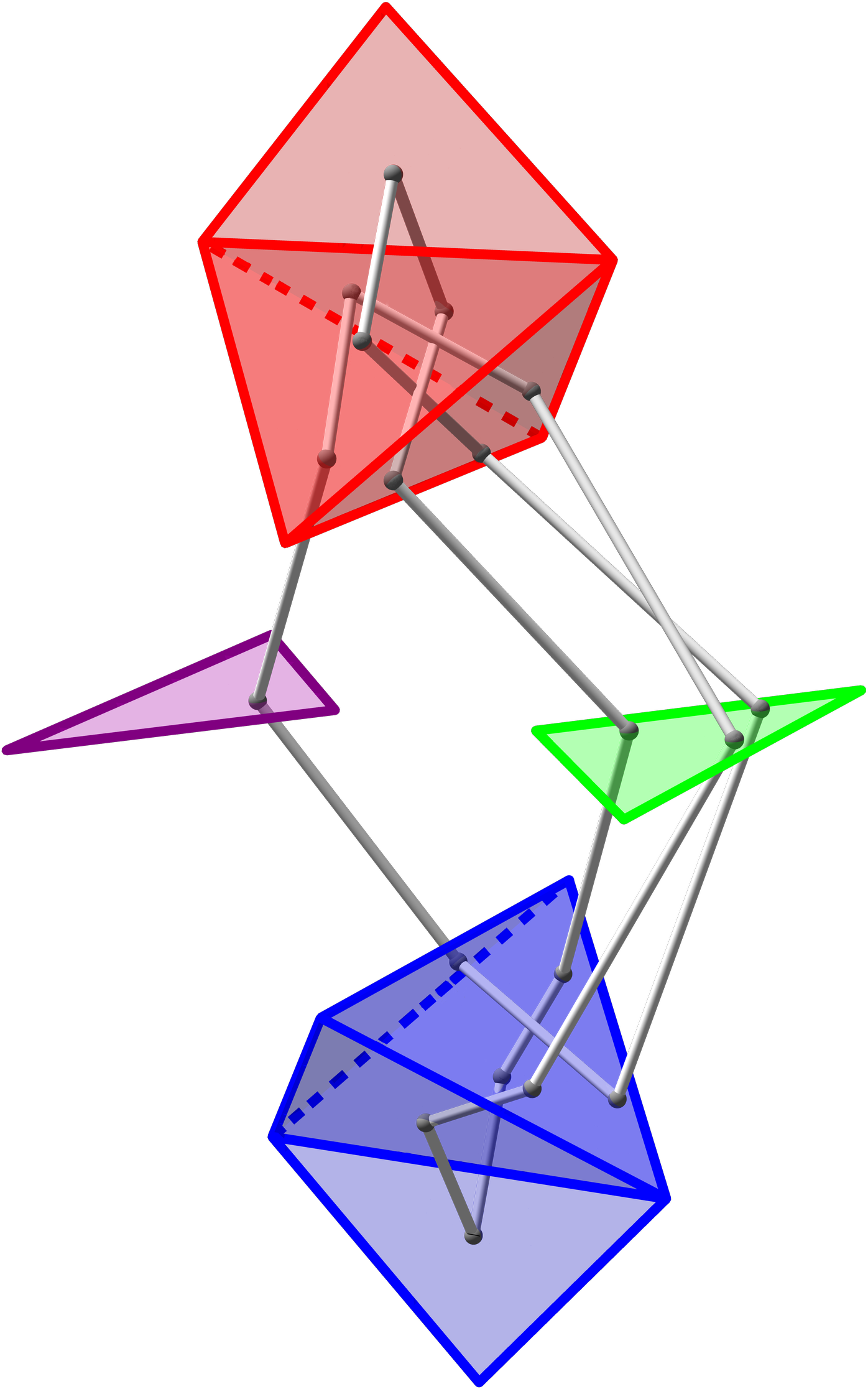}};
    \node[above right] at (9.3, 0) {\includegraphics[height=7cm]{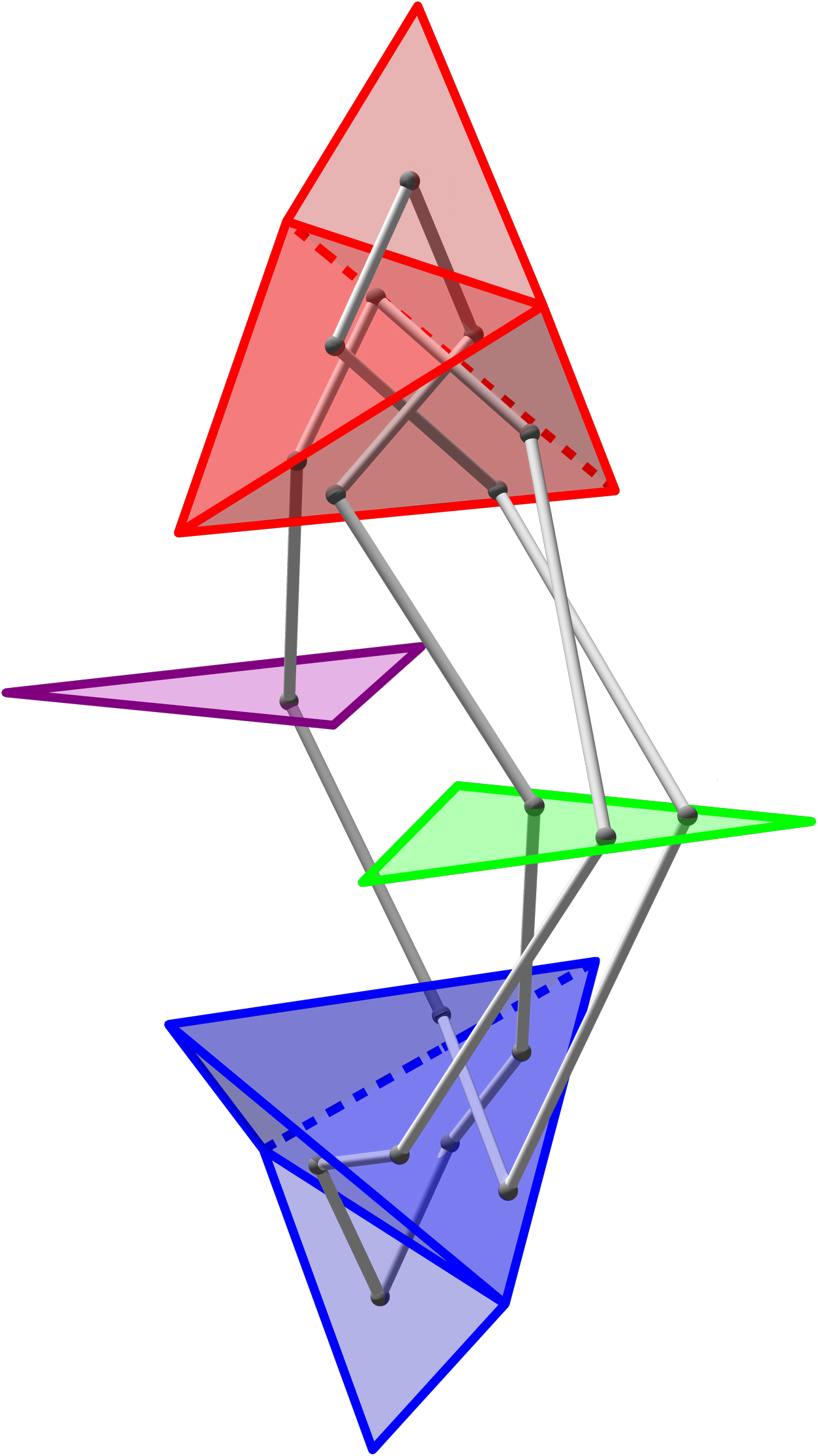}};
  \end{tikzpicture}

  \caption{Two views of the same link realized by the base triangulation
    $\cT_0$ using the fins and lenses shown.}
  \label{fig: fins and lenses}
\end{figure}

The base triangulation $\cT_0$ of $S^3$ has two tetrahedra and one
vertex and is shown in Figure~\ref{fig: base tri}; its isomorphism
signature in the sense of \cite[\S~3.2]{Burton2011}, which completely
determines the triangulation, is \texttt{cMcabbgdv}. We next give the
method for obtaining a planar diagram $D$ for a barycentric link $L$
in $\cT_0$.  We first build a PL link in $\R^3$ representing $L$ and
then project it onto a plane to get $D$.

We cut open $\cT_0$ along its faces and embed the resulting pair of
tetrahedra in $\R^3$ as shown in Figure~\ref{fig: base tri}.  This
cuts open the link $L$ along its intersections with the faces of
$\cT_0$, resulting in a collection of curves in $\R^3$ inside the two
tetrahedra.  To reconnect these curves and recover $L$, we use
\emph{fins} and \emph{lenses} as shown in Figure \ref{fig: fins and
  lenses} to interpolate between pairs of faces that are identified in
$\cT_0$.  There are two triangular fins, one attached vertically to
each tetrahedron, with each fin corresponding to one of the two
valence-1 edges of $\cT_0$.  The gluing of two faces incident to a
valence-1 edge is realized by folding them onto the corresponding
fin. Thus for each barycentric arc that ends in a face corresponding
to a fin, we add the line segment joining this endpoint of the arc to
the corresponding point in the fin.

The two triangular lenses lie between the two tetrahedra in a
horizontal plane. There is an affine map taking the corresponding face
in the top tetrahedron to its lens and a second affine map taking the
lens to the corresponding face in the bottom tetrahedron, arranged so
their composition is the face pairing in $\cT_0$. For every arc in the
top tetrahedron ending on a face corresponding to a lens, we add the
line segment between the endpoint and its image under the affine map
to the lens.  For each such segment that terminates on a lens, we add
the line segment from this endpoint to its image in the face of the
bottom tetrahedron under the affine map. This results in a PL link in
$\R^3 \subset S^3$ that must be isotopic to $L$: just imagine puffing
out the two tetrahedra to fill all of $S^3$ following the guides given
by the fins and lenses.


Given a collection of line segments in $\R^3$ corresponding to the
link $L$, we can build a diagram for $L$ by projecting the line
segments onto a plane, computing the crossing information, and
assembling this into a planar diagram.  Our default choice is roughly
to project onto the plane of the page in Figure~\ref{fig: fins and
  lenses}, with the (so far unused) fall-back of a small random matrix
in $\SL{3}{\Z}$ if a general-position failure occurs.  The link
diagrams resulting from this process have many more crossings than is
necessary, and we deal with this in Section \ref{sec: simp link}.
Still, the specific configuration of fins, lenses, and projection was
chosen to try minimize the number of crossings created at this stage;
our initial approach used a more compact embedding where the
tetrahedra shared a face, and this produced much larger diagrams.


\section{Simplifying link diagrams}
\label{sec: simp link}

\begin{figure}
  \centering
  \includegraphics[width=0.85\textwidth]{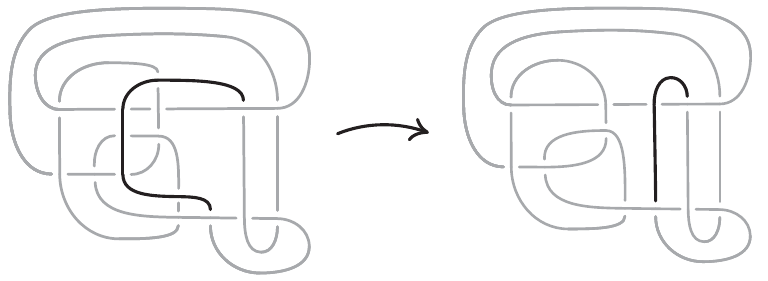}

  \caption{An example of the strand pickup method for diagram
    simplification.  At left, an \emph{overstrand}, which runs over
    each crossing it partipates in, is indicated by the darker line.
    At right is the result of isotoping the overstrand, fixing its
    endpoints, to get a diagram with fewer crossings.  The best
    possible location for an overstrand can be found by solving a
    weighted shortest-path problem in the planar dual graph to the
    original diagram.  }
  \label{fig: pickup}

\end{figure}

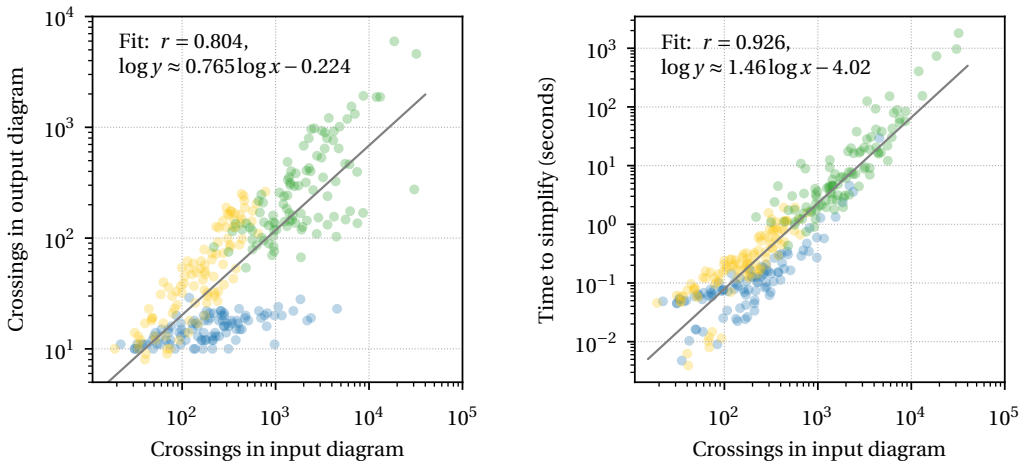
\begin{figure}
  \begin{tikzpicture}
    \node[above right] at (0, 0) {
      \begin{tikzoverlay}[height=6.3cm]{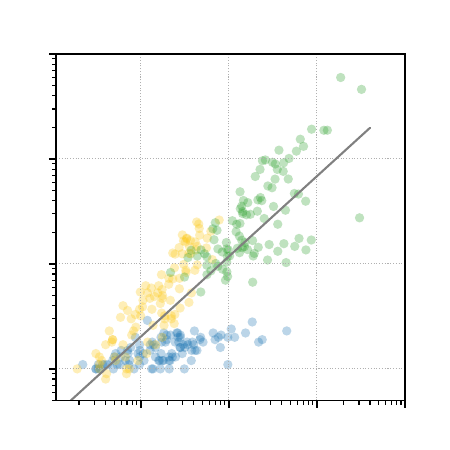}
        \draw (16.375000, 86.150000) node[below right, text width=3.4cm]
              {Fit: $r = 0.804$, \\ $\log y \approx 0.765 \log x - 0.224$ };
        \draw (51.250000, 0.586667) node[below] {Crossings in input diagram};
        \draw (31.267223, 7.759259) node[below] {${10^{2}}$};
        \draw (50.844815, 7.759259) node[below] {${10^{3}}$};
        \draw (70.422408, 7.759259) node[below] {${10^{4}}$};
        \draw (90.000000, 7.759259) node[below] {${10^{5}}$};
        \draw (-2, 49.500000) node[rotate=90.0, anchor=base]
              {Crossings in output diagram};
        \draw (10.8, 18.021842) node[left] {${10^{1}}$};
        \draw (10.8, 41.347894) node[left] {${10^{2}}$};
        \draw (10.8, 64.673947) node[left] {${10^{3}}$};
        \draw (10.8, 88.000000) node[left] {${10^{4}}$};
      \end{tikzoverlay}
    };
    \node[above right] at (7.0, 0) {
      \begin{tikzoverlay}[height=6.3cm, trim={0 0 0.7cm 0}, clip]{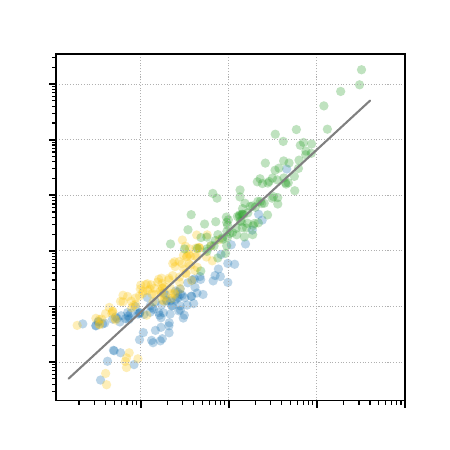}
        \begin{scope}[x=0.063cm, y=0.063cm] 
          \draw (16.375000, 86.150000) node[below right, text width=3.2cm]
                {Fit: $r = 0.926$, \\ $\log y \approx 1.46 \log x - 4.02$};
          \draw (51.250000, 0.586667) node[below] {Crossings in input diagram};
          \draw (31.267223, 7.759259) node[below] {${10^{2}}$};
          \draw (50.844815, 7.759259) node[below] {${10^{3}}$};
          \draw (70.422408, 7.759259) node[below] {${10^{4}}$};
          \draw (90.000000, 7.759259) node[below] {${10^{5}}$};
          \draw (-4.5, 49.500000) node[rotate=90.0, anchor=base] {Time to simplify (seconds)};
          \draw (10.8, 19.533742) node[left] {${10^{-2}}$};
          \draw (10.8, 31.890550) node[left] {${10^{-1}}$};
          \draw (10.8, 44.247359) node[left] {${10^{0}}$};
          \draw (10.8, 56.604167) node[left] {${10^{1}}$};
          \draw (10.8, 68.960975) node[left] {${10^{2}}$};
          \draw (10.8, 81.317783) node[left] {${10^{3}}$};
        \end{scope}
      \end{tikzoverlay}
    };
  \end{tikzpicture}

  \caption{Simplifying 300 diagrams with between 19 and 32,095
    crossings, drawn from Sections~\ref{sec: impl} and~\ref{sec:
      zero}. The dramatic amount of simplification is shown at left,
    with an $n$-crossing knot turned into one with $O(n^{0.8})$
    crossings.  The running time shown on the right is roughly $O(n^{1.5})$. }
  \label{fig: simp in action}
\end{figure}

We now sketch how we simplified the initial link diagram constructed
in Section~\ref{sec: embed}, which sometimes had 10,000--100,000
crossings, to produce the final output of our algorithm for
\FindDiagram.  Previous computational work focused on simplifying
diagrams with 20 or fewer crossings
\cite{HosteThistlethwaiteWeeks1998, Burton2020}. In that regime,
random Reidemeister moves combined with flypes are extremely effective
in reducing the number of crossings.  However, these techniques alone
proved inadequate for our much larger links.  Instead, we used the
more global \emph{strand pickup} method of Figure~\ref{fig: pickup}.
This technique was introduced by the second author and included in
SnapPy \cite{SnapPy} since version 2.3 (2015), but not previously
documented in the literature.  It has similarities with the arc
representation/grid diagram approach of \cite{Dynnikov1999,
  Dynnikov2006, DynnikovSokolova2021}, but it works with arbitrary
planar diagrams.  When applying the pickup move, we start with the
longest overstrands and work towards the shorter ones if no
improvement is made.  When a pickup move succeeds, we do more basic
simplications before looking for another pickup move.  We also do the
same move on understrands, going back and forth between the two sides
until the diagram stabilizes; for details, see \cite{ObeidinSimplify}.
The high amount of simplification and sub-quadratic running time are
shown in Figure~\ref{fig: simp in action}.  As further evidence of
its utility, we note that it strictly monotonically reduces the unknot
diagrams $D_{28}$, $D_{43}$, and $PZ_{78}$ in \cite{BurtonEtAl2021} to
the trivial diagram; in constrast, these require adding at least three
crossings if one uses only Reidemeister moves.


\section{A lower bound on computational complexity}
\label{sec: torus}

In this section, we show that the worst-case complexity of
\FindDiagram\ is at least exponential by exhibiting inputs where any
output must be exponentially large. Specifically, let $F_n$ be the
Fibonacci sequence and $\alpha = (1 + \sqrt{5})/2 \approx 1.618$ be
the golden ratio.  Let $K_n$ be the torus knot $T(F_n,
F_{n-1})$. Then:

\begin{theorem}
  \label{thm: fib}
  There is an ideal triangulation of $E(K_n)$ with $O(n)$ tetrahedra,
  but the minimum crossing number of $K_n$ is
  $F_{n-1}(F_n - 1) \sim \alpha^{2 n - 1}/5$.
\end{theorem}
We now give the short proof, which is similar to the analysis of the
complexity of a normal meridional disk in a minimal triangulation of a
layered solid torus in \cite[Section 5]{JacoRubinstein2006}.

\begin{proof}

Set $p = F_n$ and $q = F_{n-1}$. The claim about the crossing number
of $K_n$ is Proposition~7.5 of \cite{Murasugi1991}.  For the ideal
triangulation of $E(K_n)$, first we construct a 1-vertex triangulation
$\cT$ of $S^3$ where $K_n$ is an edge. We do this by gluing two
layered solid tori (see Section~\ref{sec: layered tri}) to form the
genus-1 Heegaard splitting of $S^3$ so that the torus knot $K_n$ is an
edge on their common interface. Since the partial quotients of the
continued fraction expansion of $p/q$ are all 1, this requires only
$O(n)$ tetrahedra, see \cite{JacoSedgwick2003}.  Now, take the second
barycentric subdivision of $\cT$ and delete the interior of the link
of $K_n$ to get a triangulation $\cT'$ of the compact manifold
$E(K_n)$ with $O(n)$ tetrahedra.  Finally crush as in
\cite[Theorem~7.1]{JacoRubinstein2003} to get an ideal triangulation
$\Tdot$ for $E(K_n)$ with no more tetrahedra than $\cT'$.
\end{proof}
\


\section{Implementation and initial experiments}
\label{sec: impl}

We implemented our algorithm in Python, building on the pure-Python
\texttt{t3mlite} library for \3-manifold triangulations that is part
of SnapPy \cite{SnapPy}.  We also used SnapPy's C kernel to produce
the layered filled triangulation $\cT$ of Section~\ref{sec: dehn} from
the input ideal triangulation $\Tdot$.  The needed linear algebra over
$\Q$ was handled by PARI \cite{PARI2}. Not including these libraries,
our implementation consists of 1,800 lines of Python code.  We had to
put considerable effort into optimization to handle examples as large
as that shown in Figure~\ref{fig: cong 2}. Our code and data is
archived at \cite{CodeAndData} and has been incorporated into version
3.1 of SnapPy \cite{SnapPy}.

To validate our implementation, we applied it to two sample sets, one
where the inputs were small and one where the best-possible outputs
were small.  The first, $\CK$, is the 1,267 hyperbolic knots whose
exteriors have ideal triangulations with at most 9 tetrahedra
\cite{Dunfield2020, BakerKegel2021}.  The second, $\SK$, consists of
1,000 knots with minimal crossing numbers between 10 and 19.
Specifically, $\SK$ has 100 knots for each crossing number in that
range, which were selected at random from all the hyperbolic
nonalternating knots with that crossing number \cite{Burton2020}; the
exception is that there are only 41 such 10-crossing knots, so 59
alternating 10-crossing knots were used as well.  (Alternating knots
have unusually close connections between their diagrams and exteriors,
so were excluded as possibly being an easy case for \FindDiagram.)

Our program found diagrams for all 2,267 of these exteriors.  The
running time was under 20 seconds for 96.7\% of them, with a max of
2.5 minutes (CPUs were Intel Xeon E5-2690 v3 at 2.6GHz with 4GB of
memory per core, circa 2014); see Figure~\ref{fig: running}.  The
input ideal triangulations $\Tdot$ had between 2 and 44 tetrahedra,
and the resulting layered filling triangulation $\cT$ had between 13
and 77 tetrahedra (mean of 31.5), typically 60\% larger than $\Tdot$;
see Figure~\ref{fig: tets}.  The sequence of simple Pachner moves used
to reduce $\cT$ to $\cT_0$ had length between 39 and 761 (mean of
241.0), see Figure~\ref{fig: expanded moves}; this was typically 7.5
times longer than the initial sequence of Pachner moves that included
$2 \to 0$ moves (Figure~\ref{fig: expansion factor}).  For the knots
in $\SK$, we compare the size of the output diagram to the minimal
crossing number in Figure~\ref{fig: num cross}; the output matched the
crossing number for 42.1\% of these exteriors, and it was within 3 for
87.8\%.  For $\CK$, the maximal number of crossings in the output was
303, with mean output crossing number 65.9, and median output crossing
number 40.

\begin{figure}[p]
  \centering
  \begin{tikzoverlay}[width=.9\textwidth]{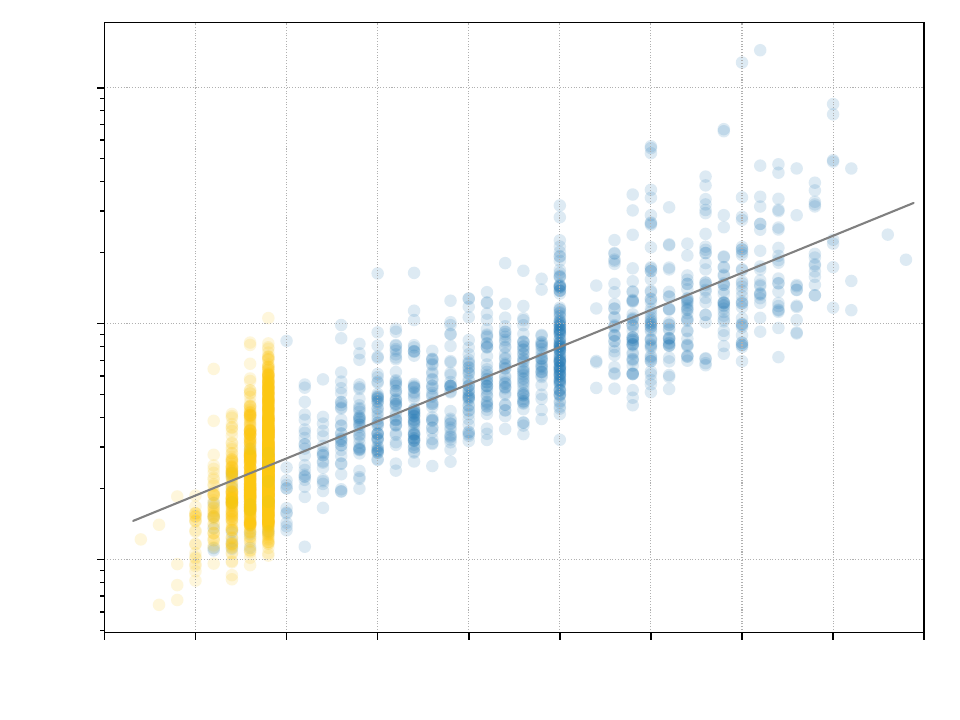}
    \draw (15.143327, 69.478733) node[below right]
      {Fit: $\log y \approx 0.031 x + 0.114$, $r = 0.850$};
    \draw (53.571529, 3.326270) node[below] {Ideal tetrahedra in $\Tdot$};
    \draw (10.873526, 7.586806) node[below] {$0$};
    \draw (20.361971, 7.586806) node[below] {$5$};
    \draw (29.850416, 7.586806) node[below] {$10$};
    \draw (39.338861, 7.586806) node[below] {$15$};
    \draw (48.827306, 7.586806) node[below] {$20$};
    \draw (58.315751, 7.586806) node[below] {$25$};
    \draw (67.804196, 7.586806) node[below] {$30$};
    \draw (77.292641, 7.586806) node[below] {$35$};
    \draw (86.781086, 7.586806) node[below] {$40$};
    \draw (96.269531, 7.586806) node[below] {$45$};
    \draw (2.283626, 40.881076) node[rotate=90.0] {Running time (seconds)};
    \draw (9.354429, 16.712146) node[left] {${10^{0}}$};
    \draw (9.354429, 41.288829) node[left] {${10^{1}}$};
    \draw (9.354429, 65.865511) node[left] {${10^{2}}$};
    \SKCKlegend{(89, 12)}
  \end{tikzoverlay}

  \caption{Mean running time for the 2,267 knot exteriors in $\SK$ and
    $\CK$ appears exponential with small base, roughly
    $O(1.07^n)$. Compare Figure~\ref{fig: num arcs} on the growth of
    the number of arcs in $\cT_0$.  }
  \label{fig: running}
\end{figure}

\begin{figure}[p]
  \centering
  \begin{tikzoverlay}[width=.9\textwidth]{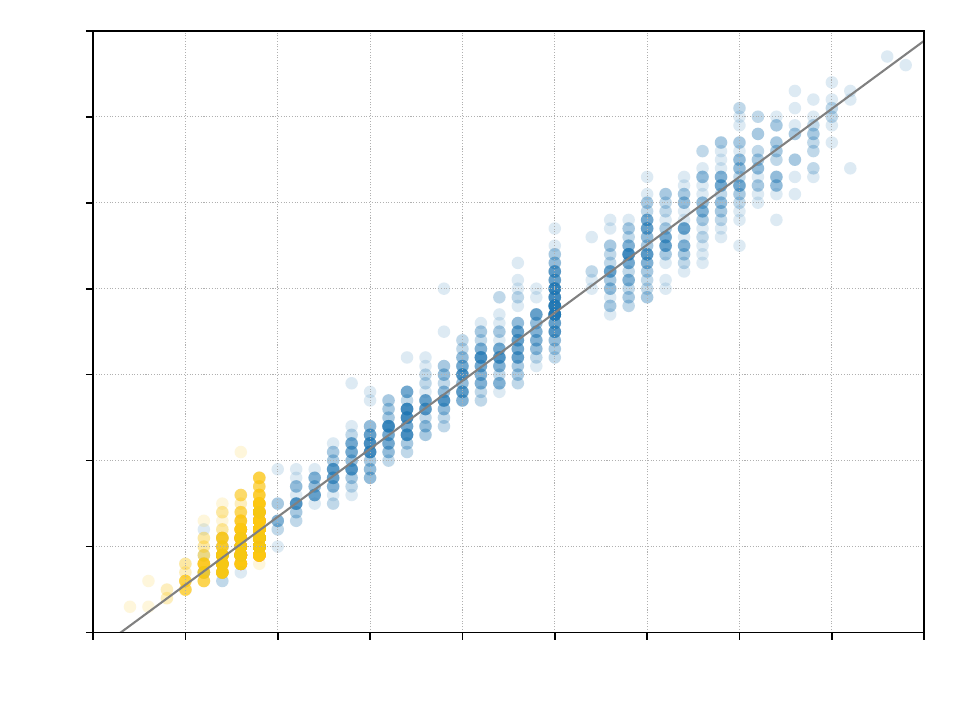}
    \draw (14.020725, 68.662326) node[below right]
        {Fit: $y \approx 1.58 x + 7.64$, $r = 0.989$};
     \draw (52.980686, 3.326270) node[below] {Ideal tetrahedra in $\Tdot$};
     \draw (9.691840, 7.586806) node[below] {$0$};
     \draw (19.311584, 7.586806) node[below] {$5$};
     \draw (28.931327, 7.586806) node[below] {$10$};
     \draw (38.551071, 7.586806) node[below] {$15$};
     \draw (48.170814, 7.586806) node[below] {$20$};
     \draw (57.790557, 7.586806) node[below] {$25$};
     \draw (67.410301, 7.586806) node[below] {$30$};
     \draw (77.030044, 7.586806) node[below] {$35$};
     \draw (86.649788, 7.586806) node[below] {$40$};
     \draw (96.269531, 7.586806) node[below] {$45$};
     \draw (2.272070, 40.451389) node[rotate=90.0, anchor=base]
           {Tetrahedra in layered filled $\cT$};
     \draw (8.172743, 9.105903) node[left] {$10$};
     \draw (8.172743, 18.061756) node[left] {$20$};
     \draw (8.172743, 27.017609) node[left] {$30$};
     \draw (8.172743, 35.973462) node[left] {$40$};
     \draw (8.172743, 44.929315) node[left] {$50$};
     \draw (8.172743, 53.885169) node[left] {$60$};
     \draw (8.172743, 62.841022) node[left] {$70$};
     \draw (8.172743, 71.796875) node[left] {$80$};
     \SKCKlegend{(89, 12)}
     \begin{scope}[shift={(9.69184028, 0.15004960)},
       xscale=1.92394869, yscale=0.89558532]
     \end{scope}
   \end{tikzoverlay}

  \caption{The number of tetrahedra in the layered filled $\cT$
    compared to the input ideal $\Tdot$.}
  \label{fig: tets}
\end{figure}

\begin{figure}[p]
  \centering
  \begin{tikzoverlay}[width=.9\textwidth]{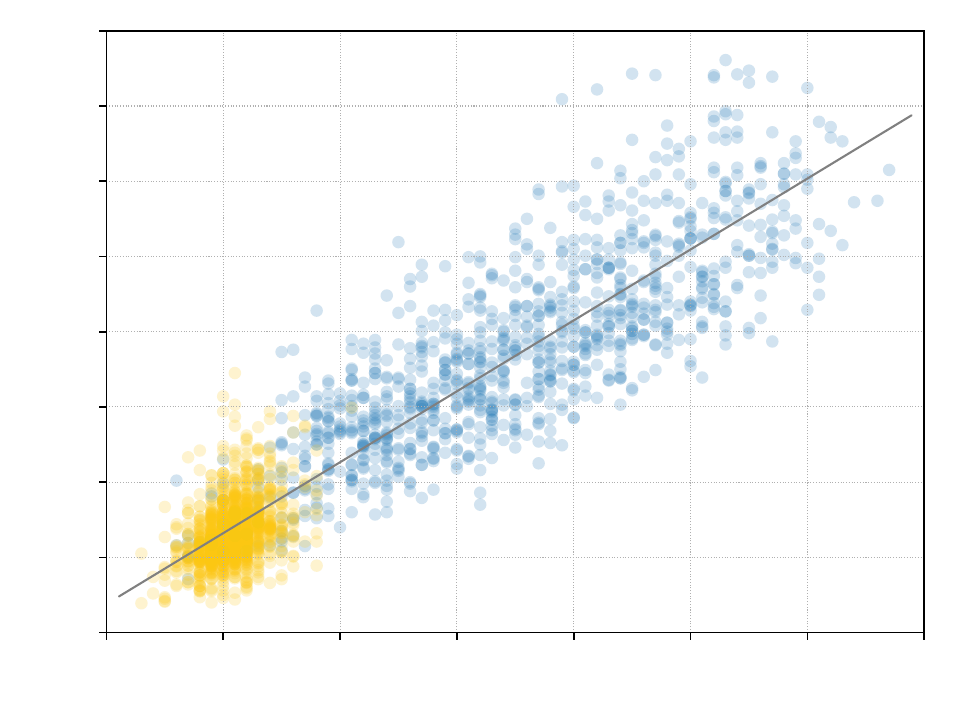}
    \draw (15.338108, 68.662326) node[below right]
          {Fit: $y \approx 9.43 x - 56.4$, $r = 0.923$};
    \draw (53.674045, 3.326270) node[below] {Tetrahedra in $\cT$};
    \draw (11.078559, 7.586806) node[below] {$10$};
    \draw (23.248698, 7.586806) node[below] {$20$};
    \draw (35.418837, 7.586806) node[below] {$30$};
    \draw (47.588976, 7.586806) node[below] {$40$};
    \draw (59.759115, 7.586806) node[below] {$50$};
    \draw (71.929253, 7.586806) node[below] {$60$};
    \draw (84.099392, 7.586806) node[below] {$70$};
    \draw (96.269531, 7.586806) node[below] {$80$};
    \draw (2.0, 40.451389) node[rotate=90.0, anchor=base]
        {Simple Pachner moves};
    \draw (9.559462, 9.105903) node[left] {$0$};
    \draw (9.559462, 16.942274) node[left] {$100$};
    \draw (9.559462, 24.778646) node[left] {$200$};
    \draw (9.559462, 32.615017) node[left] {$300$};
    \draw (9.559462, 40.451389) node[left] {$400$};
    \draw (9.559462, 48.287760) node[left] {$500$};
    \draw (9.559462, 56.124132) node[left] {$600$};
    \draw (9.559462, 63.960503) node[left] {$700$};
    \draw (9.559462, 71.796875) node[left] {$800$};
    \SKCKlegend{(89, 11.5)}
    \begin{scope}[shift={(-1.09157986, 9.10590278)},
      xscale=1.21701389, yscale=0.07836372]
    \end{scope}
  \end{tikzoverlay}

  \caption{The number of \emph{simple} Pachner moves used to transform
    the layered filled triangulation $\cT$ into the base triangulation
    $\cT_0$ is generically linear in the size of $\cT$.}
  \label{fig: expanded moves}
\end{figure}

\begin{figure}[p]
  \centering
  \begin{tikzoverlay}[width=.9\textwidth]{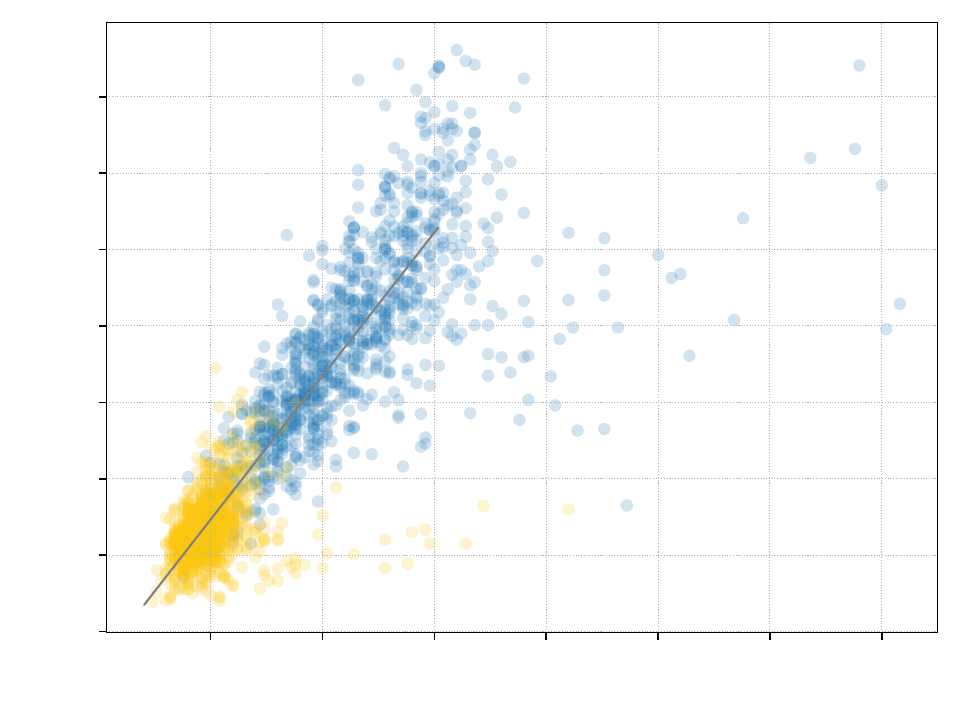}
    \draw (97, 10) node[above left]
          {Fit: $y \approx 7.52 x - 41.4$, $r = 0.896$};
    \draw (54.367405, 3.326270) node[below]
          {General Pachner moves from $\cT$ to $\cT_0$};
    \draw (21.940273, 7.586806) node[below] {$25$};
    \draw (33.592972, 7.586806) node[below] {$50$};
    \draw (45.245672, 7.586806) node[below] {$75$};
    \draw (56.898371, 7.586806) node[below] {$100$};
    \draw (68.551070, 7.586806) node[below] {$125$};
    \draw (80.203769, 7.586806) node[below] {$150$};
    \draw (91.856469, 7.586806) node[below] {$175$};
    \draw (1.0, 40.881076) node[rotate=90.0] {Simple Pachner moves};
    \draw (9.559462, 9.220915) node[left] {$0$};
    \draw (9.559462, 17.177115) node[left] {$100$};
    \draw (9.559462, 25.133315) node[left] {$200$};
    \draw (9.559462, 33.089515) node[left] {$300$};
    \draw (9.559462, 41.045716) node[left] {$400$};
    \draw (9.559462, 49.001916) node[left] {$500$};
    \draw (9.559462, 56.958116) node[left] {$600$};
    \draw (9.559462, 64.914316) node[left] {$700$};
    \SKCKlegend{(90, 16)}
    \begin{scope}[shift={(10.28757380, 9.22091522)},
      xscale=0.46610797, yscale=0.07956200]
    \end{scope}
\end{tikzoverlay}

  \caption{This plot shows the increase in the number of Pachner moves
    when we factor the $2 \to 0$ moves into simple Pachner moves.  The
    regression line is based on points with $x < 75$.}
  \label{fig: expansion factor}
\end{figure}

\begin{figure}[p]
  \centering
  \begin{tikzoverlay}[width=.9\textwidth]{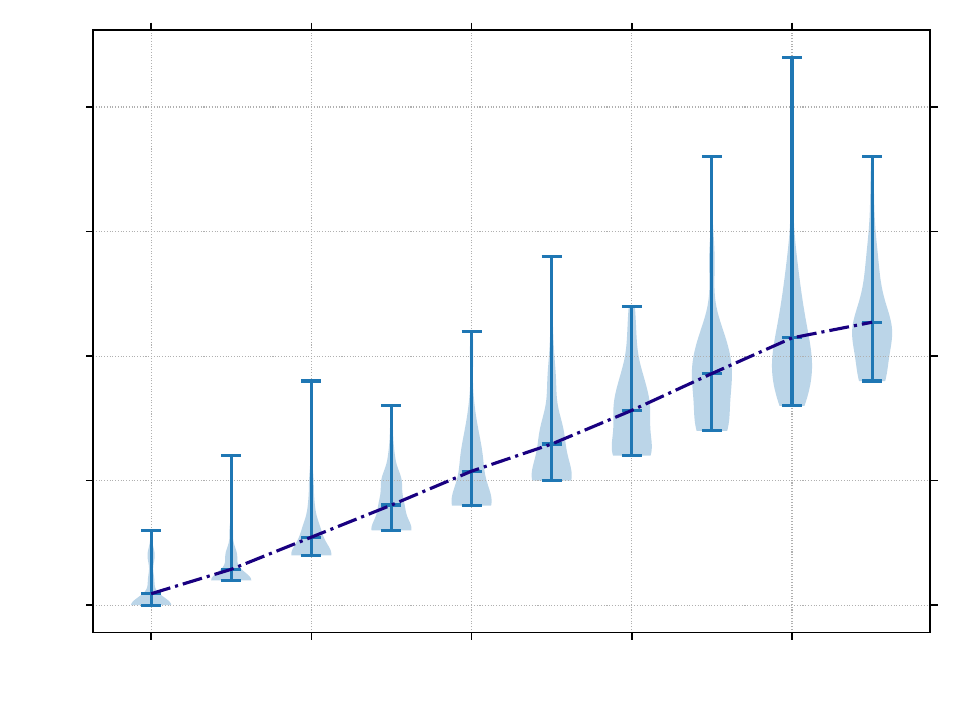}
    \draw (53.294271, 3.326270) node[below] {Minimal crossing number};
    \draw (15.741938, 7.586806) node[below] {$10$};
    \draw (32.431864, 7.586806) node[below] {$12$};
    \draw (49.121789, 7.586806) node[below] {$14$};
    \draw (65.811715, 7.586806) node[below] {$16$};
    \draw (82.501641, 7.586806) node[below] {$18$};
    \draw (1.272070, 40.501302) node[rotate=90.0, anchor=base] {Crossings in output};
    \draw (8.172743, 11.960030) node[left] {$10$};
    \draw (8.172743, 24.933335) node[left] {$15$};
    \draw (8.172743, 37.906641) node[left] {$20$};
    \draw (8.172743, 50.879946) node[left] {$25$};
    \draw (8.172743, 63.853252) node[left] {$30$};
    \begin{scope}[shift={(-67.70768956, -13.98658101)},
      xscale=8.34496279, yscale=2.59466110]
    \end{scope}
  \end{tikzoverlay}

  \caption{For the knots in $\SK$, grouped by minimum crossing number,
    the number of crossings in the diagram output by our program.  The
    dotted line indicates the mean. }
  \label{fig: num cross}
\end{figure}

\begin{figure}[p]
  \centering
  \begin{tikzoverlay}[width=.9\textwidth]{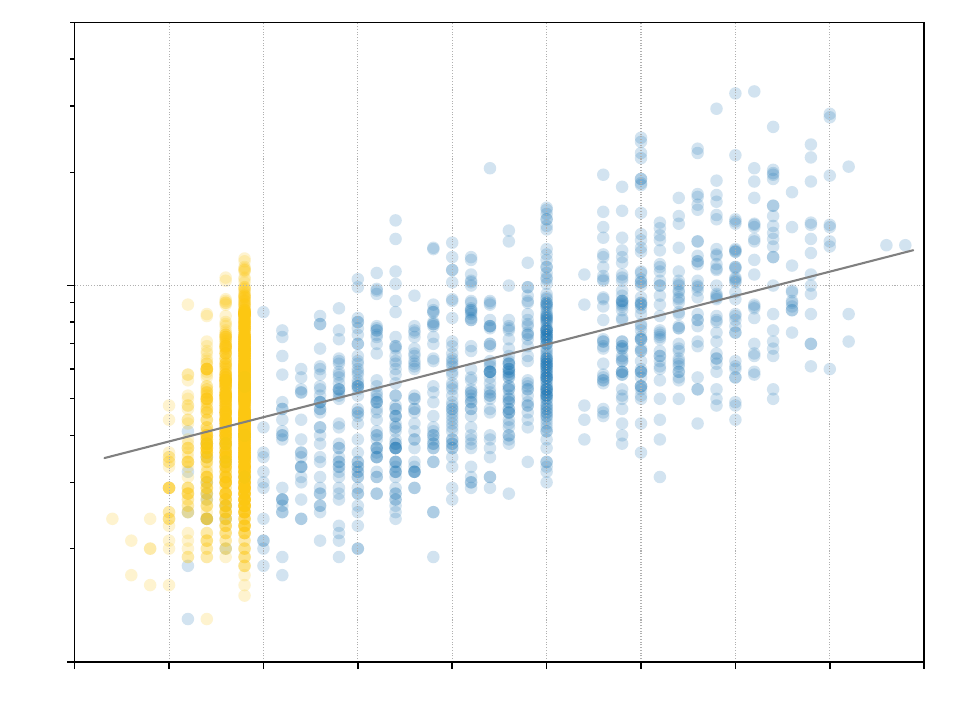}
    \draw (12.201120, 69.325955) node[below right]
          {Fit: $\log y \approx 0.013 x + 1.52$, $r = 0.573$};
    \draw (52.033985, 1.126270) node[below] {Ideal tetrahedra in $\Tdot$};
    \draw (7.776467, 4.531250) node[below] {$0$};
    \draw (17.609030, 4.531250) node[below] {$5$};
    \draw (27.441592, 4.531250) node[below] {$10$};
    \draw (37.274155, 4.531250) node[below] {$15$};
    \draw (47.106718, 4.531250) node[below] {$20$};
    \draw (56.939280, 4.531250) node[below] {$25$};
    \draw (66.771843, 4.531250) node[below] {$30$};
    \draw (76.604406, 4.531250) node[below] {$35$};
    \draw (86.436969, 4.531250) node[below] {$40$};
    \draw (96.269531, 4.531250) node[below] {$45$};
    \draw (0, 38.743924) node[rotate=90.0, anchor=base] {Number of arcs in $\cT_0$};
    \draw (6.257370, 6.050347) node[left] {${10^{1}}$};
    \draw (6.257370, 45.254043) node[left] {${10^{2}}$};
    \draw (6.257370, 72.81) node[left] {${5 \times 10^{2}}$};
  \end{tikzoverlay}

  \caption{The number of barycentric arcs when we arrive at $\cT_0$
    appears exponential in the size of the input $\Tdot$, roughly
    $O(1.03^n)$.}
  \label{fig: num arcs}

\end{figure}


\section{Applications}
\label{sec: apps}

\subsection{Congruence links}
\label{sec: cong}

Powerful tools from number theory apply to the special class of
arithmetic hyperbolic \3-manifolds. Thurston asked which link
exteriors are in the subclass of principal congruence arithmetic
manifolds; this was resolved in \cite{BakerGoernerReid2019a}: there
are exactly 48 such exteriors.  These 48 have hyperbolic volumes in
[5.33348, 1365.37] and ideal triangulations with between 6 and 1,526
tetrahedra.  Link diagrams for 15 of these 48 had previously been
found by ad~hoc methods \cite{BakerGoernerReid2019b}. Our program has
found diagrams for 23 more, including Figures~\ref{fig: cong 1} and
\ref{fig: cong 2}; collectively, we now have links for the 38
such exteriors of smallest volume, see Figure~\ref{fig: cong plot}.

\begin{figure}[bth]
  \centering
  \begin{tikzpicture}
    \node[above right] at (0, 0) {%
      \begin{tikzoverlay}[height=6.3cm]{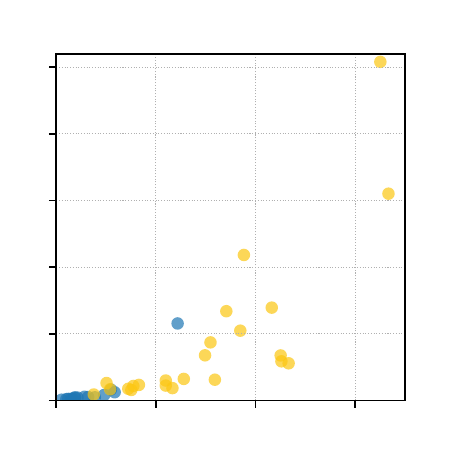}
        \draw (51.250000, 1.029375) node[below] {Hyperbolic volume};
        \draw (12.500000, 7.759259) node[below] {$0$};
        \draw (34.642857, 7.759259) node[below] {$100$};
        \draw (56.785714, 7.759259) node[below] {$200$};
        \draw (78.928571, 7.759259) node[below] {$300$};
        \draw (-5.145833, 49.500000) node[rotate=90.0, anchor=base] {Crossings};
        \draw (10.8, 11.000000) node[left] {$0$};
        \draw (10.8, 25.807692) node[left] {$500$};
        \draw (10.8, 40.615385) node[left] {$1000$};
        \draw (10.8, 55.423077) node[left] {$1500$};
        \draw (10.8, 70.230769) node[left] {$2000$};
        \draw (10.8, 85.038462) node[left] {$2500$};
      \end{tikzoverlay}
    };
    \node[above right] at (6.8, 0) {%
      \begin{tikzoverlay}[height=6.3cm, trim={0 0 0.7cm 0}, clip]{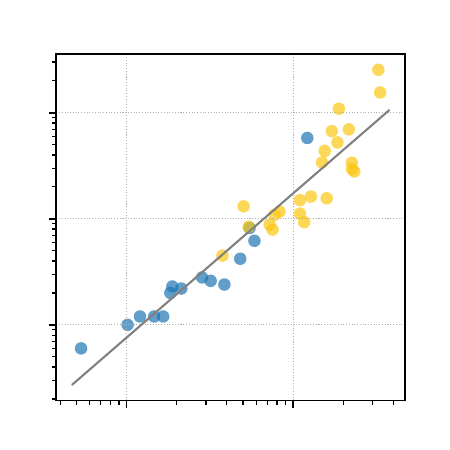}
        \begin{scope}[x=0.063cm, y=0.063cm] 
          \draw (14.375000, 86.150000) node[below right, text width=3.4cm]
                {Fit: $r = 0.946$, \\ $\log y \approx 1.356 \log x - 0.480$};
          \draw (51.250000, 0.586667) node[below] {Hyperbolic volume};
          \draw (28.119686, 7.759259) node[below] {${10^{1}}$};
          \draw (65.128425, 7.759259) node[below] {${10^{2}}$};
          \draw (-2, 49.500000) node[rotate=90.0, anchor=base] {Crossings};
          \draw (10.8, 27.808054) node[left] {${10^{1}}$};
          \draw (10.8, 51.380543) node[left] {${10^{2}}$};
          \draw (10.8, 74.953032) node[left] {${10^{3}}$};
        \end{scope}
      \end{tikzoverlay}
    };
  \end{tikzpicture}

  \caption{The 38 known link diagrams whose exteriors are principal
    congruence arithmetic; blue are the 15 from
    \cite{BakerGoernerReid2019b}, yellow are new.  The plots are the
    same except for the scales on the axes.  The regression at right
    predicts that our algorithm would produce a diagram for the link
    for the largest such exterior with about 9,000 crossings}
  \label{fig: cong plot}
\end{figure}

\begin{figure}
  \centering
  \begin{tikzoverlay}[width=0.5\textwidth]{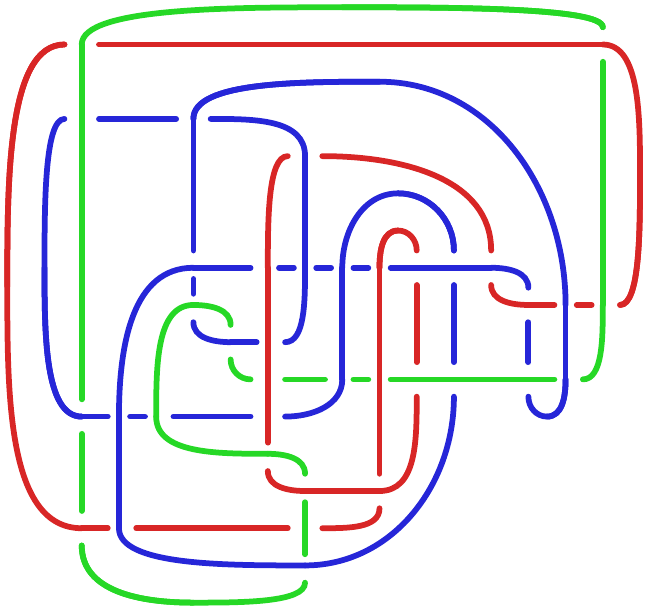}
    \begin{scope}[font=\small]
      \node[left] at (3.3,81.2) {$(0, 1)$};
      \node[below right] at (64.9,15.3) {$(-5, 1)$};
      \node[below left] at (13.4, 6.6) {$(0, 1)$};
    \end{scope}
  \end{tikzoverlay}
  \caption{A Dehn surgery description of the Seifert--Weber
    dodecahedral space.}
  \label{fig: SW}
\end{figure}

\subsection{Dehn surgery descriptions}
\label{sec: dehn desc}

Every closed orientable 3-manifold is a Dehn filling on some link
exterior in $S^3$ \cite[Chapter~9]{Rolfsen1990}, and such \emph{Dehn
  surgery descriptions} play a key role in both theory and practice.
However, finding a Dehn surgery description from e.g.~a triangulation
can be extremely challenging.  Thurston observed experimentally that,
starting with a closed hyperbolic \3-manifold, one frequently arrives
at a link exterior by repeatedly drilling out short closed geodesics,
see page 516 of \cite{Adams1999}.  Combining this with our algorithm
for \FindDiagram\ gives an effective tool for finding Dehn surgery
descriptions given a triangulation.  We applied this to the
Seifert--Weber dodecahedral space, which is an old example
\cite{WeberSeifert1933} still of much current interest
\cite{Bering2013, BurtonRubinsteinTillmann2012, LinLipnowski2020}.
The resulting description in Figure~\ref{fig: SW} seems to be the
first such published; a different description appeared subsequently in
\cite{BakerBlog}.  One could likely use a similar technique to find
Dehn surgery descriptions of a nonhyperbolic 3-manifold, for example
by first removing a complicated knot whose complement is hyperbolic
\cite{Myers1982} and then proceeding as before, or by using some other
method to select promising curves to drill out.

\subsection{Knots with the same 0-surgery}
\label{sec: zero}

The 0-surgery $\zK$ on a knot $K$ is the unique Dehn filling $N$ of
$\eK$ where $H_1(N; \Q) \neq 0$.  Pairs of knots $K$ and $K'$ with
$\zK$ homeomorphic to $Z(K')$ are of much interest in low-dimensional
topology.  Most strikingly, if such a pair $K$ and $K'$ exist with $K$
slice (i.e.~bounds a smooth $D^2$ in $D^4$) and the Rasmussen
$s$-invariant of $K'$ is nonzero, then the smooth 4-dimensional
Poincar\'e conjecture is false.  That is, there would exist a
4-manifold that is homeomorphic but not diffeomorphic to $S^4$. See
\cite{FreedmanGompfMorrison2010, ManolescuPiccirillo2021} for a
general discussion, and also \cite{Piccirillo2020} for an important
recent result using pairs with $Z(K) \cong Z(K')$.  There are many
techniques for constructing families of such pairs, which have been
unified by the red-blue-green link framework of
\cite{ManolescuPiccirillo2021}.  However, given a particular $K$, a
practical algorithm to search for $K'$ with the same 0-surgery has
been lacking.  When $Z(K)$ is hyperbolic, we attack this as follows.
First, find the short closed geodesics in $Z(K)$ using
\cite{HodgsonWeeks1994}.  Then drill out each geodesic in turn, and
test if the resulting manifold $\Mdot'$ has a Dehn filling which is
$S^3$; if it does, use our algorithm for \FindDiagram\ to $\Mdot'$ to
get a diagram for $K'$.

Figure~\ref{fig: zero} shows the result of applying our algorithm to
100 pairs $(K, \gamma)$ where $K$ is a knot with at most 18 crossings
and $\gamma$ is a short closed geodesic in $\zK$ whose exterior is
also that of a knot $K'$ in $S^3$. In all cases, we were able to
recover a diagram for $K'$, and these were more challenging on average
than the examples in Section~\ref{sec: impl}.

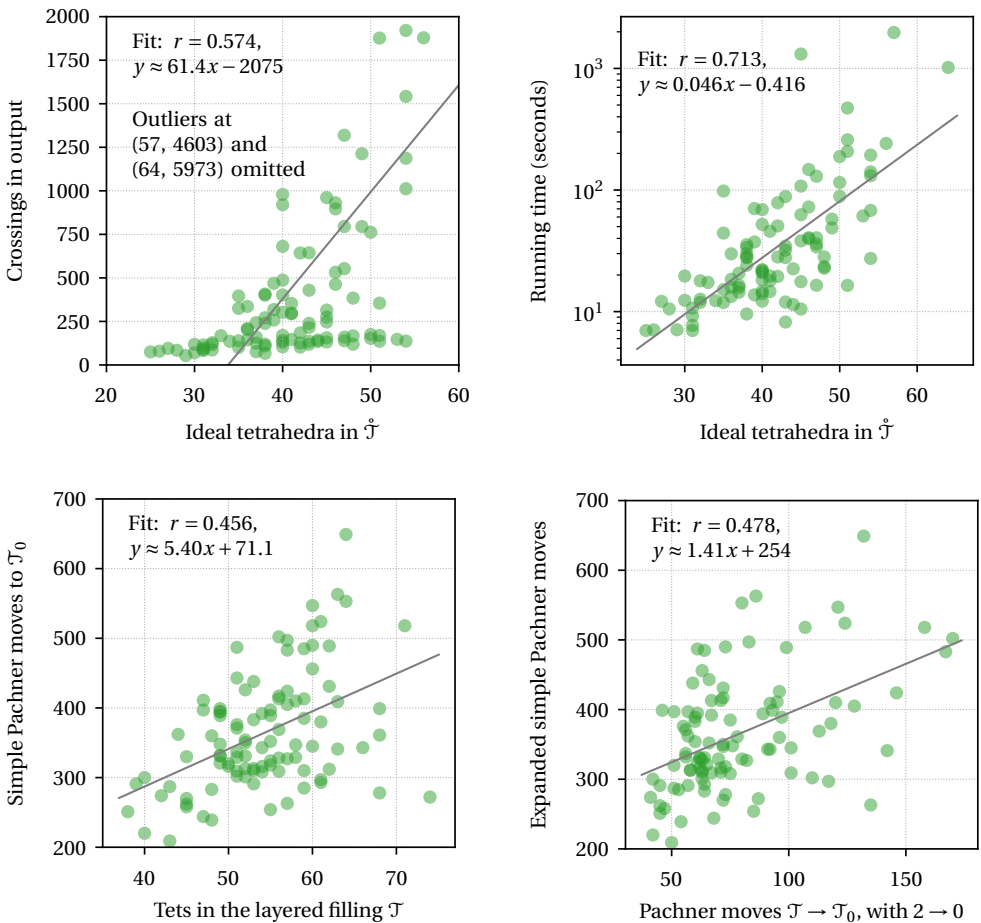
\begin{figure}[b]
  \centering

  \

  \vspace{-1cm}
  
  \begin{tikzpicture}
    \node[above right] at (0, 6.4) {
    \begin{tikzoverlay}[width=6cm]{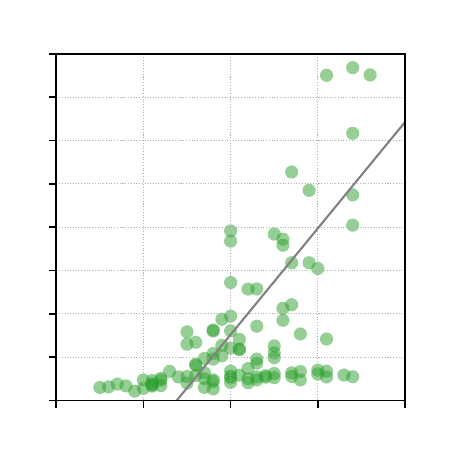}
      \draw (16.375000, 86.150000) node[below right, text width=2.5cm]
            {Fit: $r = 0.574$, \\ $y \approx 61.4 x - 2075$};
      \draw (51.250000, 1.029375) node[below] {Ideal tetrahedra in $\Tdot$};
      \draw (16.375000, 69.150000) node[below right, text width=2.5cm]
            {Outliers at \\ (57, 4603) and \\ (64, 5973) omitted};
      \draw (12.500000, 7.759259) node[below] {$20$};
      \draw (31.875000, 7.759259) node[below] {$30$};
      \draw (51.250000, 7.759259) node[below] {$40$};
      \draw (70.625000, 7.759259) node[below] {$50$};
      \draw (90.000000, 7.759259) node[below] {$60$};
      \draw (-6, 49.500000) node[rotate=90.0, anchor=base] {Crossings in output};
      \draw (10.8, 11.000000) node[left] {$0$};
      \draw (10.8, 20.625000) node[left] {$250$};
      \draw (10.8, 30.250000) node[left] {$500$};
      \draw (10.8, 39.875000) node[left] {$750$};
      \draw (10.8, 49.500000) node[left] {$1000$};
      \draw (10.8, 59.125000) node[left] {$1250$};
      \draw (10.8, 68.750000) node[left] {$1500$};
      \draw (10.8, 78.375000) node[left] {$1750$};
      \draw (10.8, 88.000000) node[left] {$2000$};
      \begin{scope}[shift={(-26.25000000, 11.00000000)},
                    xscale=1.93750000, yscale=0.03850000]
      \end{scope}
    \end{tikzoverlay}
    };
    \node[above right] at (6.9, 6.4) {
      \begin{tikzoverlay}[width=6cm]{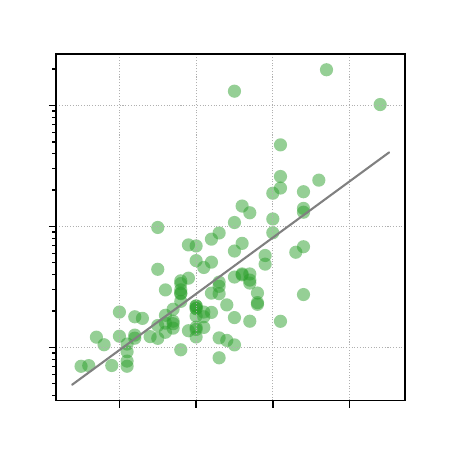}
        \draw (14.375000, 82.150000) node[below right, text width=3.2cm]
              {Fit: $r = 0.713$, \\ $y \approx 0.046 x - 0.416$};
        \draw (51.250000, 1.029375) node[below] {Ideal tetrahedra in $\Tdot$};
        \draw (26.538077, 7.759259) node[below] {$30$};
        \draw (43.580782, 7.759259) node[below] {$40$};
        \draw (60.623488, 7.759259) node[below] {$50$};
        \draw (77.666194, 7.759259) node[below] {$60$};
        \draw (-4, 49.500000) node[rotate=90.0, anchor=base] {Running time (seconds)};
        \draw (10.8, 22.796621) node[left] {${10^{1}}$};
        \draw (10.8, 49.677508) node[left] {${10^{2}}$};
        \draw (10.8, 76.558394) node[left] {${10^{3}}$};
    \end{tikzoverlay}
    };
    \node[above right] at (0, 0) {
      \begin{tikzoverlay}[width=6cm]{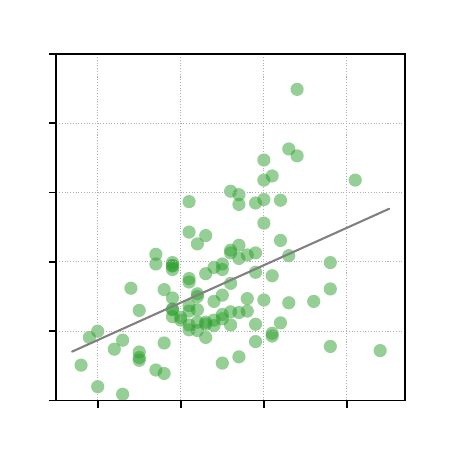}
        \draw (16.375000, 86.150000) node[below right, text width=3cm]
              {Fit: $r = 0.456$, \\  $y \approx 5.40 x + 71.1$ };
        \draw (51.250000, 1.029375) node[below] {Tets in the layered filling $\cT$};
        \draw (21.709310, 7.759259) node[below] {$40$};
        \draw (40.172241, 7.759259) node[below] {$50$};
        \draw (58.635172, 7.759259) node[below] {$60$};
        \draw (77.098104, 7.759259) node[below] {$70$};
        \draw (-5, 49.500000) node[rotate=90.0, anchor=base]
              {Simple Pachner moves to $\cT_0$};
        \draw (10.8, 11.000000) node[left] {$200$};
        \draw (10.8, 26.400000) node[left] {$300$};
        \draw (10.8, 41.800000) node[left] {$400$};
        \draw (10.8, 57.200000) node[left] {$500$};
        \draw (10.8, 72.600000) node[left] {$600$};
        \draw (10.8, 88.000000) node[left] {$700$};
        \begin{scope}[shift={(-52.14241471, -19.80000000)},
                      xscale=1.84629312, yscale=0.15400000]
        \end{scope}
      \end{tikzoverlay}
    };
    \node[above right] at (6.9, 0) {
      \begin{tikzoverlay}[width=6cm]{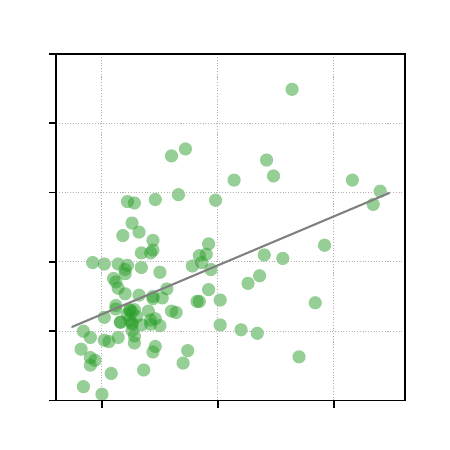}
        \draw (16.375000, 86.150000) node[below right, text width=3cm]
              {Fit: $r = 0.478$, \\ $y \approx 1.41 x + 254$};
        \draw (51.250000, 1.029375) node[below]
              {Pachner moves $\cT \to \cT_0$, with $2 \to 0$};
        \draw (22.653925, 7.759259) node[below] {$50$};
        \draw (48.416155, 7.759259) node[below] {$100$};
        \draw (74.178384, 7.759259) node[below] {$150$};
        \draw (-5, 49.500000) node[rotate=90.0, anchor=base]
              {Expanded simple Pachner moves};
        \draw (10.8, 11.000000) node[left] {$200$};
        \draw (10.8, 26.400000) node[left] {$300$};
        \draw (10.8, 41.800000) node[left] {$400$};
        \draw (10.8, 57.200000) node[left] {$500$};
        \draw (10.8, 72.600000) node[left] {$600$};
        \draw (10.8, 88.000000) node[left] {$700$};
        \begin{scope}[shift={(-3.10830441, -19.80000000)},
                     xscale=0.51524459, yscale=0.15400000]
        \end{scope}
      \end{tikzoverlay}
    };
  \end{tikzpicture}
  \caption{Data on the 100 knot exteriors from Section~\ref{sec: zero}.}
  \label{fig: zero}
\end{figure}


\section{Conclusion and open questions}
\label{sec: future}

To assess the practical effectiveness of our algorithm for
\FindDiagram, it is worth considering how large a link diagram can be
used as input for a subsequent computation.  Many key link invariants,
such as the Alexander polynomial and the Seifert genus, can be
computed directly (and at least as efficiently) from a triangulation
of the link exterior as from the diagram, so we focus on those that
require a diagram to compute.  The Jones polynomial is among the
simpler such ``truly diagrammatic'' invariants, and it is \#P-hard
\cite{JaegerVertiganWelsh90} to compute, though it is fixed-parameter
tractable in the tree-width of the diagram \cite{Makowsky2005}.  Even
the very best implementations for computing the Jones polynomial are
unable to handle most diagrams with more than 200 crossings
\cite{KnotJob}.  More refined invariants, such as knot Floer homology
\cite{OzsvathSzabo2019, HFKCalc} and Khovanov homology
\cite{Bar-Natan2007, Schutz2021, KnotJob}, are typically impractical
above 50-100 crossings, or even less depending on the precise variant.

As Figure~\ref{fig: zero} demonstrates, our implementation easily
finds link diagrams which are too big to allow computation of such
diagrammatic invariants.  While these diagrams presumably do not
minimize the number of crossings, we expect they are close enough that
the point still stands.  Alternatively, recall the diagram in
Figure~\ref{fig: cong 1} produced by our algorithm has 294 crossings,
and, by volume considerations, any diagram of this link must have at
least 66 crossings, and that alone would push the limits of many
diagrammatic calculations.

Given the practical effectiveness of our implementation of
\FindDiagram, we have incorporated it as a standard feature of SnapPy
\cite{SnapPy}, so that it can be widely used.  We conclude with
several open questions.

\begin{enumerate}
\item To what extent can the mean running time of $O(1.07^n)$ be
  reduced?  While Theorem~\ref{thm: fib} shows that the worst case
  running time must be at least exponential, it is not implausible
  that the mean running time is polynomial in the size of the
  \emph{output}.  The key issue is that the number of arcs in the base
  triangulation $\cT_0$ is currently exponential in the size of both
  the input \emph{and} the output, compare Figure~\ref{fig: num arcs}.

  In our current implementation, as the size of the input exceeds 40
  tetrahedra, the computation time becomes dominated by the final
  diagram simplification step.  This is despite the very favorable
  $O(n^{1.5})$ performance of our diagram simplification algorithm
  (Figure~\ref{fig: simp in action}). This again suggests we need to
  simplify the barycentric link more during the Pachner move steps.

\item To reduce the number of arcs when applying Pachner moves, one
  could consider additional local PL simplification moves, or try the
  current moves in larger balls in $\cT$ made up of several
  tetrahedra. We tried using the straighten simplification method
  inside the octahedron formed during the $4\to4$ move, but, in
  the examples we tested, this had the surprising effect of increasing
  the number of arcs at the final stage of the algorithm.  We also
  tried adding more complicated PL simplifications moves beyond
  straighten and push, but these had minimal effect.

\item It would be interesting to see if knot energy minimization, see
  \cite{Ohara91} and \cite{Simon94}, provides a practical way to
  simplify configurations of arcs within tetrahedra or to simplify the
  final link in $\R^3$ before projecting.

\item Our algorithm provides a new way to move in the space of
  diagrams for a given link: take a link diagram, produce a
  triangulation of the exterior, apply Pachner moves to modify the
  triangulation, then apply our algorithm to this new triangulation to
  obtain a new diagram of the link. It would be interesting to see if
  this approach effectively finds simple diagrams of links, or allows
  us to produce distinct simple diagrams that require many
  Reidemeister moves (or many additional crossings) to go between.

\item Is it possible to find diagrams for the remaining 10 congruence
  link exteriors discussed in Section~\ref{sec: cong}?

\end{enumerate}

\bibliography{exterior}
\end{document}